\documentclass[reqno,11pt]{amsart}
\usepackage{amsmath,amssymb,amscd,epsfig,amsthm}
\usepackage{color, graphics}
\usepackage{stackrel}
\usepackage{comment}
\usepackage[poly,all]{xy}
\usepackage{tikz}
\usetikzlibrary{snakes}
\usetikzlibrary{positioning}
\usetikzlibrary{matrix,calc}
\theoremstyle{theorem}\newtheorem{thm}{Theorem}[section]
\newtheorem{prop}[thm]{Proposition}
\newtheorem{lem}[thm]{Lemma}
\theoremstyle{definition}\newtheorem{df}[thm]{Definition}
\newtheorem{rem}[thm]{Remark}
\newtheorem{example}[thm]{Example}
\newtheorem{algo}[thm]{Algorithm}
\newtheorem*{algo-non}{Algorithm}
\renewenvironment{proof}{{\noindent \bfseries Proof.}}{\qed}
\newcommand{\Ld}{\Lambda}
\newcommand{\ld}{\lambda}
\newcommand{\bt}{\beta}
\newcommand{\g}{\gamma}
\newcommand{\G}{\Gamma}
\newcommand{\al}{\alpha}
\newcommand{\C}{\mathbb C}
\newcommand{\Z}{\mathbb Z}
\newcommand{\up}{\uparrow}
\newcommand{\down}{\downarrow}
\newcommand{\SW}{{\sf sw}}
\newcommand{\stan}{{\sf stan}}
\newcommand{\rstan}{{\sf stan}^*}
\newcommand{\st}{{\sf st}}
\newcommand{\rec}{{\rm rec}}
\newcommand{\rest}{{\rm rest}}
\newcommand{\D}{\mathcal{D}}
\newcommand{\B}{\mathcal{B}}
\newcommand{\rd}{{\rm \underline{{\textbf{read}}}}}
\newcommand{\w}{{\rm w}}
\newcommand{\wt}{{\rm wt}}
\newcommand{\ost}{{\rm st}}
\newcommand{\re}{{\sf re}}
\newcommand{\sh}{{\rm sh}}
\newcommand{\Rect}{{\rm Rect}}
\newcommand{\bfa}{{\mathbf{a}}}
\newcommand{\bfb}{{\mathbf{b}}}
\newcommand{\bfc}{{\mathbf{c}}}
\newcommand{\bfd}{{\mathbf{d}}}
\newcommand{\bfe}{{\mathbf{e}}}
\newcommand{\bff}{{\mathbf{f}}}
\newcommand{\bfw}{{\mathbf{w}}}
\newcommand{\bfi}{{\mathbf{i}}}
\newcommand{\bfj}{{\mathbf{j}}}
\newcommand{\bfk}{{\mathbf{k}}}
\newcommand{\bfm}{{\mathbf{m}}}
\newcommand{\bfn}{{\mathbf{n}}}
\newcommand{\bfs}{{\mathbf{s}}}
\newcommand{\bft}{{\mathbf{t}}}
\newcommand{\bfu}{{\mathbf{u}}}
\newcommand{\bfv}{{\mathbf{v}}}
\newcommand{\bfx}{{\mathbf{x}}}
\newcommand{\bfy}{{\mathbf{y}}}
\renewcommand{\baselinestretch}{1.2}

\begin{document}
\title[Shifted Tableau Switchings and shifted LR coefficients]
{Shifted Tableau Switchings and shifted Littlewood-Richardson coefficients}

\author{Seung-Il Choi}
\address{Department of Mathematics \\ Seoul National University\\Seoul 08826, Korea}
\email{ignatioschoi@snu.ac.kr}

\author{Sun-Young Nam}
\address{School of Computational Sciences \\ KIAS\\Seoul 02455, Korea}
\email{synam@kias.re.kr}

\author{Young-Tak Oh}
\address{Department of Mathematics \\ Sogang University\\Seoul 04107, Korea}
\email{ytoh@sogang.ac.kr}
\maketitle
\baselineskip=12pt

\begin{abstract}
We provide two shifted analogues of the tableau switching process due to Benkart, Sottile, and Stroomer;
the shifted tableau switching process and the modified shifted tableau switching process.
They are performed by applying a sequence of specially contrived elementary transformations called {\it switches}
and turn out to have some spectacular properties.
For instance, the maps induced from these algorithms are involutive and behave very nicely with respect to
shifted Young tableaux whose reading words satisfy the lattice property.
As an application, we give combinatorial interpretations of Schur $P$- and Schur $Q$-function identities.
We also demonstrate the relationship between the shifted tableau switching process and the shifted $J$-operation due to Worley.
\end{abstract}

\section{Introduction}
\renewcommand{\thefootnote}{}
\footnotetext{
\renewcommand{\baselinestretch}{1.2}
\hfill \break
\noindent
This research was supported by NRF Grant $\sharp$2012R1A1A2001635
and NRF Grant $\sharp$2015R1D1-A1A01056670.
\hfill \break 2010 MSC: 05E05
\hfill \break Keywords: shifted tableau switchings, shifted jeu de taquin, Schur $P$- and Schur $Q$-functions, shifted Littlewood-Richardson coefficients}

The (tableau) switching process, introduced by Benkart-Sottile-Stroomer \cite{BSS},
is an algorithm which takes a pair $(S,T)$ of semistandard Young tableaux
such that $T$ extends $S$, and moves them through each other,
yielding a pair $({}^ST, S_T)$ of semistandard Young tableaux such that 
the outer border of ${}^ST$ coincides with the inner border of $S_T$.
It has been one of the fundamental tools
in the study of Young tableaux and their applications to 
Schubert calculus, representation theory, geometry, and so on.
Particularly it unifies the previously defined switching algorithms
due to Haiman \cite{Hai}, James-Kerber \cite{JK}, Remmel \cite{Re}, Shimozono \cite{Shi}
and provides nice combinatorial interpretations of Schur function identities.
Naively speaking, this algorithm may be viewed as a mixture of Sch\"{u}tzenberger's jeu de taquin \cite{Sc1} and its inverse process
in the sense that the map induced from the algorithm is involutive and ${}^ST$ is the rectification of $T$ when $S \cup T$ is of normal shape.
Here the rectification of $T$ is
the unique Young tableau of normal shape
whose reading word is Knuth equivalent to that of $T$.

The purpose of the present is to provide two shifted analogues of the switching process;
{\it the shifted (tableau) switching process} and {\it the modified shifted (tableau) switching process}.
The former is an algorithm which takes and yields a pair of shifted Young tableaux,
whereas the latter takes and yields a pair of semistandard shifted Young tableaux.
Here a semistandard shifted Young tableau means a shifted Young tableau with no primed entry on the main diagonal.

We construct the shifted switching process by mimicking the strategy of \cite{BSS}.
We first establish the notion of shifted perforated $(\bfa,\bfb)$-pairs and then contrive the following seven elementary transformations called
{\it switches} to interchange two letters in a shifted perforated $(\bfa,\bfb)$-pair:
\begin{center}
\begin{tikzpicture}
\def \hhh{5mm}    
\def \vvv{5mm}    
\def \hhhhh{60mm}   
\node[] at (-\hhh*1.3,0) () {(S1)};
\draw[fill=black!30] (0,-\vvv*0.5) rectangle (\hhh*1,\vvv*0.5);
\draw[-] (\hhh*1,-\vvv*0.5) rectangle (\hhh*2,\vvv*0.5);
\draw[|->] (\hhh*3,\vvv*0) to (\hhh*4,\vvv*0);
\draw[-] (\hhh*5,-\vvv*0.5) rectangle (\hhh*6,\vvv*0.5);
\draw[fill=black!30] (\hhh*6,-\vvv*0.5) rectangle (\hhh*7,\vvv*0.5);
\node[] at (\hhh*0.5,\vvv*0) () {$\bfa$};
\node[] at (\hhh*1.5,\vvv*0) () {$\bfb$};
\node[] at (\hhh*5.5,\vvv*0) () {$\bfb$};
\node[] at (\hhh*6.5,\vvv*0) () {$\bfa$};
\node[] at (-\hhh*1.3+\hhhhh*1,0) () {(S2)};
\draw[fill=black!30] (0+\hhhhh*1,0) rectangle (\hhh*1+\hhhhh*1,\vvv*1);
\draw[-] (\hhh*0+\hhhhh*1,-\vvv*1) rectangle (\hhh*1+\hhhhh*1,0);
\draw[|->] (\hhh*2+\hhhhh*1,\vvv*0) to (\hhh*3+\hhhhh*1,\vvv*0);
\draw[-,black!0] (\hhh*4+\hhhhh*1,0) rectangle (\hhh*7+\hhhhh*1,\vvv*1);
\draw[-] (\hhh*4+\hhhhh*1,0) rectangle (\hhh*5+\hhhhh*1,\vvv*1);
\draw[fill=black!30] (\hhh*4+\hhhhh*1,-\vvv*1) rectangle (\hhh*5+\hhhhh*1,\vvv*0);
\node[] at (\hhh*0.5+\hhhhh*1,\vvv*0.5) () {$\bfa$};
\node[] at (\hhh*0.5+\hhhhh*1,-\vvv*0.5) () {$\bfb$};
\node[] at (\hhh*4.5+\hhhhh*1,\vvv*0.5) () {$\bfb$};
\node[] at (\hhh*4.5+\hhhhh*1,-\vvv*0.5) () {$\bfa$};
\end{tikzpicture}
\vskip 3mm
\begin{tikzpicture}
\def \hhh{5mm}    
\def \vvv{5mm}    
\def \hhhh{60mm}   
\node[] at (-\hhh*1.3,0) () {(S3)};
\draw[fill=black!30] (0,0) rectangle (\hhh*1,\vvv*1);
\draw[-] (\hhh*1,0) rectangle (\hhh*2,\vvv*1);
\draw[-] (\hhh*1,-\vvv*1) rectangle (\hhh*2,0);
\draw[|->] (\hhh*3,0) to (\hhh*4,0);
\draw[fill=black!30] (\hhh*1+\hhh*5,0) rectangle (\hhh*2+\hhh*5,-\vvv*1);
\draw[-] (\hhh*0+\hhh*5,0) rectangle (\hhh*1+\hhh*5,\vvv*1);
\draw[-] (\hhh*1+\hhh*5,0) rectangle (\hhh*2+\hhh*5,\vvv*1);
\node[] at (-\hhh*1.3+\hhhh*1,0) () {(S4)};
\draw[fill=black!30] (0+\hhhh,0) rectangle (\hhh*1+\hhhh,\vvv*1);
\draw[fill=black!30] (\hhh*1+\hhhh,0) rectangle (\hhh*2+\hhhh,\vvv*1);
\draw[-] (\hhh*1+\hhhh,-\vvv*1) rectangle (\hhh*2+\hhhh,0);
\draw[|->] (\hhh*3+\hhhh,0) to (\hhh*4+\hhhh*1,0);
\draw[-] (\hhh*0+\hhh*5+\hhhh*1,0) rectangle (\hhh*1+\hhh*5+\hhhh*1,\vvv*1);
\draw[fill=black!30] (\hhh*1+\hhh*5+\hhhh*1,0) rectangle (\hhh*2+\hhh*5+\hhhh*1,-\vvv*1);
\draw[fill=black!30] (\hhh*1+\hhh*5+\hhhh*1,0) rectangle (\hhh*2+\hhh*5+\hhhh*1,\vvv*1);
\node[] at (\hhh*0.5,\vvv*0.5) () {$\bfa$};
\node[] at (\hhh*1.5,\vvv*0.5) () {$b'$};
\node[] at (\hhh*1.5,-\vvv*0.5) () {$\bfb$};
\node[] at (\hhh*0.5+\hhh*5,\vvv*0.5) () {$\bfb$};
\node[] at (\hhh*1.5+\hhh*5,\vvv*0.5) () {$b$};
\node[] at (\hhh*1.5+\hhh*5,-\vvv*0.5) () {$\bfa$};
\node[] at (\hhh*0.5+\hhhh,\vvv*0.5) () {$\bfa$};
\node[] at (\hhh*1.5+\hhhh,\vvv*0.5) () {$a$};
\node[] at (\hhh*1.5+\hhhh,-\vvv*0.5) () {$\bfb$};
\node[] at (\hhh*0.5+\hhh*5+\hhhh*1,\vvv*0.5) () {$\bfb$};
\node[] at (\hhh*1.5+\hhh*5+\hhhh*1,\vvv*0.5) () {$a'$};
\node[] at (\hhh*1.5+\hhh*5+\hhhh*1,-\vvv*0.5) () {$\bfa$};
\end{tikzpicture}
\vskip 3mm
\begin{tikzpicture}
\def \hhh{5mm}   
\def \vvv{5mm}   
\def \hhhh{60mm}  
\def \mult{8}    
\node[] at (-\hhh*1.3,0) () {(S5)};
\draw[fill=black!30] (0,0) rectangle (\hhh*1,\vvv*1);
\draw[-] (\hhh*1,0) rectangle (\hhh*2,\vvv*1);
\draw[-] (\hhh*0,-\vvv*1) rectangle (\hhh*1,0);
\draw[|->] (\hhh*3,0) to (\hhh*4,0);
\draw[-] (\hhh*0+\hhh*5,0) rectangle (\hhh*1+\hhh*5,\vvv*1);
\draw[fill=black!30] (\hhh*1+\hhh*5,0) rectangle (\hhh*2+\hhh*5,\vvv*1);
\draw[-] (\hhh*0+\hhh*5,-\vvv*1) rectangle (\hhh*1+\hhh*5,0);
\node[] at (\hhh*0.5,\vvv*0.5) () {$\bfa$};
\node[] at (\hhh*1.5,\vvv*0.5) () {$b'$};
\node[] at (\hhh*0.5,-\vvv*0.5) () {$\bfb$};
\node[] at (\hhh*0.5+\hhh*5,\vvv*0.5) () {$b'$};
\node[] at (\hhh*1.5+\hhh*5,\vvv*0.5) () {$\bfa$};
\node[] at (\hhh*0.5+\hhh*5,-\vvv*0.5) () {$\bfb$};
\node[] at (-\hhh*1.3+\hhhh*1,0) () {(S6)};
\draw[fill=black!30] (0+\hhhh,0) rectangle (\hhh*1+\hhhh,\vvv*1);
\draw[-] (\hhh*1+\hhhh,0) rectangle (\hhh*2+\hhhh,\vvv*1);
\draw[-] (\hhh*0+\hhhh,-\vvv*1) rectangle (\hhh*1+\hhhh,0);
\draw[|->] (\hhh*3+\hhhh,0) to (\hhh*4+\hhhh,0);
\draw[-] (\hhh*1+\hhh*5+\hhhh*1,0) rectangle (\hhh*2+\hhh*5+\hhhh*1,\vvv*1);
\draw[-] (\hhh*0+\hhh*5+\hhhh*1,0) rectangle (\hhh*1+\hhh*5+\hhhh*1,\vvv*1);
\draw[fill=black!30] (\hhh*5+\hhhh*1,-\vvv*1) rectangle (\hhh*1+\hhh*5+\hhhh*1,0);
\node[] at (\hhh*0.5+\hhhh,\vvv*0.5) () {$\bfa$};
\node[] at (\hhh*1.5+\hhhh,\vvv*0.5) () {$b$};
\node[] at (\hhh*0.5+\hhhh,-\vvv*0.5) () {$\bfb$};
\node[] at (\hhh*0.5+\hhh*5+\hhhh*1,\vvv*0.5) () {$\bfb$};
\node[] at (\hhh*0.5+\hhh*5+\hhhh*1,-\vvv*0.5) () {$\bfa$};
\node[] at (\hhh*1.5+\hhh*5+\hhhh*1,\vvv*0.5) () {$b$};
\end{tikzpicture}
\vskip 3mm
\begin{tikzpicture}
\def \hhh{5mm}    
\def \vvv{5mm}    
\def \hhhhh{60mm}   
\node[] at (-\hhh*1.3,0) () {(S7)};
\draw[fill=black!30] (0,0) rectangle (\hhh*1,\vvv*1);
\draw[fill=black!30] (\hhh*1,0) rectangle (\hhh*2,\vvv*1);
\draw[-] (\hhh*1,-\vvv*1) rectangle (\hhh*2,0);
\draw[-] (\hhh*2,0) rectangle (\hhh*3,\vvv*1);
\draw[|->] (\hhh*4,0) to (\hhh*5,0);
\draw[-] (\hhh*6,0) rectangle (\hhh*7,\vvv*1);
\draw[fill=black!30] (\hhh*7,0) rectangle (\hhh*8,\vvv*1);
\draw[-] (\hhh*8,0) rectangle (\hhh*9,\vvv*1);
\draw[fill=black!30] (\hhh*7,0) rectangle (\hhh*8,-\vvv*1);
\draw[-,black!0] (\hhh*18,-\vvv*1) rectangle (\hhh*19,0);
\node[] at (\hhh*0.5,\vvv*0.5) () {$\bfa$};
\node[] at (\hhh*1.5,\vvv*0.5) () {$a$};
\node[] at (\hhh*2.5,\vvv*0.5) () {$b$};
\node[] at (\hhh*1.5,-\vvv*0.5) () {$\bfb$};
\node[] at (\hhh*6.5,\vvv*0.5) () {$\bfb$};
\node[] at (\hhh*7.5,\vvv*0.5) () {$a'$};
\node[] at (\hhh*8.5,\vvv*0.5) () {$b$};
\node[] at (\hhh*7.5,-\vvv*0.5) () {$\bfa$};
\end{tikzpicture}
\end{center}
\noindent
The above switches are so designed that they recover the shifted slides used in the shifted jeu de taquin
when $\bfa$-boxes are viewed as empty boxes.
Here the shifted jeu de taquin, introduced by Worley \cite{Wo} and Sagan \cite{Sa} independently as the shifted analogue of Sch\"{u}tzenberger's jeu de taquin,
is an algorithm which takes a shifted Young tableau and yields another shifted Young tableau with the same weight, but different shape.

Applying (S1) through (S7) in succession, we obtain the shifted switching process on the set of shifted perforated $(\bfa,\bfb)$-pairs $(A,B)$ such that $B$ extends $A$.
The most crucial step here 
is to show that the shifted perforateness is being maintained throughout
this process (Proposition \ref{prop-preserving-sp}).
Keeping this process, we finally obtain an algorithm
which takes a pair $(S,T)$ of shifted Young tableaux
such that $T$ extends $S$, and moves them through each other,
yielding a pair $({}^ST, S_T)$ of shifted Young tableaux such that $S_T$ extends ${}^ST$
(Theorem \ref{Thm-Sec3-SW}).
It should be pointed out that unlike the switching process,
our shifted switching process depends on the order of inner corners
to which switches are applied (Remark \ref{Remark-Diff}).

The shifted switching process so obtained has many similarities with
the switching process.
For instance, the map induced from the shifted switching process, 
called {\it the shifted tableau switching}, is involutive (Theorem \ref{thm-main-section4}).
Its proof, contrary to the case of the switching process,
is far from being obvious since our shifted switching process depends on the
order of inner corners.
And, when $S\cup T$ is of normal shape,
we can see that
${}^ST$ coincides with the rectification of $T$, the unique shifted Young tableau of normal shape whose reading word is shifted Knuth equivalent to that of $T$.
For other similarities, see Section \ref{sect-prop-ss}.

Despite of the above good properties, the shifted switching process is rather unsatisfactory in that
the shifted tableau switching cannot be reduced to a bijection on the set of pairs $(S, T)$ of semistandard shifted Young tableaux such that $T$ extends $S$
since the shifted jeu de taquin does not necessarily take a semistandard shifted Young tableau to another semistandard shifted Young tableau.
In more detail, if a switch is applied to the following $\bfa$-boxes
\vskip 2mm
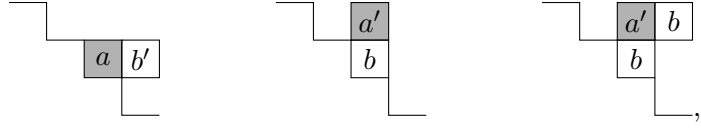
\begin{center}
\begin{tikzpicture}
\def \hhh{5mm}    
\def \vvv{5mm}    
\draw[-] (-\hhh*2,\vvv*2) -- (-\hhh*1,\vvv*2) -- (-\hhh*1,\vvv*1) -- (\hhh*0,\vvv*1);
\draw[fill=black!30] (0,0) rectangle (\hhh*1,\vvv*1);
\draw[-] (\hhh*1,0) rectangle (\hhh*2,\vvv*1);
\draw[-] (\hhh*1,\vvv*0) -- (\hhh*1,-\vvv*1) -- (\hhh*2,-\vvv*1);
\node[] at (\hhh*0.5,\vvv*0.5) () {$a$};
\node[] at (\hhh*1.5,\vvv*0.5) () {$b'$};
\end{tikzpicture}
\hskip 14mm
\begin{tikzpicture}
\def \hhh{5mm}    
\def \vvv{5mm}    
\draw[-] (-\hhh*2,\vvv*1) -- (-\hhh*1,\vvv*1) -- (-\hhh*1,\vvv*0) -- (\hhh*0,\vvv*0);
\draw[fill=black!30] (0,0) rectangle (\hhh*1,\vvv*1);
\draw[-] (\hhh*0,-\vvv*1) rectangle (\hhh*1,0);
\draw[-] (\hhh*1,-\vvv*1) -- (\hhh*1,-\vvv*2) -- (\hhh*2,-\vvv*2);
\node[] at (\hhh*0.5,\vvv*0.5) () {$a'$};
\node[] at (\hhh*0.5,-\vvv*0.5) () {$b$};
\end{tikzpicture}
\hskip 14mm
\begin{tikzpicture}
\def \hhh{5mm}   
\def \vvv{5mm}   
\draw[-] (-\hhh*1,\vvv*1) -- (-\hhh*0,\vvv*1) -- (-\hhh*0,\vvv*0) -- (\hhh*1,\vvv*0);
\draw[fill=black!30] (\hhh*1,0) rectangle (\hhh*2,\vvv*1);
\draw[-] (\hhh*2,0) rectangle (\hhh*3,\vvv*1);
\draw[-] (\hhh*1,-\vvv*1) rectangle (\hhh*2,0);
\draw[-] (\hhh*2,-\vvv*1) -- (\hhh*2,-\vvv*2) -- (\hhh*3,-\vvv*2);
\node[] at (\hhh*1.5,\vvv*0.5) () {$a'$};
\node[] at (\hhh*2.5,\vvv*0.5) () {$b$};
\node[] at (\hhh*1.5,-\vvv*0.5) () {$b$};
\end{tikzpicture},
\end{center}
then the resulting pairs have a primed letter on the main diagonal.
To avoid this phenomenon,
we introduce the following three {\it modified switches}:
\vskip 2mm
\begin{center}
\begin{tikzpicture}
\def \hhh{5mm}    
\def \vvv{5mm}    
\def \hhhh{50mm}   
\def \hhhhh{90mm}   
\node[] at (-\hhh*1.3,0) () {(S1$'$)};
\filldraw[fill=black!30] (0,-\vvv*0.5) rectangle (\hhh*1,\vvv*0.5);
\draw[-] (\hhh*1,-\vvv*0.5) rectangle (\hhh*2,\vvv*0.5);
\draw[|->] (\hhh*2.5,\vvv*0) to (\hhh*3.5,\vvv*0);
\draw[-] (\hhh*4,-\vvv*0.5) rectangle (\hhh*5,\vvv*0.5);
\draw[fill=black!30] (\hhh*5,-\vvv*0.5) rectangle (\hhh*6,\vvv*0.5);
\node[] at (\hhh*0.5,\vvv*0) () {$a$};
\node[] at (\hhh*1.5,\vvv*0) () {$b'$};
\node[] at (\hhh*4.5,\vvv*0) () {$b$};
\node[] at (\hhh*5.5,\vvv*0) () {$a'$};
\node[] at (-\hhh*1.5+\hhhh,0) () {(S2$'$)};
\draw[fill=black!30] (0+\hhhh,0) rectangle (\hhh*1+\hhhh,\vvv*1);
\draw[-] (\hhh*0+\hhhh,-\vvv*1) rectangle (\hhh*1+\hhhh,0);
\draw[|->] (\hhh*1.5+\hhhh,\vvv*0) to (\hhh*2.5+\hhhh,\vvv*0);
\draw[-] (\hhh*3+\hhhh,\vvv*0) rectangle (\hhh*4+\hhhh,\vvv*1);
\draw[fill=black!30] (\hhh*3+\hhhh,-\vvv*1) rectangle (\hhh*4+\hhhh,\vvv*0);
\node[] at (\hhh*0.5+\hhhh,\vvv*0.5) () {$a'$};
\node[] at (\hhh*0.5+\hhhh,-\vvv*0.5) () {$b$};
\node[] at (\hhh*3.5+\hhhh,\vvv*0.5) () {$b'$};
\node[] at (\hhh*3.5+\hhhh,-\vvv*0.5) () {$a$};
\node[] at (-\hhh*1.5+\hhhhh,0) () {(S6$'$)};
\draw[|->] (\hhh*2.5+\hhhhh,0) to (\hhh*3.5+\hhhhh,0);
\draw[fill=black!30] (\hhh*0+\hhhhh,0) rectangle (\hhh*1+\hhhhh,\vvv*1);
\draw[-] (\hhh*1+\hhhhh,0) rectangle (\hhh*2+\hhhhh,\vvv*1);
\draw[-] (\hhh*0+\hhhhh,-\vvv*1) rectangle (\hhh*1+\hhhhh,0);
\draw[-] (\hhh*4+\hhhhh,\vvv*0) rectangle (\hhh*5+\hhhhh,\vvv*1);
\draw[-] (\hhh*5+\hhhhh,\vvv*0) rectangle (\hhh*6+\hhhhh,\vvv*1);
\draw[fill=black!30] (\hhh*4+\hhhhh,-\vvv*1) rectangle (\hhh*5+\hhhhh,\vvv*0);
\node[] at (\hhh*0.5+\hhhhh,\vvv*0.5) () {$a'$};
\node[] at (\hhh*1.5+\hhhhh,\vvv*0.5) () {$b$};
\node[] at (\hhh*0.5+\hhhhh,-\vvv*0.5) () {$b$};
\node[] at (\hhh*4.5+\hhhhh,\vvv*0.5) () {$b'$};
\node[] at (\hhh*5.5+\hhhhh,\vvv*0.5) () {$b$};
\node[] at (\hhh*4.5+\hhhhh,-\vvv*0.5) () {$a$};
\end{tikzpicture}
\end{center}

We construct the modified shifted switching process 
by applying (S1$'$), (S2$'$) and (S6$'$) in the above three cases and (S1) through (S7) in the other cases exactly in the same way
as the shifted switching process has been defined.
We first verify that the shifted perforateness is being maintained throughout the modified shifted switching process when it is applied to
any shifted perforated $(\bfa,\bfb)$-pair $(A,B)$ such that
$B$ extends $A$ and
no primed letters are on the main diagonal of $A \cup B$
(Proposition \ref{prop-main-sect6}).
Once the modified shifted switching process is established, we show that the induced map,
called {\it the modified shifted tableau switching}, is involutive
(Theorem \ref{thm-main-section6}).
It would be very nice to develop a modified version of shifted jeu de taquin which plays the same role as that of the shifted jeu de taquin in the shifted switching process.

Main applications of our shifted tableau switchings concern
Schur $P$- and Schur $Q$-function identities.
Let $P_\lambda(x)$ and $Q_\lambda(x)$ respectively denote the Schur $P$- and Schur $Q$-function associated with a strict partition $\ld$,
and $P_{\nu/\lambda}(x)$ and $Q_{\nu/\lambda}(x)$ the Schur $P$- and Schur $Q$-function associated with a skew shape ${\nu/\lambda}$.
Letting
$P_\lambda(x) P_\mu(x) = \sum_{\nu} f^{\nu}_{\lambda \mu} P_\nu(x),$
it is well known that $f^{\nu}_{\lambda \mu}$ counts
shifted Young tableaux $T$
of shape $\nu/\lambda$ and weight $\mu$ such that
\begin{itemize}
  \item[(a)] $w = \w (T)$ satisfies the lattice property;
  \item[(b)] the rightmost $i$ of $|w|$ is unmarked in $w$ $\left(1\leq i \leq \ell(\mu)\right)$.
\end{itemize}
These tableaux are called {\it Littlewood-Richardson-Stembridge (LRS) tableaux}.
Using the shifted switching process we give combinatorial interpretations of $f^{\nu}_{\lambda \mu}=f^{\nu}_{\mu \lambda}$
and
$Q_{\nu/\lambda}(x) = \sum_{\mu} f^{\nu}_{\lambda \mu} Q_\mu(x)$.
On the other hand, letting
$Q_\lambda(x)Q_\mu(x) = \sum_{\nu} g^{\nu}_{\lambda \mu} Q_\nu(x),$
it can be easily seen that $g^{\nu}_{\lambda \mu}$ counts
semistandard shifted Young tableaux $T$ of shape $\nu/\lambda$ and weight $\mu$ satisfying the condition (a).
These tableaux are called {\it modified Littlewood-Richardson-Stembridge (LRS) tableaux}.
Compared with LRS tableaux,
the study of modified LRS tableaux is not yet well established.
In the best knowledge of the authors, even the combinatorial interpretation of $g^{\nu}_{\lambda \mu}=g^{\nu}_{\mu\lambda}$ has not been known.
Using the modified shifted tableau switching we here give its combinatorial interpretation
and also that of $P_{\nu/\lambda}(x) = \sum_{\mu} g^{\nu}_{\lambda \mu} P_\mu(x)$.

Finally we remark that there have already been two operations similar to the shifted switching process.
The one is due to Buch-Kresch-Tamvakis \cite{BKT}, who exploited sliding unmarked and marked holes in any LRS tableau
to construct a bijection between NW and SE holed LRS tableaux.
These slides can be viewed as switches by identifying ${\rm o}$-holes with $a$-boxes and ${\rm o}'$-holes with $a'$-boxes.
Although some of the switches obtained in this way are different from our switches, it is quite interesting to note that
the sliding paths \cite{BKT} are exactly same to our shifted switching paths.
This fact plays a crucial role in the proof of Lemma \ref{lem-path-intersect}.
The other is about the shifted $J$-operation due to Worley \cite{Wo}.
He utilized this operation intensively to show the symmetric identity 
$f^{\nu}_{\ld \mu} = f^{\nu}_{\mu \ld}$
and properties about the coefficients of expansions involving $Q_{\ld/\mu}$.
We demonstrate the relationship between the shifted switching process 
and the shifted $J$-operation
in case where $S \cup T$ is of normal shape (Theorem \ref{thm-main-Section5}).

The present paper is organized as follows:
In Section \ref{Preliminaries}, we collect the definitions and notation
necessary to develop our arguments.
In Section \ref{sect-ss-proce}, we introduce the shifted tableau switching and its properties.
To do this we first define the shifted switching process on shifted perforated pairs (Algorithm \ref{algorithm1}).
Using this we define the shifted switching process on pairs of shifted Young tableaux
(Algorithm \ref{algorithm2}).
Section \ref{sect-prop-ss} is devoted to applications of the shifted tableau switching.
In Section \ref{sect-Worley-oper-J},
we provide a shifted analogue of the generalized evacuation in \cite{BKT}
and
redescribe the shifted switching process in the spirit of
the shifted $J$-operation.
The final section is devoted to 
the modified shifted tableau switching and its applications.

\section{Definitions, terminologies, and notation}
\label{Preliminaries}
In this section,
we collect definitions, terminologies and notation
required to develop our arguments.
More details can be found in
\cite{BKT, Ful, Sa, Ste1, Wo}.

For a strict partition $\lambda = (\lambda_1 > \lambda_2 > \cdots > \lambda_\ell > 0 )$,
let the {\it length} $\ell(\ld)$ of $\lambda$ be the number of positive parts of $\ld$
and the {\it size} $|\ld|$ of $\ld$ the sum of positive parts of $\ld$, that is,
$\ell(\ld)=\ell$ and $|\ld| = \sum \ld_i$.
Let $\Ld^+$ denote the set of all strict partitions.

For $\ld \in \Ld^+$, define the {\it shifted shape} $S(\lambda)$ of $\lambda$ by
$$
\left\{ (i,j) \in \mathbb{Z}^2 :
i \leq j \leq \lambda_i+i-1, \ 1 \leq i \leq \ell(\ld) \right\}
$$
and the {\em main diagonal} of $S(\lambda)$ by
$\{(i,i) \in S(\lambda)  : 1 \leq i \leq \ell(\lambda) \}$.
For $\lambda , \mu \in \Lambda^+$ with $\lambda_i \geq \mu_i$ for all $i\geq 1$,
we write $\lambda \supseteq \mu$ and let $|\ld/\mu|$ denote $|\ld|-|\mu|$.
In this case, the {\it skew shifted shape} $S(\lambda/\mu)$ is defined to be
$S(\lambda)\setminus S(\mu)$.
Particularly when $\mu = \emptyset$, $S(\lambda/\mu)$ is understood as $S(\lambda)$.
Throughout this paper, the elements of $S(\ld/\mu)$ are visualized as boxes in a plane with
matrix-style coordinates
and $S(\ld/\mu)$ is frequently denoted by $\ld/\mu$.
A {\it component} of $\lambda/\mu$ is defined to be a maximal connected subset of $\lambda/\mu$.

A skew shifted shape is called a {\it border strip} if
it contains no subset of the form $\{(i,j),(i+1,j+1)\}$
and a {\it double border strip} if
it contains no subset of the form
$\{(i,j),(i+1,j+1),(i+2,j+2)\}$.
Note that border and double border strips are allowed to be disconnected.
For instance,
\vskip 2mm
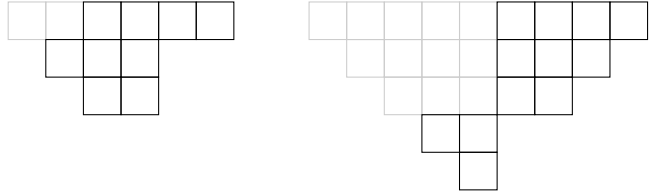
\begin{center}
\begin{tikzpicture}
\def \hhh{5mm}    
\def \vvv{5mm}    
\def \hhhhh{50mm}  
\draw[-,black!20] (\hhh*1+\hhhhh*0,0) rectangle (\hhh*2+\hhhhh*0,\vvv*1);
\draw[-,black!20] (\hhh*2+\hhhhh*0,0) rectangle (\hhh*3+\hhhhh*0,\vvv*1);
\draw[-] (\hhh*3+\hhhhh*0,0) rectangle (\hhh*4+\hhhhh*0,\vvv*1);
\draw[-] (\hhh*4+\hhhhh*0,0) rectangle (\hhh*5+\hhhhh*0,\vvv*1);
\draw[-] (\hhh*5+\hhhhh*0,0) rectangle (\hhh*6+\hhhhh*0,\vvv*1);
\draw[-] (\hhh*6+\hhhhh*0,0) rectangle (\hhh*7+\hhhhh*0,\vvv*1);
\draw[-] (\hhh*2+\hhhhh*0,-\vvv*1) rectangle (\hhh*3+\hhhhh*0,\vvv*0);
\draw[-] (\hhh*3+\hhhhh*0,-\vvv*1) rectangle (\hhh*4+\hhhhh*0,\vvv*0);
\draw[-] (\hhh*4+\hhhhh*0,-\vvv*1) rectangle (\hhh*5+\hhhhh*0,\vvv*0);
%
\draw[-] (\hhh*3+\hhhhh*0,-\vvv*2) rectangle (\hhh*4+\hhhhh*0,-\vvv*1);
\draw[-] (\hhh*4+\hhhhh*0,-\vvv*2) rectangle (\hhh*5+\hhhhh*0,-\vvv*1);
\draw[-,black!20] (-\hhh*1+\hhhhh*1,0) rectangle (\hhh*0+\hhhhh*1,\vvv*1);
\draw[-,black!20] (\hhh*0+\hhhhh*1,0) rectangle (\hhh*1+\hhhhh*1,\vvv*1);
\draw[-,black!20] (\hhh*1+\hhhhh*1,0) rectangle (\hhh*2+\hhhhh*1,\vvv*1);
\draw[-,black!20] (\hhh*2+\hhhhh*1,0) rectangle (\hhh*3+\hhhhh*1,\vvv*1);
\draw[-,black!20] (\hhh*3+\hhhhh*1,0) rectangle (\hhh*4+\hhhhh*1,\vvv*1);
\draw[-] (\hhh*4+\hhhhh*1,0) rectangle (\hhh*5+\hhhhh*1,\vvv*1);
\draw[-] (\hhh*5+\hhhhh*1,0) rectangle (\hhh*6+\hhhhh*1,\vvv*1);
\draw[-] (\hhh*6+\hhhhh*1,0) rectangle (\hhh*7+\hhhhh*1,\vvv*1);
\draw[-] (\hhh*7+\hhhhh*1,0) rectangle (\hhh*8+\hhhhh*1,\vvv*1);
\draw[-,black!20] (\hhh*0+\hhhhh*1,-\vvv*1) rectangle (\hhh*1+\hhhhh*1,\vvv*0);
\draw[-,black!20] (\hhh*1+\hhhhh*1,-\vvv*1) rectangle (\hhh*2+\hhhhh*1,\vvv*0);
\draw[-,black!20] (\hhh*2+\hhhhh*1,-\vvv*1) rectangle (\hhh*3+\hhhhh*1,\vvv*0);
\draw[-,black!20] (\hhh*3+\hhhhh*1,-\vvv*1) rectangle (\hhh*4+\hhhhh*1,\vvv*0);
\draw[-] (\hhh*4+\hhhhh*1,-\vvv*1) rectangle (\hhh*5+\hhhhh*1,\vvv*0);
\draw[-] (\hhh*5+\hhhhh*1,-\vvv*1) rectangle (\hhh*6+\hhhhh*1,\vvv*0);
\draw[-] (\hhh*6+\hhhhh*1,-\vvv*1) rectangle (\hhh*7+\hhhhh*1,\vvv*0);
\draw[-,black!20] (\hhh*1+\hhhhh*1,-\vvv*2) rectangle (\hhh*2+\hhhhh*1,-\vvv*1);
\draw[-,black!20] (\hhh*2+\hhhhh*1,-\vvv*2) rectangle (\hhh*3+\hhhhh*1,-\vvv*1);
\draw[-,black!20] (\hhh*3+\hhhhh*1,-\vvv*2) rectangle (\hhh*4+\hhhhh*1,-\vvv*1);
\draw[-] (\hhh*4+\hhhhh*1,-\vvv*2) rectangle (\hhh*5+\hhhhh*1,-\vvv*1);
\draw[-] (\hhh*5+\hhhhh*1,-\vvv*2) rectangle (\hhh*6+\hhhhh*1,-\vvv*1);
\draw[-] (\hhh*2+\hhhhh*1,-\vvv*3) rectangle (\hhh*3+\hhhhh*1,-\vvv*2);
\draw[-] (\hhh*3+\hhhhh*1,-\vvv*3) rectangle (\hhh*4+\hhhhh*1,-\vvv*2);
\draw[-] (\hhh*3+\hhhhh*1,-\vvv*4) rectangle (\hhh*4+\hhhhh*1,-\vvv*3);
\end{tikzpicture}
\end{center}
are double border strips, but
\vskip 2mm
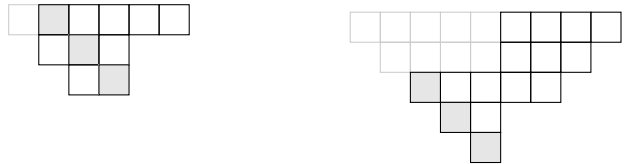
\begin{center}
\begin{tikzpicture}[baseline=0mm]
\def \hhh{4mm}    
\def \vvv{4mm}    
\draw[-,black!20] (\hhh*1,0) rectangle (\hhh*2,\vvv*1);
\draw[-,fill=black!10] (\hhh*2,0) rectangle (\hhh*3,\vvv*1);
\draw[-] (\hhh*3,0) rectangle (\hhh*4,\vvv*1);
\draw[-] (\hhh*4,0) rectangle (\hhh*5,\vvv*1);
\draw[-] (\hhh*5,0) rectangle (\hhh*6,\vvv*1);
\draw[-] (\hhh*6,0) rectangle (\hhh*7,\vvv*1);
\draw[-] (\hhh*2,-\vvv*1) rectangle (\hhh*3,\vvv*0);
\draw[-,fill=black!10] (\hhh*3,-\vvv*1) rectangle (\hhh*4,\vvv*0);
\draw[-] (\hhh*4,-\vvv*1) rectangle (\hhh*5,\vvv*0);
\draw[-] (\hhh*3,-\vvv*2) rectangle (\hhh*4,-\vvv*1);
\draw[-,fill=black!10] (\hhh*4,-\vvv*2) rectangle (\hhh*5,-\vvv*1);
\end{tikzpicture}
\hskip 20mm
\begin{tikzpicture}[baseline=5mm]
\def \hhh{4mm}    
\def \vvv{4mm}    
\draw[-,black!20] (\hhh*0,\vvv*1) rectangle (\hhh*1,\vvv*2);
\draw[-,black!20] (\hhh*1,\vvv*1) rectangle (\hhh*2,\vvv*2);
\draw[-,black!20] (\hhh*2,\vvv*1) rectangle (\hhh*3,\vvv*2);
\draw[-,black!20] (\hhh*3,\vvv*1) rectangle (\hhh*4,\vvv*2);
\draw[-,black!20] (\hhh*4,\vvv*1) rectangle (\hhh*5,\vvv*2);
\draw[-] (\hhh*5,\vvv*1) rectangle (\hhh*6,\vvv*2);
\draw[-] (\hhh*6,\vvv*1) rectangle (\hhh*7,\vvv*2);
\draw[-] (\hhh*7,\vvv*1) rectangle (\hhh*8,\vvv*2);
\draw[-] (\hhh*8,\vvv*1) rectangle (\hhh*9,\vvv*2);
\draw[-,black!20] (\hhh*1,0) rectangle (\hhh*2,\vvv*1);
\draw[-,black!20] (\hhh*2,0) rectangle (\hhh*3,\vvv*1);
\draw[-,black!20] (\hhh*3,0) rectangle (\hhh*4,\vvv*1);
\draw[-,black!20] (\hhh*4,0) rectangle (\hhh*5,\vvv*1);
\draw[-] (\hhh*5,0) rectangle (\hhh*6,\vvv*1);
\draw[-] (\hhh*6,0) rectangle (\hhh*7,\vvv*1);
\draw[-] (\hhh*7,0) rectangle (\hhh*8,\vvv*1);
\draw[fill=black!10] (\hhh*2,-\vvv*1) rectangle (\hhh*3,\vvv*0);
\draw[-] (\hhh*3,-\vvv*1) rectangle (\hhh*4,\vvv*0);
\draw[-] (\hhh*4,-\vvv*1) rectangle (\hhh*5,\vvv*0);
\draw[-] (\hhh*5,-\vvv*1) rectangle (\hhh*6,\vvv*0);
\draw[-] (\hhh*6,-\vvv*1) rectangle (\hhh*7,\vvv*0);
\draw[fill=black!10] (\hhh*3,-\vvv*2) rectangle (\hhh*4,-\vvv*1);
\draw[-] (\hhh*4,-\vvv*2) rectangle (\hhh*5,-\vvv*1);
\draw[fill=black!10] (\hhh*4,-\vvv*3) rectangle (\hhh*5,-\vvv*2);
\end{tikzpicture}
\end{center}
are not.

For two boxes $(i,j)$ and $(k,l)$ in a skew shifted shape,
$(i,j)$ is said to be $north$ of $(k,l)$ if
$i \le k$,
and $northwest$ of $(k,l)$
if $i \le k$ and $j \le l$.
When $i=k-1$ and $j=l$, $(i,j)$ is said to be the {\it neighbor to the north} of $(k,l)$.
The other directions and neighbors can be defined in the same manner.

Let $X$ be the set of the ordered alphabets $1' < 1 < 2' < 2 < \cdots$.
When we wish to refer to a letter without specifying whether it is primed or not,
we will write $\bfk$ for letters $k$ or $k'$.
Let $|\bfk|$ denote the unprimed letter $k$.

\begin{df}\label{GSYT}{\rm
Let $\lambda ,\mu \in \Lambda^+$ with $\mu \subseteq \lambda$.

(a) A {\it shifted Young tableau} $T$ of shape $\lambda/\mu$
is a filling of $\lambda/\mu$ with letters in $X$ such that
\begin{itemize}
  \item[$\bullet$] the entries are weakly increasing along each row and down each column,
  \item[$\bullet$] each column contains at most one $k$ for each $k \geq 1$, and
  \item[$\bullet$] each row contains at most one $k'$ for each $k \geq 1$.
\end{itemize}

(b) A {\it semistandard shifted Young tableau} $T$ of shape $\lambda/\mu$ is
a shifted Young tableau satisfying the following extra condition:
\begin{itemize}
\item[$\bullet$] there is no primed entry on the main diagonal.
\end{itemize}
}\end{df}

Given a filling $T$ of skew shape, we write ${\rm sh}(T)$ for the shape of $T$
and $|T|$ for the size of the shape of $T$.
Let $T_{i,j}$ be the entry of $(i,j)$ in $T$.
We call $(i,j)$ an {\it $\bfa$-box} in case where $T_{i,j}=\bfa$.

For a shifted Young tableau $T$,
let the {\it weight} $\wt(T)$ of $T$ be $(\wt_1(T),$ $\wt_2(T),\ldots)$,
where $\wt_i(T)$ is the number of occurrences of $i$ and $i'$ in $T$.
In particular, $T$ is {\it standard} if its weight is $(1,1,\ldots , 1)$ and it has no primed letter.
The weight generating function $x^T$ of $T$ is defined to be $\prod_{i\ge 1}x_i^{\wt_i(T)}$.
The {\it maximum} (resp., {\it minimum}) of $T$, denoted by  $\max(T)$ (resp., $\min(T)$),
is defined to be the largest integer (resp., smallest integer) $i$ such that $\wt_i(T)$ is positive.
\vskip 2mm
Throughout this paper, we will use the following abbreviations and notation.
\begin{itemize}
\item
SYT: shifted Young tableau

\item
SSYT: semistandard shifted Young tableau

\item
$\widetilde{\mathcal{Y}}(\lambda/\mu)$: the set of all SYTs of shape $\lambda/\mu$

\item
$\mathcal{Y}(\lambda/\mu)$: the set of all SSYTs of shape $\lambda/\mu$
\end{itemize}
\vskip 1mm

If $\nu \subseteq \mu \subseteq \lambda$, we say that $\ld /\mu$ $extends$ $\mu/\nu$.
When $\ld /\mu$ extends a single box $B$, $B$ is called an {\it inner corner} of $\ld/\mu$.
It should be distinguished from a {\it removable corner} of $\ld /\mu$, which is a box in $\ld /\mu$
such that neither neighbor to the south nor to the east is in $\ld /\mu$.

When $T_1$ is an SYT of shape $\ld/\mu$ and $T_2$ an SYT of shape $\gamma/\delta$
such that $\sh(T_1)\cap \sh(T_2)=\emptyset$, we can naturally consider $T_1 \cup T_2$, which is the object obtained by drawing $T_1$ and $T_2$ in a plane simultaneously.
For an SYT $T$, let $T^{(i)}$ be the tableau of border strip shape consisting of all $\bfi$-boxes in $T$.
If ${\rm max}(T)=m$, then $T$ is expressed as
$T^{(1)} \cup T^{(2)} \cup \cdots \cup T^{(m)}.$
For instance, if
\begin{center}
\begin{tikzpicture}
\def \hhh{5mm}    
\def \vvv{5mm}    
\node[] at (-\hhh*0.1,-\vvv*0.4) {$T=$};
\draw[-] (\hhh*2,0) rectangle (\hhh*3,\vvv*1);
\draw[-] (\hhh*3,0) rectangle (\hhh*4,\vvv*1);
\draw[-] (\hhh*4,0) rectangle (\hhh*5,\vvv*1);
\draw[-] (\hhh*5,0) rectangle (\hhh*6,\vvv*1);
\draw[-] (\hhh*1,-\vvv*1) rectangle (\hhh*2,\vvv*0);
\draw[-] (\hhh*2,-\vvv*1) rectangle (\hhh*3,\vvv*0);
\draw[-] (\hhh*3,-\vvv*1) rectangle (\hhh*4,\vvv*0);
\draw[-] (\hhh*2,-\vvv*2) rectangle (\hhh*3,-\vvv*1);
\draw[-] (\hhh*3,-\vvv*2) rectangle (\hhh*4,-\vvv*1);
\node[] at (\hhh*2.5,\vvv*0.5) {$1'$};
\node[] at (\hhh*3.5,\vvv*0.5) {$1$};
\node[] at (\hhh*4.5,\vvv*0.5) {$2'$};
\node[] at (\hhh*5.5,\vvv*0.5) {$2$};
\node[] at (\hhh*1.5,-\vvv*0.5) {$1'$};
\node[] at (\hhh*2.5,-\vvv*0.5) {$2'$};
\node[] at (\hhh*3.5,-\vvv*0.5) {$2$};
\node[] at (\hhh*2.5,-\vvv*1.5) {$2'$};
\node[] at (\hhh*3.5,-\vvv*1.5) {$3'$};
\end{tikzpicture}
\end{center}
then $T=T^{(1)} \cup T^{(2)} \cup T^{(3)}$, where
\vskip 2mm
\begin{center}
\begin{tikzpicture}
\def \hhh{5mm}    
\def \vvv{5mm}    
\def \hhhh{0mm}   
\def \hhhhh{45mm}  
\node[] at (\hhh*0.6+\hhhhh+\hhhh,-\vvv*0.4) {$T^{(1)}=$};
\draw[-,black!20] (\hhh*5+\hhhhh+\hhhh,0) rectangle (\hhh*6+\hhhhh+\hhhh,\vvv*1);
\draw[-,black!20] (\hhh*6+\hhhhh+\hhhh,0) rectangle (\hhh*7+\hhhhh+\hhhh,\vvv*1);
\draw[-,black!20] (\hhh*3+\hhhhh+\hhhh,-\vvv*1) rectangle (\hhh*4+\hhhhh+\hhhh,\vvv*0);
\draw[-,black!20] (\hhh*4+\hhhhh+\hhhh,-\vvv*1) rectangle (\hhh*5+\hhhhh+\hhhh,\vvv*0);
\draw[-,black!20] (\hhh*3+\hhhhh+\hhhh,-\vvv*2) rectangle (\hhh*4+\hhhhh+\hhhh,-\vvv*1);
\draw[-,black!20] (\hhh*4+\hhhhh+\hhhh,-\vvv*2) rectangle (\hhh*5+\hhhhh+\hhhh,-\vvv*1);
\draw[-] (\hhh*3+\hhhhh+\hhhh,0) rectangle (\hhh*4+\hhhhh+\hhhh,\vvv*1);
\draw[-] (\hhh*4+\hhhhh+\hhhh,0) rectangle (\hhh*5+\hhhhh+\hhhh,\vvv*1);
\draw[-] (\hhh*2+\hhhhh+\hhhh,-\vvv*1) rectangle (\hhh*3+\hhhhh+\hhhh,\vvv*0);
\node[] at (\hhh*3.5+\hhhhh+\hhhh,\vvv*0.5) {$1'$};
\node[] at (\hhh*4.5+\hhhhh+\hhhh,\vvv*0.5) {$1$};
\node[] at (\hhh*2.5+\hhhhh+\hhhh,-\vvv*0.5) {$1'$};
\node[] at (\hhh*0.6+\hhhhh*2+\hhhh,-\vvv*0.4) {$T^{(2)}=$};
\draw[-,black!20] (\hhh*3+\hhhhh*2+\hhhh,0) rectangle (\hhh*4+\hhhhh*2+\hhhh,\vvv*1);
\draw[-,black!20] (\hhh*4+\hhhhh*2+\hhhh,0) rectangle (\hhh*5+\hhhhh*2+\hhhh,\vvv*1);
\draw[-,black!20] (\hhh*2+\hhhhh*2+\hhhh,-\vvv*1) rectangle (\hhh*3+\hhhhh*2+\hhhh,-\vvv*0);
\draw[-,black!20] (\hhh*4+\hhhhh*2+\hhhh,-\vvv*2) rectangle (\hhh*5+\hhhhh*2+\hhhh,-\vvv*1);

\draw[-] (\hhh*5+\hhhhh*2+\hhhh,0) rectangle (\hhh*6+\hhhhh*2+\hhhh,\vvv*1);
\draw[-] (\hhh*6+\hhhhh*2+\hhhh,0) rectangle (\hhh*7+\hhhhh*2+\hhhh,\vvv*1);
\draw[-] (\hhh*3+\hhhhh*2+\hhhh,-\vvv*1) rectangle (\hhh*4+\hhhhh*2+\hhhh,\vvv*0);
\draw[-] (\hhh*4+\hhhhh*2+\hhhh,-\vvv*1) rectangle (\hhh*5+\hhhhh*2+\hhhh,\vvv*0);
\draw[-] (\hhh*3+\hhhhh*2+\hhhh,-\vvv*2) rectangle (\hhh*4+\hhhhh*2+\hhhh,-\vvv*1);
\node[] at (\hhh*5.5+\hhhhh*2+\hhhh,\vvv*0.5) {$2'$};
\node[] at (\hhh*6.5+\hhhhh*2+\hhhh,\vvv*0.5) {$2$};
\node[] at (\hhh*3.5+\hhhhh*2+\hhhh,-\vvv*0.5) {$2'$};
\node[] at (\hhh*4.5+\hhhhh*2+\hhhh,-\vvv*0.5) {$2$};
\node[] at (\hhh*3.5+\hhhhh*2+\hhhh,-\vvv*1.5) {$2'$};
\node[] at (\hhh*0.6+\hhhhh*3+\hhhh,-\vvv*0.4) {$T^{(3)}=$};
\draw[-,black!20] (\hhh*3+\hhhhh*3+\hhhh,0) rectangle (\hhh*4+\hhhhh*3+\hhhh,\vvv*1);
\draw[-,black!20] (\hhh*4+\hhhhh*3+\hhhh,0) rectangle (\hhh*5+\hhhhh*3+\hhhh,\vvv*1);
\draw[-,black!20] (\hhh*2+\hhhhh*3+\hhhh,-\vvv*1) rectangle (\hhh*3+\hhhhh*3+\hhhh,-\vvv*0);
\draw[-,black!20] (\hhh*4+\hhhhh*3+\hhhh,-\vvv*2) rectangle (\hhh*5+\hhhhh*3+\hhhh,-\vvv*1);

\draw[-,black!20] (\hhh*5+\hhhhh*3+\hhhh,0) rectangle (\hhh*6+\hhhhh*3+\hhhh,\vvv*1);
\draw[-,black!20] (\hhh*6+\hhhhh*3+\hhhh,0) rectangle (\hhh*7+\hhhhh*3+\hhhh,\vvv*1);
\draw[-,black!20] (\hhh*3+\hhhhh*3+\hhhh,-\vvv*1) rectangle (\hhh*4+\hhhhh*3+\hhhh,\vvv*0);
\draw[-,black!20] (\hhh*4+\hhhhh*3+\hhhh,-\vvv*1) rectangle (\hhh*5+\hhhhh*3+\hhhh,\vvv*0);
\draw[-,black!20] (\hhh*3+\hhhhh*3+\hhhh,-\vvv*2) rectangle (\hhh*4+\hhhhh*3+\hhhh,-\vvv*1);
\draw[-] (\hhh*4+\hhhhh*3+\hhhh,-\vvv*2) rectangle (\hhh*5+\hhhhh*3+\hhhh,-\vvv*1);
\node[] at (\hhh*4.5+\hhhhh*3+\hhhh,-\vvv*1.5) {$3'$};
\end{tikzpicture}.
\end{center}
In the obvious way we can extend the definition of $T^{(i)}$ to an arbitrary filling $T$ of shape $\ld/\mu$.

\vskip 3mm
\begin{df}
Let $T$ be an SYT.
\begin{itemize}
\item[(a)]
 We define the {\it standardization} of $T$, denoted by $\stan(T)$,
to be the standard SYT obtained from $T$
by replacing all $1'$'s with $1,2,\ldots,k_1$ from top to bottom,
and next all $1$'s with $k_1+1,k_1+2,\ldots,k_1+k_2$ from left to right,
and so on.
Here $k_1$ and $k_2$ denote the number of occurrences of $1'$ and $1$, respectively.

\item[(b)]
We define the {\it reverse standardization} of $T$, denoted by $\rstan(T)$,
to be the filling of $\sh(T)$ obtained from $\stan(T)$ by replacing $i$ with $|T|-i+1$.
\end{itemize}
\end{df}

For example, if
\begin{center}
\begin{tikzpicture}
\def \hhh{5mm}    
\def \vvv{5mm}    
\node[] at (\hhh*1.7,-\vvv*0.4) {$T=$};
\draw[-] (\hhh*4,0) rectangle (\hhh*5,\vvv*1);
\draw[-] (\hhh*5,0) rectangle (\hhh*6,\vvv*1);
\draw[-] (\hhh*6,0) rectangle (\hhh*7,\vvv*1);
\draw[-] (\hhh*3,-\vvv*1) rectangle (\hhh*4,\vvv*0);
\draw[-] (\hhh*4,-\vvv*1) rectangle (\hhh*5,\vvv*0);
\draw[-] (\hhh*3,-\vvv*2) rectangle (\hhh*4,-\vvv*1);
\node[] at (\hhh*4.5,\vvv*0.5) {$1'$};
\node[] at (\hhh*5.5,\vvv*0.5) {$1$};
\node[] at (\hhh*6.5,\vvv*0.5) {$2'$};
\node[] at (\hhh*3.5,-\vvv*0.5) {$1'$};
\node[] at (\hhh*4.5,-\vvv*0.5) {$1$};
\node[] at (\hhh*3.5,-\vvv*1.5) {$2$};
\end{tikzpicture},
\end{center}
then
\begin{center}
\begin{tikzpicture}
\def \hhh{5mm}    
\def \vvv{5mm}    
\node[] at (\hhh*0.3,-\vvv*0.4) {$\stan(T)=$};
\draw[-] (\hhh*4,0) rectangle (\hhh*5,\vvv*1);
\draw[-] (\hhh*5,0) rectangle (\hhh*6,\vvv*1);
\draw[-] (\hhh*6,0) rectangle (\hhh*7,\vvv*1);
\draw[-] (\hhh*3,-\vvv*1) rectangle (\hhh*4,\vvv*0);
\draw[-] (\hhh*4,-\vvv*1) rectangle (\hhh*5,\vvv*0);
\draw[-] (\hhh*3,-\vvv*2) rectangle (\hhh*4,-\vvv*1);
\node[] at (\hhh*4.5,\vvv*0.5) {$1$};
\node[] at (\hhh*5.5,\vvv*0.5) {$4$};
\node[] at (\hhh*6.5,\vvv*0.5) {$5$};
\node[] at (\hhh*3.5,-\vvv*0.5) {$2$};
\node[] at (\hhh*4.5,-\vvv*0.5) {$3$};
\node[] at (\hhh*3.5,-\vvv*1.5) {$6$};
\end{tikzpicture}
\hskip 8mm
\begin{tikzpicture}
\def \hhh{5mm}    
\def \vvv{5mm}    
\node[] at (-\hhh*4.3,-\vvv*0.4) {{\rm and}};
\node[] at (\hhh*0.4,-\vvv*0.4) {$\rstan(A)=$};
\draw[-] (\hhh*4,0) rectangle (\hhh*5,\vvv*1);
\draw[-] (\hhh*5,0) rectangle (\hhh*6,\vvv*1);
\draw[-] (\hhh*6,0) rectangle (\hhh*7,\vvv*1);
\draw[-] (\hhh*3,-\vvv*1) rectangle (\hhh*4,\vvv*0);
\draw[-] (\hhh*4,-\vvv*1) rectangle (\hhh*5,\vvv*0);
\draw[-] (\hhh*3,-\vvv*2) rectangle (\hhh*4,-\vvv*1);
\node[] at (\hhh*4.5,\vvv*0.5) {$6$};
\node[] at (\hhh*5.5,\vvv*0.5) {$3$};
\node[] at (\hhh*6.5,\vvv*0.5) {$2$};
\node[] at (\hhh*3.5,-\vvv*0.5) {$5$};
\node[] at (\hhh*4.5,-\vvv*0.5) {$4$};
\node[] at (\hhh*3.5,-\vvv*1.5) {$1$};
\end{tikzpicture}.
\end{center}

\vskip 2mm
We close this section by introducing shifted jeu de taquin and shifted Knuth relations.
In 1963, Sch\"{u}tzenberger \cite{Sc1} introduced a remarkable combinatorial algorithm,
called {\it jeu de taquin},
which takes a skew semistandard Young tableau to another skew semistandard Young tableau
with the same weight, but different shape.
This can be accomplished by applying the jeu de taquin slides in succession:
\vskip 2mm
\begin{center}
\begin{tikzpicture}[baseline=0pt]
\def \hhh{5mm}      
\def \vvv{5mm}      
\def \vvvv{9mm}   
\def \hhhhh{0mm}
\draw[fill=black!40] (0+\hhhhh*0,0) rectangle (\hhh*1+\hhhhh*0,\vvv*1);
\draw[-] (\hhh*1+\hhhhh*0,0) rectangle (\hhh*2+\hhhhh*0,\vvv*1);
\draw[-] (\hhh*0+\hhhhh*0,-\vvv*1) rectangle (\hhh*2+\hhhhh*0,0);
\draw[-] (\hhh*1+\hhhhh*0,-\vvv*1) rectangle (\hhh*2+\hhhhh*0,0);
\draw[|->] (\hhh*2.7+\hhhhh*0,\vvv*1) to (\hhh*4.3+\hhhhh*0,0+\vvvv);
\draw[|->] (\hhh*2.7+\hhhhh*0,-\vvv*1) to (\hhh*4.3+\hhhhh*0,-\vvvv);
\draw[fill=black!40] (\hhh*6+\hhhhh*0,0+\vvvv) rectangle (\hhh*7+\hhhhh*0,\vvv*1+\vvvv);
\draw[-] (\hhh*0+\hhh*5+\hhhhh*0,0+\vvvv) rectangle (\hhh*1+\hhh*5+\hhhhh*0,\vvv*1+\vvvv);
\draw[-] (\hhh*0+\hhh*5+\hhhhh*0,-\vvv*1+\vvvv) rectangle (\hhh*1+\hhh*5+\hhhhh*0,0+\vvvv);
\draw[-] (\hhh*1+\hhh*5+\hhhhh*0,-\vvv*1+\vvvv) rectangle (\hhh*2+\hhh*5+\hhhhh*0,0+\vvvv);
\draw[-] (\hhh*0+\hhh*5+\hhhhh*0,0-\vvvv) rectangle (\hhh*6+\hhhhh*0,\vvv*1-\vvvv);
\draw[-] (\hhh*0+\hhh*6+\hhhhh*0,0-\vvvv) rectangle (\hhh*7+\hhhhh*0,\vvv*1-\vvvv);
\draw[fill=black!40] (\hhh*5+\hhhhh*0,-\vvv*1-\vvvv) rectangle (\hhh*6+\hhhhh*0,0-\vvvv);
\draw[-] (\hhh*6+\hhhhh*0,-\vvv*1-\vvvv) rectangle (\hhh*7+\hhhhh*0,-\vvvv);
\node (A11) at (\hhh*6+\hhhhh*0,0+\vvvv) {};
\node () [right=of A11] {if $a < b$};
\node at (\hhh*6,-\vvvv) (A21) {};
\node () [right=of A21] {if $a \geq b$};
\node[] at (\hhh*1.5+\hhhhh*0,\vvv*0.5) () {$a$};
\node[] at (\hhh*0.5+\hhhhh*0,-\vvv*0.5) () {$b$};
\node[] at (\hhh*1.5+\hhhhh*0,-\vvv*0.5) () {$c$};
\node[] at (\hhh*0.5+\hhh*5+\hhhhh*0,\vvv*0.5+\vvvv) () {$a$};
\node[] at (\hhh*0.5+\hhh*5+\hhhhh*0,-\vvv*0.5+\vvvv) () {$b$};
\node[] at (\hhh*1.5+\hhh*5+\hhhhh*0,-\vvv*0.5+\vvvv) () {$c$};
\node[] at (\hhh*0.5+\hhh*5+\hhhhh*0,\vvv*0.5-\vvvv) () {$b$};
\node[] at (\hhh*1.5+\hhh*5+\hhhhh*0,\vvv*0.5-\vvvv) () {$a$};
\node[] at (\hhh*1.5+\hhh*5+\hhhhh*0,-\vvv*0.5-\vvvv) () {$c$};
\end{tikzpicture}
\end{center}
\vskip 2mm
In the 1980's, Worley \cite{Wo} and Sagan \cite{Sa} independently introduced the shifted analogue of
Sch\"{u}tzenberger's jeu de taquin, which is a combinatorial algorithm taking
an SYT to another SYT with the same weight, but different shape.
This can be carried out by applying the shifted jeu de taquin slides
described in Figure~\ref{figure1} (for short, shifted slides) in succession.
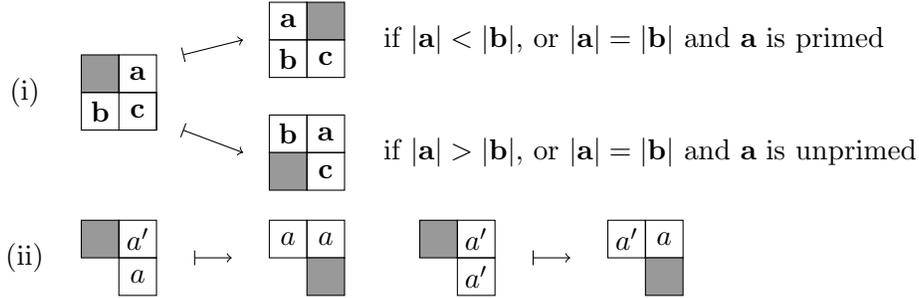
\begin{figure}[h!]
\centering
\begin{tikzpicture}[baseline=0pt]
\def \hhh{5mm}      
\def \vvv{5mm}      
\def \vvvv{7mm}   
\def \vvvvv{-4mm}    
\node[] at (-\hhh*1.5,-\vvv*0.0-\vvvv) () {(i)};
\draw[fill=black!40] (0,0-\vvvv) rectangle (\hhh*1,\vvv*1-\vvvv);
\draw[-] (\hhh*1,0-\vvvv) rectangle (\hhh*2,\vvv*1-\vvvv);
\draw[-] (\hhh*0,-\vvv*1-\vvvv) rectangle (\hhh*2,0-\vvvv);
\draw[-] (\hhh*1,-\vvv*1-\vvvv) rectangle (\hhh*2,0-\vvvv);
\draw[|->] (\hhh*2.7,\vvv*1-\vvvv) to (\hhh*4.3,0);
\draw[|->] (\hhh*2.7,-\vvv*1-\vvvv) to (\hhh*4.3,-\vvv*3);
\draw[fill=black!40] (\hhh*6,0) rectangle (\hhh*7,\vvv*1);
\draw[-] (\hhh*0+\hhh*5,0) rectangle (\hhh*1+\hhh*5,\vvv*1);
\draw[-] (\hhh*0+\hhh*5,-\vvv*1) rectangle (\hhh*1+\hhh*5,0);
\draw[-] (\hhh*1+\hhh*5,-\vvv*1) rectangle (\hhh*2+\hhh*5,0);
\draw[-] (\hhh*0+\hhh*5,0-\vvv*3) rectangle (\hhh*6,\vvv*1-\vvv*3);
\draw[-] (\hhh*0+\hhh*6,0-\vvv*3) rectangle (\hhh*7,\vvv*1-\vvv*3);
\draw[fill=black!40] (\hhh*5,-\vvv*1-\vvv*3) rectangle (\hhh*6,0-\vvv*3);
\draw[-] (\hhh*6,-\vvv*1-\vvv*3) rectangle (\hhh*7,-\vvv*3);
\node (A) at (\hhh*5.5,0) {};
\node () [right=of A]
{if $|\bfa| < |\bfb|$, or $|\bfa|=|\bfb|$ and $\bfa$ is primed};
\node at (\hhh*5.5,-\vvv*3) (A21) {};
\node () [right=of A21]
{if $|\bfa| > |\bfb|$, or $|\bfa|=|\bfb|$ and $\bfa$ is unprimed};
\node[] at (\hhh*1.5,\vvv*0.5-\vvvv) () {$\bfa$};
\node[] at (\hhh*0.5,-\vvv*0.5-\vvvv) () {$\bfb$};
\node[] at (\hhh*1.5,-\vvv*0.5-\vvvv) () {$\bfc$};
\node[] at (\hhh*0.5+\hhh*5,\vvv*0.5) () {$\bfa$};
\node[] at (\hhh*0.5+\hhh*5,-\vvv*0.5) () {$\bfb$};
\node[] at (\hhh*1.5+\hhh*5,-\vvv*0.5) () {$\bfc$};
\node[] at (\hhh*0.5+\hhh*5,-\vvv*2.5) () {$\bfb$};
\node[] at (\hhh*1.5+\hhh*5,-\vvv*2.5) () {$\bfa$};
\node[] at (\hhh*1.5+\hhh*5,-\vvv*3.5) () {$\bfc$};
\node[] at (-\hhh*1.5,-\vvv*5+\vvvvv) () {(ii)};
\draw[fill=black!40] (0,-\vvv*5+\vvvvv) rectangle (\hhh*1,-\vvv*4+\vvvvv);
\draw[-] (\hhh*1,-\vvv*5+\vvvvv) rectangle (\hhh*2,-\vvv*4+\vvvvv);
\draw[-] (\hhh*1,-\vvv*6+\vvvvv) rectangle (\hhh*2,-\vvv*5+\vvvvv);
\node[] at (\hhh*1.5,-\vvv*4.5+\vvvvv) () {$a'$};
\node[] at (\hhh*1.5,-\vvv*5.5+\vvvvv) () {$a$};
\draw[|->] (\hhh*3,-\vvv*5+\vvvvv) to (\hhh*4,-\vvv*5+\vvvvv);
\draw[-] (\hhh*5,-\vvv*5+\vvvvv) rectangle (\hhh*6,-\vvv*4+\vvvvv);
\draw[-] (\hhh*6,-\vvv*5+\vvvvv) rectangle (\hhh*7,-\vvv*4+\vvvvv);
\draw[fill=black!40] (\hhh*6,-\vvv*6+\vvvvv) rectangle (\hhh*7,-\vvv*5+\vvvvv);
\node[] at (\hhh*5.5,-\vvv*4.5+\vvvvv) () {$a$};
\node[] at (\hhh*6.5,-\vvv*4.5+\vvvvv) () {$a$};
\draw[fill=black!40] (\hhh*9,-\vvv*5+\vvvvv) rectangle (\hhh*10,-\vvv*4+\vvvvv);
\draw[-] (\hhh*10,-\vvv*5+\vvvvv) rectangle (\hhh*11,-\vvv*4+\vvvvv);
\draw[-] (\hhh*10,-\vvv*6+\vvvvv) rectangle (\hhh*11,-\vvv*5+\vvvvv);
\node[] at (\hhh*10.5,-\vvv*4.5+\vvvvv) () {$a'$};
\node[] at (\hhh*10.5,-\vvv*5.5+\vvvvv) () {$a'$};
\draw[|->] (\hhh*12,-\vvv*5+\vvvvv) to (\hhh*13,-\vvv*5+\vvvvv);
\draw[-] (\hhh*14,-\vvv*5+\vvvvv) rectangle (\hhh*15,-\vvv*4+\vvvvv);
\draw[-] (\hhh*15,-\vvv*5+\vvvvv) rectangle (\hhh*16,-\vvv*4+\vvvvv);
\draw[fill=black!40] (\hhh*15,-\vvv*6+\vvvvv) rectangle (\hhh*16,-\vvv*5+\vvvvv);
\node[] at (\hhh*14.5,-\vvv*4.5+\vvvvv) () {$a'$};
\node[] at (\hhh*15.5,-\vvv*4.5+\vvvvv) () {$a$};
\end{tikzpicture}
\caption{The shifted slides}
\label{figure1}
\end{figure}

Given an SYT $T$ of skew shape,
one can apply shifted slides to $T$ iteratively
to get an SYT of normal shape
(which means no more shifted slides are possible).
This can be done in many different ways, but
the resulting SYT of normal shape is known to be
same for all possible choices. This SYT is called
the {\it rectification} of $T$ and denoted by $\Rect(T)$.

\begin{example}\label{example-jdt}
The following shows how the shifted jeu de taquin
is being proceeded.
\vskip 2mm
\begin{center}
\begin{tikzpicture}
\def \hhh{4.6mm}    
\def \vvv{4.6mm}    
\def \hhhhh{30mm}  
\def \iiii{18mm}    
\node[] at (-\hhh*1+\hhhhh*0,-\vvv*0.0) () {$T=$};
\draw[-,black!30] (\hhh*0+\hhhhh*0,0) rectangle (\hhh*1+\hhhhh*0,\vvv*1);
\draw[-,black!30] (\hhh*1+\hhhhh*0,0) rectangle (\hhh*2+\hhhhh*0,\vvv*1);
\draw[-] (\hhh*2+\hhhhh*0,0) rectangle (\hhh*3+\hhhhh*0,\vvv*1);
\draw[-] (\hhh*3+\hhhhh*0,0) rectangle (\hhh*4+\hhhhh*0,\vvv*1);
\draw[fill=black!20] (\hhh*1,-\vvv*1) rectangle (\hhh*2,\vvv*0);
\draw[-] (\hhh*2+\hhhhh*0,-\vvv*1) rectangle (\hhh*3+\hhhhh*0,-\vvv*0);
\draw[-] (\hhh*3+\hhhhh*0,-\vvv*1) rectangle (\hhh*4+\hhhhh*0,-\vvv*0);
\draw[-] (\hhh*2+\hhhhh*0,-\vvv*2) rectangle (\hhh*3+\hhhhh*0,-\vvv*1);
\node[] at (\hhh*2.5+\hhhhh*0,\vvv*0.5) () {\small $1'$};
\node[] at (\hhh*3.5+\hhhhh*0,\vvv*0.5) () {$2'$};
\node[] at (\hhh*2.5+\hhhhh*0,-\vvv*0.5) () {$2'$};
\node[] at (\hhh*3.5+\hhhhh*0,-\vvv*0.5) () {$3$};
\node[] at (\hhh*2.5+\hhhhh*0,-\vvv*1.5) () {$2'$};
\draw[->,decorate,decoration={snake,amplitude=.4mm,segment length=2mm,post length=1mm}]
(\hhh*4.8+\hhhhh*0,-\vvv*0) to (-\hhh*0.8+\hhhhh*1,-\vvv*0);
\draw[-,black!30] (\hhh*2+\hhhhh*1,-\vvv*2) rectangle (\hhh*3+\hhhhh*1,-\vvv*1);
\draw[-,black!30] (\hhh*0+\hhhhh*1,0) rectangle (\hhh*1+\hhhhh*1,\vvv*1);
\draw[fill=black!20] (\hhh*1+\hhhhh*1,0) rectangle (\hhh*2+\hhhhh*1,\vvv*1);
\draw[-] (\hhh*2+\hhhhh*1,0) rectangle (\hhh*3+\hhhhh*1,\vvv*1);
\draw[-] (\hhh*3+\hhhhh*1,0) rectangle (\hhh*4+\hhhhh*1,\vvv*1);
\draw[-] (\hhh*1+\hhhhh*1,-\vvv*1) rectangle (\hhh*2+\hhhhh*1,-\vvv*0);
\draw[-] (\hhh*2+\hhhhh*1,-\vvv*1) rectangle (\hhh*3+\hhhhh*1,-\vvv*0);
\draw[-] (\hhh*3+\hhhhh*1,-\vvv*1) rectangle (\hhh*4+\hhhhh*1,-\vvv*0);
\node[] at (\hhh*2.5+\hhhhh*1,\vvv*0.5) () {$1'$};
\node[] at (\hhh*3.5+\hhhhh*1,\vvv*0.5) () {$2'$};
\node[] at (\hhh*1.5+\hhhhh*1,-\vvv*0.5) () {$2'$};
\node[] at (\hhh*2.5+\hhhhh*1,-\vvv*0.5) () {$2$};
\node[] at (\hhh*3.5+\hhhhh*1,-\vvv*0.5) () {$3$};
\draw[->,decorate,decoration={snake,amplitude=.4mm,segment length=2mm,post length=1mm}]
(\hhh*4.8+\hhhhh*1,-\vvv*0) to (-\hhh*0.8+\hhhhh*2,-\vvv*0);
\draw[-,black!30] (\hhh*3+\hhhhh*2,-\vvv*1) rectangle (\hhh*4+\hhhhh*2,\vvv*0);
\draw[-,black!30] (\hhh*2+\hhhhh*2,-\vvv*2) rectangle (\hhh*3+\hhhhh*2,-\vvv*1);
\draw[fill=black!20] (\hhh*0+\hhhhh*2,0) rectangle (\hhh*1+\hhhhh*2,\vvv*1);
\draw[-] (\hhh*1+\hhhhh*2,0) rectangle (\hhh*2+\hhhhh*2,\vvv*1);
\draw[-] (\hhh*2+\hhhhh*2,0) rectangle (\hhh*3+\hhhhh*2,\vvv*1);
\draw[-] (\hhh*3+\hhhhh*2,0) rectangle (\hhh*4+\hhhhh*2,\vvv*1);
\draw[-] (\hhh*1+\hhhhh*2,-\vvv*1) rectangle (\hhh*2+\hhhhh*2,-\vvv*0);
\draw[-] (\hhh*2+\hhhhh*2,-\vvv*1) rectangle (\hhh*3+\hhhhh*2,-\vvv*0);
\node[] at (\hhh*1.5+\hhhhh*2,\vvv*0.5) () {$1'$};
\node[] at (\hhh*2.5+\hhhhh*2,\vvv*0.5) () {$2'$};
\node[] at (\hhh*3.5+\hhhhh*2,\vvv*0.5) () {$3$};
\node[] at (\hhh*1.5+\hhhhh*2,-\vvv*0.5) () {$2'$};
\node[] at (\hhh*2.5+\hhhhh*2,-\vvv*0.5) () {$2$};
\draw[->,decorate,decoration={snake,amplitude=.4mm,segment length=2mm,post length=1mm}]
(\hhh*4.8+\hhhhh*2,\vvv*0) to (-\hhh*0.4+\hhhhh*3,\vvv*0);
\node[] at (-\hhh*2.1+\hhhhh*3+\iiii,-\vvv*0.0) () {$\Rect(T)=$};
\draw[-,black!30] (\hhh*3+\hhhhh*3+\iiii,-\vvv*0) rectangle (\hhh*4+\hhhhh*3+\iiii,\vvv*1);
\draw[-,black!30] (\hhh*3+\hhhhh*3+\iiii,-\vvv*1) rectangle (\hhh*4+\hhhhh*3+\iiii,\vvv*0);
\draw[-,black!30] (\hhh*2+\hhhhh*3+\iiii,-\vvv*2) rectangle (\hhh*3+\hhhhh*3+\iiii,-\vvv*1);
\draw[-,black!30] (\hhh*2+\hhhhh*3+\iiii,-\vvv*1) rectangle (\hhh*3+\hhhhh*3+\iiii,-\vvv*0);
\draw[-] (\hhh*0+\hhhhh*3+\iiii,0) rectangle (\hhh*1+\hhhhh*3+\iiii,\vvv*1);
\draw[-] (\hhh*1+\hhhhh*3+\iiii,0) rectangle (\hhh*2+\hhhhh*3+\iiii,\vvv*1);
\draw[-] (\hhh*2+\hhhhh*3+\iiii,0) rectangle (\hhh*3+\hhhhh*3+\iiii,\vvv*1);
\draw[-] (\hhh*3+\hhhhh*3+\iiii,0) rectangle (\hhh*4+\hhhhh*3+\iiii,\vvv*1);
\draw[-] (\hhh*1+\hhhhh*3+\iiii,-\vvv*1) rectangle (\hhh*2+\hhhhh*3+\iiii,-\vvv*0);
\node[] at (\hhh*0.5+\hhhhh*3+\iiii,\vvv*0.5) () {$1'$};
\node[] at (\hhh*1.5+\hhhhh*3+\iiii,\vvv*0.5) () {$2'$};
\node[] at (\hhh*2.5+\hhhhh*3+\iiii,\vvv*0.5) () {$2$};
\node[] at (\hhh*3.5+\hhhhh*3+\iiii,\vvv*0.5) () {$3$};
\node[] at (\hhh*1.5+\hhhhh*3+\iiii,-\vvv*0.5) () {$2'$};
\end{tikzpicture}
\end{center}
Here the shaded boxes denote inner corners.
\end{example}

We say that two SYTs are {\it shifted jeu de taquin equivalent}
if they can be changed into each other by a sequence of shifted slides,
with their inverses,
if and only if
they have the same rectification.
As in the ordinary case,
shifted jeu de taquin equivalence can be described in terms of shifted Knuth equivalence.
Given any word $u = u_1 u_2 \cdots u_n$ in $X$,
the possible shifted Knuth relations that can
be applied to $u$ in order to form another word are
\begin{itemize}
\item replace $xzy$ by $zxy$ or vice versa
if $x < y \leq z$,

\item replace $yxz$ by $yzx$ or vice versa
if $x \leq y < z$,

\item for the last two elements of $u$ we may

replace $xy$ by $yx$ or vice versa if $|x| \neq |y|$,

replace $xx$ by $x'x$ or vice versa.
\end{itemize}
Here, $(x')'$ denotes $x$ for an unprimed letter $x$.
Shifted Knuth equivalence of words, denoted by $\equiv_{\rm \tiny SK}$,
can induce the equivalence relation on the set of words in $X$ in such a way that
$u$ and $v$ are {\it shifted Knuth equivalent} if
one can be transformed into the other under the relations above.
For more information, refer to \cite{Sa}.

For any filling $T$ of skew shape,
{\it the word of $T$}, denoted by $\w(T)$, is defined to be the word obtained by
reading entries from right to left, starting with the top row in $T$.

\begin{prop}{\rm (\cite{Sa, Wo})}
Given two SYTs, they are shifted jeu de taquin equivalent if and only if their words are shifted Knuth equivalent.
\end{prop}

\section{The shifted tableau switching}
\label{sect-ss-proce}
\subsection{The shifted tableau switching}
\label{subsect-sh-tab-sw}
We here introduce the shifted version of the tableau switching given in \cite{BSS}.
For this purpose we first introduce the notion of shifted perforated tableaux.

\begin{df}\label{Def-Perforated}{\rm
For a double border strip $\theta$,
a {\it shifted perforated $\bfa$-tableau} $A$ in $\theta$
is a filling of some of the boxes in $\theta$ with $a', a \in X$ satisfying
the following conditions:
\begin{itemize}
  \item no $a'$-boxes are placed to southeast of any $a$-box,
  \item each column of $A$ contains at most one $a$,
  \item each row of $A$ contains at most one $a'$, and
  \item the main diagonal of $A$ contains at most one $\bfa$.
\end{itemize}
}\end{df}

Define the {\it shape} $\sh(A)$ of $A$ by the subset of $\theta$ consisting of all $\bfa$-boxes.
For a shifted perforated $\bfa$-tableau $A$ and $\bfb$-tableau $B$ in $\theta$,
$(A,B)$ is called a {\it shifted perforated $(\bfa,\bfb)$-pair of shape $\theta$} if
$\theta$ is the disjoint union of $\sh(A)$ and $\sh(B)$.
For a shifted perforated $(\bfa,\bfb)$-pair $(A,B)$ of shape $\theta$
we can consider $A \cup B$ in the obvious way.
For instance,
\vskip 2mm
\begin{center}
\begin{tikzpicture}
\def \hhh{5mm}    
\def \vvv{5mm}    
\def \hhhhh{55mm}  
\node[] at (\hhh*0.8,-\vvv*0.4) {$A=$};
\draw[fill=black!30] (\hhh*3+\hhhhh*0,0) rectangle (\hhh*4+\hhhhh*0,\vvv*1);
\draw[-] (\hhh*4+\hhhhh*0,0) rectangle (\hhh*5+\hhhhh*0,\vvv*1);
\draw[-] (\hhh*5+\hhhhh*0,0) rectangle (\hhh*6+\hhhhh*0,\vvv*1);
\draw[fill=black!30] (\hhh*6+\hhhhh*0,0) rectangle (\hhh*7+\hhhhh*0,\vvv*1);
\draw[fill=black!30] (\hhh*2+\hhhhh*0,-\vvv*1) rectangle (\hhh*3+\hhhhh*0,\vvv*0);
\draw[-] (\hhh*3+\hhhhh*0,-\vvv*1) rectangle (\hhh*4+\hhhhh*0,\vvv*0);
\draw[-] (\hhh*4+\hhhhh*0,-\vvv*1) rectangle (\hhh*5+\hhhhh*0,\vvv*0);
\draw[-] (\hhh*3+\hhhhh*0,-\vvv*2) rectangle (\hhh*4+\hhhhh*0,-\vvv*1);
\draw[fill=black!30] (\hhh*4+\hhhhh*0,-\vvv*2) rectangle (\hhh*5+\hhhhh*0,-\vvv*1);
\node[] at (\hhh*3.5+\hhhhh*0,\vvv*0.5) {$1'$};
\node[] at (\hhh*6.5+\hhhhh*0,\vvv*0.5) {$1$};
\node[] at (\hhh*2.5+\hhhhh*0,-\vvv*0.5) {$1$};
\node[] at (\hhh*4.5+\hhhhh*0,-\vvv*1.5) {$1$};
\node[] at (\hhh*3.8+\hhhhh*0.5,-\vvv*0.4) {and};
\node[] at (\hhh*0.8+\hhhhh*1,-\vvv*0.4) {$B=$};
\draw[fill=black!30] (\hhh*3+\hhhhh*1,0) rectangle (\hhh*4+\hhhhh*1,\vvv*1);
\draw[-] (\hhh*4+\hhhhh*1,0) rectangle (\hhh*5+\hhhhh*1,\vvv*1);
\draw[-] (\hhh*5+\hhhhh*1,0) rectangle (\hhh*6+\hhhhh*1,\vvv*1);
\draw[fill=black!30] (\hhh*6+\hhhhh*1,0) rectangle (\hhh*7+\hhhhh*1,\vvv*1);
\draw[fill=black!30] (\hhh*2+\hhhhh*1,-\vvv*1) rectangle (\hhh*3+\hhhhh*1,\vvv*0);
\draw[-] (\hhh*3+\hhhhh*1,-\vvv*1) rectangle (\hhh*4+\hhhhh*1,\vvv*0);
\draw[-] (\hhh*4+\hhhhh*1,-\vvv*1) rectangle (\hhh*5+\hhhhh*1,\vvv*0);
\draw[-] (\hhh*3+\hhhhh*1,-\vvv*2) rectangle (\hhh*4+\hhhhh*1,-\vvv*1);
\draw[fill=black!30] (\hhh*4+\hhhhh*1,-\vvv*2) rectangle (\hhh*5+\hhhhh*1,-\vvv*1);
\node[] at (\hhh*4.5+\hhhhh*1,\vvv*0.5) {$1'$};
\node[] at (\hhh*5.5+\hhhhh*1,\vvv*0.5) {$1$};
\node[] at (\hhh*3.5+\hhhhh*1,-\vvv*0.5) {$1'$};
\node[] at (\hhh*4.5+\hhhhh*1,-\vvv*0.5) {$1$};
\node[] at (\hhh*3.5+\hhhhh*1,-\vvv*1.5) {$1$};
\end{tikzpicture}
\end{center}
form a shifted perforated $(\bf{1},\bf{1})$-pair of $(6,3,2)/(2)$ and
\vskip 2mm
\begin{center}
\begin{tikzpicture}
\def \hhh{5mm}    
\def \vvv{5mm}    
\node[] at (-\hhh*0.2,-\vvv*0.4) {$A \cup B =$};
\draw[fill=black!30] (\hhh*3,0) rectangle (\hhh*4,\vvv*1);
\draw[-] (\hhh*4,0) rectangle (\hhh*5,\vvv*1);
\draw[-] (\hhh*5,0) rectangle (\hhh*6,\vvv*1);
\draw[fill=black!30] (\hhh*6,0) rectangle (\hhh*7,\vvv*1);
\draw[fill=black!30] (\hhh*2,-\vvv*1) rectangle (\hhh*3,\vvv*0);
\draw[-] (\hhh*3,-\vvv*1) rectangle (\hhh*4,\vvv*0);
\draw[-] (\hhh*4,-\vvv*1) rectangle (\hhh*5,\vvv*0);
\draw[-] (\hhh*3,-\vvv*2) rectangle (\hhh*4,-\vvv*1);
\draw[fill=black!30] (\hhh*4,-\vvv*2) rectangle (\hhh*5,-\vvv*1);
\node[] at (\hhh*3.5,\vvv*0.5) {$1'$};
\node[] at (\hhh*4.5,\vvv*0.5) {$1'$};
\node[] at (\hhh*5.5,\vvv*0.5) {$1$};
\node[] at (\hhh*6.5,\vvv*0.5) {$1$};
\node[] at (\hhh*2.5,-\vvv*0.5) {$1$};
\node[] at (\hhh*3.5,-\vvv*0.5) {$1'$};
\node[] at (\hhh*4.5,-\vvv*0.5) {$1$};
\node[] at (\hhh*3.5,-\vvv*1.5) {$1$};
\node[] at (\hhh*4.5,-\vvv*1.5) {$1$};
\end{tikzpicture}.
\end{center}
We say that
$B$ {\it extends} $A$ if ${\rm sh}(B)$ extends ${\rm sh}(A)$.

As the tableau switching in \cite{BSS}, our shifted tableau switching can be performed by applying
a sequence of elementary transformations
called switches.
Let $(A,B)$ be a shifted perforated $(\bfa,\bfb)$-pair.
Interchanging an $\bfa$-box and a $\bfb$-box in $A\cup B$
is called a {\it switch} if it is subject to the following rules.
\begin{itemize}
\item The case where an $\bfa$-box is adjacent to a unique $\bfb$-box:
\vskip 2mm
\begin{tikzpicture}
\def \hhh{5mm}    
\def \vvv{5mm}    
\def \hhhhh{60mm}   
\node[] at (-\hhh*1.3,0) () {(S1)};
\draw[fill=black!30] (0,-\vvv*0.5) rectangle (\hhh*1,\vvv*0.5);
\draw[-] (\hhh*1,-\vvv*0.5) rectangle (\hhh*2,\vvv*0.5);
\draw[|->] (\hhh*3,\vvv*0) to (\hhh*4,\vvv*0);
\draw[-] (\hhh*5,-\vvv*0.5) rectangle (\hhh*6,\vvv*0.5);
\draw[fill=black!30] (\hhh*6,-\vvv*0.5) rectangle (\hhh*7,\vvv*0.5);
\node[] at (\hhh*0.5,\vvv*0) () {$\bfa$};
\node[] at (\hhh*1.5,\vvv*0) () {$\bfb$};
\node[] at (\hhh*5.5,\vvv*0) () {$\bfb$};
\node[] at (\hhh*6.5,\vvv*0) () {$\bfa$};
\node[] at (-\hhh*1.3+\hhhhh*1,0) () {(S2)};
\draw[fill=black!30] (0+\hhhhh*1,0) rectangle (\hhh*1+\hhhhh*1,\vvv*1);
\draw[-] (\hhh*0+\hhhhh*1,-\vvv*1) rectangle (\hhh*1+\hhhhh*1,0);
\draw[|->] (\hhh*2+\hhhhh*1,\vvv*0) to (\hhh*3+\hhhhh*1,\vvv*0);
\draw[-] (\hhh*4+\hhhhh*1,0) rectangle (\hhh*5+\hhhhh*1,\vvv*1);
\draw[fill=black!30] (\hhh*4+\hhhhh*1,-\vvv*1) rectangle (\hhh*5+\hhhhh*1,\vvv*0);
\node[] at (\hhh*0.5+\hhhhh*1,\vvv*0.5) () {$\bfa$};
\node[] at (\hhh*0.5+\hhhhh*1,-\vvv*0.5) () {$\bfb$};
\node[] at (\hhh*4.5+\hhhhh*1,\vvv*0.5) () {$\bfb$};
\node[] at (\hhh*4.5+\hhhhh*1,-\vvv*0.5) () {$\bfa$};
\end{tikzpicture}
\vskip 3mm
\begin{tikzpicture}
\def \hhh{5mm}    
\def \vvv{5mm}    
\def \hhhh{60mm}   
\node[] at (-\hhh*1.3,0) () {(S3)};
\draw[fill=black!30] (0,0) rectangle (\hhh*1,\vvv*1);
\draw[-] (\hhh*1,0) rectangle (\hhh*2,\vvv*1);
\draw[-] (\hhh*1,-\vvv*1) rectangle (\hhh*2,0);
\draw[|->] (\hhh*3,0) to (\hhh*4,0);
\draw[fill=black!30] (\hhh*1+\hhh*5,0) rectangle (\hhh*2+\hhh*5,-\vvv*1);
\draw[-] (\hhh*0+\hhh*5,0) rectangle (\hhh*1+\hhh*5,\vvv*1);
\draw[-] (\hhh*1+\hhh*5,0) rectangle (\hhh*2+\hhh*5,\vvv*1);
\node[] at (-\hhh*1.3+\hhhh*1,0) () {(S4)};
\draw[fill=black!30] (0+\hhhh,0) rectangle (\hhh*1+\hhhh,\vvv*1);
\draw[fill=black!30] (\hhh*1+\hhhh,0) rectangle (\hhh*2+\hhhh,\vvv*1);
\draw[-] (\hhh*1+\hhhh,-\vvv*1) rectangle (\hhh*2+\hhhh,0);
\draw[|->] (\hhh*3+\hhhh,0) to (\hhh*4+\hhhh*1,0);
\draw[-] (\hhh*0+\hhh*5+\hhhh*1,0) rectangle (\hhh*1+\hhh*5+\hhhh*1,\vvv*1);
\draw[fill=black!30] (\hhh*1+\hhh*5+\hhhh*1,0) rectangle (\hhh*2+\hhh*5+\hhhh*1,-\vvv*1);
\draw[fill=black!30] (\hhh*1+\hhh*5+\hhhh*1,0) rectangle (\hhh*2+\hhh*5+\hhhh*1,\vvv*1);
\node[] at (\hhh*0.5,\vvv*0.5) () {$\bfa$};
\node[] at (\hhh*1.5,\vvv*0.5) () {$b'$};
\node[] at (\hhh*1.5,-\vvv*0.5) () {$\bfb$};
\node[] at (\hhh*0.5+\hhh*5,\vvv*0.5) () {$\bfb$};
\node[] at (\hhh*1.5+\hhh*5,\vvv*0.5) () {$b$};
\node[] at (\hhh*1.5+\hhh*5,-\vvv*0.5) () {$\bfa$};
\node[] at (\hhh*0.5+\hhhh,\vvv*0.5) () {$\bfa$};
\node[] at (\hhh*1.5+\hhhh,\vvv*0.5) () {$a$};
\node[] at (\hhh*1.5+\hhhh,-\vvv*0.5) () {$\bfb$};
\node[] at (\hhh*0.5+\hhh*5+\hhhh*1,\vvv*0.5) () {$\bfb$};
\node[] at (\hhh*1.5+\hhh*5+\hhhh*1,\vvv*0.5) () {$a'$};
\node[] at (\hhh*1.5+\hhh*5+\hhhh*1,-\vvv*0.5) () {$\bfa$};
\end{tikzpicture}

\item The case where an $\bfa$-box is adjacent to two $\bfb$-boxes:
\vskip 3mm
\begin{tikzpicture}
\def \hhh{5mm}   
\def \vvv{5mm}   
\def \hhhh{60mm}  
\def \mult{8}    
\node[] at (-\hhh*1.3,0) () {(S5)};
\draw[fill=black!30] (0,0) rectangle (\hhh*1,\vvv*1);
\draw[-] (\hhh*1,0) rectangle (\hhh*2,\vvv*1);
\draw[-] (\hhh*0,-\vvv*1) rectangle (\hhh*1,0);
\draw[|->] (\hhh*3,0) to (\hhh*4,0);
\draw[-] (\hhh*0+\hhh*5,0) rectangle (\hhh*1+\hhh*5,\vvv*1);
\draw[fill=black!30] (\hhh*1+\hhh*5,0) rectangle (\hhh*2+\hhh*5,\vvv*1);
\draw[-] (\hhh*0+\hhh*5,-\vvv*1) rectangle (\hhh*1+\hhh*5,0);
\node[] at (\hhh*0.5,\vvv*0.5) () {$\bfa$};
\node[] at (\hhh*1.5,\vvv*0.5) () {$b'$};
\node[] at (\hhh*0.5,-\vvv*0.5) () {$\bfb$};
\node[] at (\hhh*0.5+\hhh*5,\vvv*0.5) () {$b'$};
\node[] at (\hhh*1.5+\hhh*5,\vvv*0.5) () {$\bfa$};
\node[] at (\hhh*0.5+\hhh*5,-\vvv*0.5) () {$\bfb$};
\node[] at (-\hhh*1.3+\hhhh*1,0) () {(S6)};
\draw[fill=black!30] (0+\hhhh,0) rectangle (\hhh*1+\hhhh,\vvv*1);
\draw[-] (\hhh*1+\hhhh,0) rectangle (\hhh*2+\hhhh,\vvv*1);
\draw[-] (\hhh*0+\hhhh,-\vvv*1) rectangle (\hhh*1+\hhhh,0);
\draw[|->] (\hhh*3+\hhhh,0) to (\hhh*4+\hhhh,0);
\draw[-] (\hhh*1+\hhh*5+\hhhh*1,0) rectangle (\hhh*2+\hhh*5+\hhhh*1,\vvv*1);
\draw[-] (\hhh*0+\hhh*5+\hhhh*1,0) rectangle (\hhh*1+\hhh*5+\hhhh*1,\vvv*1);
\draw[fill=black!30] (\hhh*5+\hhhh*1,-\vvv*1) rectangle (\hhh*1+\hhh*5+\hhhh*1,0);
\node[] at (\hhh*0.5+\hhhh,\vvv*0.5) () {$\bfa$};
\node[] at (\hhh*1.5+\hhhh,\vvv*0.5) () {$b$};
\node[] at (\hhh*0.5+\hhhh,-\vvv*0.5) () {$\bfb$};
\node[] at (\hhh*0.5+\hhh*5+\hhhh*1,\vvv*0.5) () {$\bfb$};
\node[] at (\hhh*0.5+\hhh*5+\hhhh*1,-\vvv*0.5) () {$\bfa$};
\node[] at (\hhh*1.5+\hhh*5+\hhhh*1,\vvv*0.5) () {$b$};
\end{tikzpicture}
\vskip 3mm
\begin{tikzpicture}
\def \hhh{5mm}    
\def \vvv{5mm}    
\def \hhhhh{60mm}   
\node[] at (-\hhh*1.3,0) () {(S7)};
\draw[fill=black!30] (0,0) rectangle (\hhh*1,\vvv*1);
\draw[fill=black!30] (\hhh*1,0) rectangle (\hhh*2,\vvv*1);
\draw[-] (\hhh*1,-\vvv*1) rectangle (\hhh*2,0);
\draw[-] (\hhh*2,0) rectangle (\hhh*3,\vvv*1);
\draw[|->] (\hhh*4,0) to (\hhh*5,0);
\draw[-] (\hhh*6,0) rectangle (\hhh*7,\vvv*1);
\draw[fill=black!30] (\hhh*7,0) rectangle (\hhh*8,\vvv*1);
\draw[-] (\hhh*8,0) rectangle (\hhh*9,\vvv*1);
\draw[fill=black!30] (\hhh*7,0) rectangle (\hhh*8,-\vvv*1);
\node[] at (\hhh*0.5,\vvv*0.5) () {$\bfa$};
\node[] at (\hhh*1.5,\vvv*0.5) () {$a$};
\node[] at (\hhh*2.5,\vvv*0.5) () {$b$};
\node[] at (\hhh*1.5,-\vvv*0.5) () {$\bfb$};
\node[] at (\hhh*6.5,\vvv*0.5) () {$\bfb$};
\node[] at (\hhh*7.5,\vvv*0.5) () {$a'$};
\node[] at (\hhh*8.5,\vvv*0.5) () {$b$};
\node[] at (\hhh*7.5,-\vvv*0.5) () {$\bfa$};
\end{tikzpicture}
\end{itemize}
An $\bfa$-box in $A \cup B$ is said to be {\it fully switched}
if it can be switched with no $\bfb$-boxes.
If every $\bfa$-box in $A \cup B$ is fully switched,
then $A \cup B$ is said to be {\it fully switched}.

\begin{rem}\label{remark-switch=jdt}
(i) The shifted slides in FIGURE \ref{figure1} can be recovered from the above switches by viewing
$\bfa$-boxes as empty boxes.

(ii) The switches (S3), (S4), and (S7) denote the diagonal switches.

(iii) The switches (S3), (S4), and (S7) cannot be expressed as a composition of two other switches.
For instance, (S7) is different from the composition of (S6) and (S1).  

(iv)
In the switches (S3), (S4), and (S7),
one may think in such a way that the $(i,i)$ $\bfa$-box on the main diagonal is moved directly to the $(i+1,i+1)$ box.
However, to define the shifted switching path later, it would be more desirable to think in the fashion that
in case of (S3), the $(i,i)$ box is moved to the east and then to the south and in cases of (S4) and (S7),
the $(i,i+1)$ box is moved to the south and then the $(i,i)$ box to the east.
Refer to the following figure:
\begin{center}
\begin{tikzpicture}
\def \hhh{5mm}    
\def \vvv{5mm}    
\node[] at (-\hhh*1.3,0) () {(S3)};
\draw[fill=black!30] (0,0) rectangle (\hhh*1,\vvv*1);
\draw[-] (\hhh*1,0) rectangle (\hhh*2,\vvv*1);
\draw[-] (\hhh*1,-\vvv*1) rectangle (\hhh*2,0);
\draw[|->] (\hhh*3,0) to (\hhh*4,0);
\draw[-] (\hhh*0+\hhh*5,0) rectangle (\hhh*1+\hhh*5,\vvv*1);
\draw[-] (\hhh*1+\hhh*5,0) rectangle (\hhh*2+\hhh*5,\vvv*1);
\draw[fill=black!30] (\hhh*1+\hhh*5,0) rectangle (\hhh*2+\hhh*5,-\vvv*1);
\node[] at (\hhh*0.5,\vvv*0.5) () {$\bfa$};
\node[] at (\hhh*1.5,\vvv*0.5) () {$b'$};
\node[] at (\hhh*1.5,-\vvv*0.5) () {$\bfb$};
\node[] at (\hhh*0.5+\hhh*5,\vvv*0.5) () {$\bfb$};
\node[] at (\hhh*1.5+\hhh*5,\vvv*0.5) () {$b$};
\node[] at (\hhh*1.5+\hhh*5,-\vvv*0.5) () {$\bfa$};
\draw[->,dotted] (\hhh*0.5,\vvv*1.1) to [bend left=45] (\hhh*1.5,\vvv*1.1);
\draw[->,dotted] (\hhh*2.1,\vvv*0.5) to [bend left=45] (\hhh*2.1,-\vvv*0.5);
\end{tikzpicture}
\hskip 10mm
\begin{tikzpicture}
\def \hhh{5mm}    
\def \vvv{5mm}    
\node[] at (-\hhh*1.3,0) () {(S4)};
\draw[fill=black!30] (0,0) rectangle (\hhh*1,\vvv*1);
\draw[fill=black!30] (\hhh*1,0) rectangle (\hhh*2,\vvv*1);
\draw[-] (\hhh*1,-\vvv*1) rectangle (\hhh*2,0);
\draw[|->] (\hhh*3,0) to (\hhh*4,0);
\draw[-] (\hhh*0+\hhh*5,0) rectangle (\hhh*1+\hhh*5,\vvv*1);
\draw[fill=black!30] (\hhh*1+\hhh*5,0) rectangle (\hhh*2+\hhh*5,-\vvv*1);
\draw[fill=black!30] (\hhh*1+\hhh*5,0) rectangle (\hhh*2+\hhh*5,\vvv*1);
\node[] at (\hhh*0.5,\vvv*0.5) () {$\bfa$};
\node[] at (\hhh*1.5,\vvv*0.5) () {$a$};
\node[] at (\hhh*1.5,-\vvv*0.5) () {$\bfb$};
\node[] at (\hhh*0.5+\hhh*5,\vvv*0.5) () {$\bfb$};
\node[] at (\hhh*1.5+\hhh*5,\vvv*0.5) () {$a'$};
\node[] at (\hhh*1.5+\hhh*5,-\vvv*0.5) () {$\bfa$};
\draw[->,dotted] (\hhh*0.5,\vvv*1.1) to [bend left=45] (\hhh*1.5,\vvv*1.1);
\draw[->,dotted] (\hhh*2.1,\vvv*0.5) to [bend left=45] (\hhh*2.1,-\vvv*0.5);
\end{tikzpicture}
\vskip 2mm
\begin{tikzpicture}
\def \hhh{5mm}    
\def \vvv{5mm}    
\def \hhhhh{60mm}   
\node[] at (-\hhh*1.3,0) () {(S7)};
\draw[fill=black!30] (0,0) rectangle (\hhh*1,\vvv*1);
\draw[fill=black!30] (\hhh*1,0) rectangle (\hhh*2,\vvv*1);
\draw[-] (\hhh*1,-\vvv*1) rectangle (\hhh*2,0);
\draw[-] (\hhh*2,0) rectangle (\hhh*3,\vvv*1);
\draw[|->] (\hhh*4,0) to (\hhh*5,0);
\draw[-] (\hhh*6,0) rectangle (\hhh*7,\vvv*1);
\draw[fill=black!30] (\hhh*7,0) rectangle (\hhh*8,\vvv*1);
\draw[-] (\hhh*8,0) rectangle (\hhh*9,\vvv*1);
\draw[fill=black!30] (\hhh*7,0) rectangle (\hhh*8,-\vvv*1);
\node[] at (\hhh*0.5,\vvv*0.5) () {$\bfa$};
\node[] at (\hhh*1.5,\vvv*0.5) () {$a$};
\node[] at (\hhh*2.5,\vvv*0.5) () {$b$};
\node[] at (\hhh*1.5,-\vvv*0.5) () {$\bfb$};
\node[] at (\hhh*6.5,\vvv*0.5) () {$\bfb$};
\node[] at (\hhh*7.5,\vvv*0.5) () {$a'$};
\node[] at (\hhh*8.5,\vvv*0.5) () {$b$};
\node[] at (\hhh*7.5,-\vvv*0.5) () {$\bfa$};
\node[] at (\hhh*9.5,-\vvv*0) () {};
\node[] at (\hhh*18.5,-\vvv*0.5) () {\text{}};
\draw[->,dotted] (\hhh*0.5,\vvv*1.1) to [bend left=45] (\hhh*1.5,\vvv*1.1);
\draw[->,dotted] (\hhh*2.1,\vvv*0.5) to [bend left=45] (\hhh*2.1,-\vvv*0.5);
\end{tikzpicture}
\end{center}
\end{rem}

\vskip 2mm
Let $T$ be a filling of double border strip shape
with letters $\bfa$ and $\bfb$.
Suppose that $T$ is not fully switched.
Choose the easternmost $a$-box in $T^{(a)}$ which is the neighbor to the north or west of a $\bfb$-box
if such an $a$-box exists.
Otherwise, choose the southernmost $a'$-box
which is the neighbor to the north or west of a $\bfb$-box.
Denote this box by $\Box(T)$.
Then we apply the adequate switch among (S1) through (S7) to $\Box(T)$
and denote the resulting filling by $\sigma(T)$.
Let us apply this process in succession until $\sigma^n(T)$ gets fully switched, where $\sigma^i(T)$ is inductively defined by $\sigma(\sigma^{i-1}(T))$ for $i=2,3,\ldots, n$.
For simplicity, let us call the process from $T$ to $\sigma^n (T)$ the
{\it shifted (tableau) switching process}.
When $T = A \cup B$ for a shifted perforated $(\bfa,\bfb)$-pair $(A,B)$,
we write $A_B$ for $\sigma^n(A \cup B)^{(a)}$ and ${}^AB$ for $\sigma^n(A \cup B)^{(b)}$
to respect the notation in \cite{BSS}.

From now on, we will only consider a shifted perforated $(\bfa,\bfb)$-pair $(A, B)$ such that $A \cup B$ is not fully switched.
As the following example shows, $(\sigma(A \cup B)^{(a)}, \sigma(A \cup B)^{(b)})$
is not necessarily a shifted perforated $(\bfa,\bfb)$-pair.
\begin{center}
\begin{tikzpicture}
\def \hhh{5mm}    
\def \vvv{5mm}    
\node[] at (-\hhh*1.6,\vvv*0) {$A \cup B=$};
\draw[fill=black!30] (\hhh*0,0) rectangle (\hhh*1,\vvv*1);
\draw[-] (\hhh*1,0) rectangle (\hhh*2,\vvv*1);
\draw[-] (\hhh*0,-\vvv*1) rectangle (\hhh*1,\vvv*0);
\draw[fill=black!30] (\hhh*1,-\vvv*1) rectangle (\hhh*2,\vvv*0);
\node[] at (\hhh*0.5,\vvv*0.5) {$a$};
\node[] at (\hhh*1.5,\vvv*0.5) {$b'$};
\node[] at (\hhh*0.5,-\vvv*0.5) {$b$};
\node[] at (\hhh*1.5,-\vvv*0.5) {$a$};
\draw[-,thick,red] (\hhh*0,\vvv*1) -- (\hhh*2,\vvv*1) -- (\hhh*2,-\vvv*0)
-- (\hhh*1,-\vvv*0) -- (\hhh*1,-\vvv*1) -- (\hhh*0,-\vvv*1) -- (\hhh*0,\vvv*1);
\end{tikzpicture}
\hskip 3mm
\begin{tikzpicture}
\def \hhh{4mm}    
\def \vvv{5mm}    
\draw[|->] (\hhh*0,\vvv*0) -- (\hhh*3,\vvv*0)
node[midway,above] {\small (S5)};
\draw[black!0] (\hhh*1,-\vvv*1) rectangle (\hhh*2,\vvv*1);
\end{tikzpicture}
\begin{tikzpicture}
\def \hhh{5mm}    
\def \vvv{5mm}    
\draw[-] (\hhh*0,0) rectangle (\hhh*1,\vvv*1);
\draw[fill=black!30] (\hhh*1,0) rectangle (\hhh*2,\vvv*1);
\draw[-] (\hhh*0,-\vvv*1) rectangle (\hhh*1,\vvv*0);
\draw[fill=black!30] (\hhh*1,-\vvv*1) rectangle (\hhh*2,\vvv*0);
\node[] at (-\hhh*2,\vvv*0) {$\sigma(A \cup B)=$};
\node[] at (\hhh*0.5,\vvv*0.5) {$b'$};
\node[] at (\hhh*1.5,\vvv*0.5) {$a$};
\node[] at (\hhh*0.5,-\vvv*0.5) {$b$};
\node[] at (\hhh*1.5,-\vvv*0.5) {$a$};
\end{tikzpicture}
\end{center}
This problem, however, does not arise in practice in case where $B$ extends $A$.
Going further, we can state the following proposition.

\begin{prop}\label{prop-preserving-sp}
Let $(A, B)$ be a shifted perforated $(\bfa,\bfb)$-pair such that
$B$ extends $A$.
Suppose that
$\sigma^n(A\cup B)$ is fully switched for some positive integer $n$.
Then $\sigma^m(A\cup B)$ is still a shifted perforated $(\bfa,\bfb)$-pair
for all $m=1,2, \ldots, n$.
\end{prop}

The proof of Proposition \ref{prop-preserving-sp} will appear in the subsection \ref{subsection-proof-prop}.
By virtue of Proposition \ref{prop-preserving-sp}, we can provide the following algorithm.

\begin{algo}\label{algorithm1}
(The shifted switching process on shifted perforated pairs) \
Let $(A,B)$ be a shifted perforated $(\bfa,\bfb)$-pair such that
$B$ extends $A$.

\begin{description}
\item[Step 1] Start with $A \cup B$.

\item[Step 2]
If $A\cup B$ is fully switched, finish the algorithm.\\
If $A\cup B$ is not fully switched, apply $\sigma$ to $A \cup B$.
Put $A:=\sigma(A \cup B)^{(a)}$ and $B:=\sigma(A \cup B)^{(b)}$,
and then go to {\bf Step 1}.

\end{description}
\end{algo}

Given two SYTs $T_1$ and $T_2$, we say that
$T_2$ {\it extends} $T_1$ if $\sh(T_2)$ extends $\sh(T_1)$.
Then the following theorem can be stated.

\begin{thm}\label{Thm-Perforated}
Let $(A,B)$ be a shifted perforated $(\bfa,\bfb)$-pair such that $B$ extends $A$.
Then we have
\begin{itemize}
  \item[{\rm (a)}] ${}^AB \cup A_B$ has the same shape as $A \cup B$.

  \item[{\rm (b)}] $A_B$ extends ${}^AB$. In addition, $A_B$ and ${}^AB$ are SYTs of border strip shape.
\end{itemize}
\end{thm}
\begin{proof}
(a) The assertion follows from the fact that every switch does not affect on the shape.

(b) Since $B$ extends $A$, $B$ is clearly an SYT of border strip shape.
On the other hand, by Remark \ref{remark-switch=jdt}(i), ${}^AB$ can be obtained by applying
a sequence of shifted slides to $B$.
Here the order in which we choose inner corners is given by $\rstan(A)$.
This tells us that ${}^AB$ is an SYT of border strip shape and $A_B$ extends ${}^AB$.
In addition, by Proposition \ref{prop-preserving-sp}, $({}^AB, A_B)$ is a shifted perforated $(\bfb,\bfa)$-pair
and particularly $A_B$ is a shifted perforated $\bfa$-tableau of border strip shape.
From this it follows that $A_B$ is an SYT of border strip shape.
\end{proof}

\begin{example}\label{example-sw-perforated-pair}
Let
\begin{center}
\begin{tikzpicture}
\def \hhh{4mm}    
\def \vvv{4mm}    
\def \hhhhh{55mm}  
\node[] at (\hhh*0.8,-\vvv*0.4) {$A=$};
\draw[fill=black!30] (\hhh*3,0) rectangle (\hhh*4,\vvv*1);
\draw[fill=black!30] (\hhh*4,0) rectangle (\hhh*5,\vvv*1);
\draw[-] (\hhh*5,0) rectangle (\hhh*6,\vvv*1);
\draw[-] (\hhh*6,0) rectangle (\hhh*7,\vvv*1);
\draw[fill=black!30] (\hhh*2,-\vvv*1) rectangle (\hhh*3,\vvv*0);
\draw[fill=black!30] (\hhh*3,-\vvv*1) rectangle (\hhh*4,\vvv*0);
\draw[-] (\hhh*4,-\vvv*1) rectangle (\hhh*5,\vvv*0);
\draw[-] (\hhh*3,-\vvv*2) rectangle (\hhh*4,-\vvv*1);
\draw[-] (\hhh*4,-\vvv*2) rectangle (\hhh*5,-\vvv*1);
\node[] at (\hhh*3.5,\vvv*0.5) {$1'$};
\node[] at (\hhh*4.5,\vvv*0.5) {$1$};
\node[] at (\hhh*5.5,\vvv*0.5) {};
\node[] at (\hhh*6.5,\vvv*0.5) {};
\node[] at (\hhh*2.5,-\vvv*0.5) {$1$};
\node[] at (\hhh*3.5,-\vvv*0.5) {$1$};
\node[] at (\hhh*4.5,-\vvv*0.5) {};
\node[] at (\hhh*3.5,-\vvv*1.5) {};
\node[] at (\hhh*4.5,-\vvv*1.5) {};
\node[] at (\hhh*3.3+\hhhhh*0.5,-\vvv*0.4) {and};
\node[] at (\hhh*0.8+\hhhhh,-\vvv*0.4) {$B=$};
\draw[-] (\hhh*3+\hhhhh,0) rectangle (\hhh*4+\hhhhh,\vvv*1);
\draw[-] (\hhh*4+\hhhhh,0) rectangle (\hhh*5+\hhhhh,\vvv*1);
\draw[-] (\hhh*5+\hhhhh,0) rectangle (\hhh*6+\hhhhh,\vvv*1);
\draw[-] (\hhh*6+\hhhhh,0) rectangle (\hhh*7+\hhhhh,\vvv*1);
\draw[-] (\hhh*2+\hhhhh,-\vvv*1) rectangle (\hhh*3+\hhhhh,\vvv*0);
\draw[-] (\hhh*3+\hhhhh,-\vvv*1) rectangle (\hhh*4+\hhhhh,\vvv*0);
\draw[-] (\hhh*4+\hhhhh,-\vvv*1) rectangle (\hhh*5+\hhhhh,\vvv*0);
\draw[-] (\hhh*3+\hhhhh,-\vvv*2) rectangle (\hhh*4+\hhhhh,-\vvv*1);
\draw[-] (\hhh*4+\hhhhh,-\vvv*2) rectangle (\hhh*5+\hhhhh,-\vvv*1);
\node[] at (\hhh*3.5+\hhhhh,\vvv*0.5) {};
\node[] at (\hhh*4.5+\hhhhh,\vvv*0.5) {};
\node[] at (\hhh*5.5+\hhhhh,\vvv*0.5) {$1'$};
\node[] at (\hhh*6.5+\hhhhh,\vvv*0.5) {$1$};
\node[] at (\hhh*2.5+\hhhhh,-\vvv*0.5) {};
\node[] at (\hhh*3.5+\hhhhh,-\vvv*0.5) {};
\node[] at (\hhh*4.5+\hhhhh,-\vvv*0.5) {$1'$};
\node[] at (\hhh*3.5+\hhhhh,-\vvv*1.5) {$1'$};
\node[] at (\hhh*4.5+\hhhhh,-\vvv*1.5) {$1$};
\end{tikzpicture}.
\end{center}
Then ${}^AB \cup A_B$ can be obtained from $A \cup B$ as follows:
\vskip 2mm
\begin{center}
\begin{tikzpicture}
\def \hhh{4mm}    
\def \vvv{4mm}    
\def \aaa{2}      
\def \bbb{-1}      
\draw[fill=black!30] (\hhh*3,0) rectangle (\hhh*4,\vvv*1);
\draw[fill=black!30] (\hhh*4,0) rectangle (\hhh*5,\vvv*1);
\draw[-] (\hhh*5,0) rectangle (\hhh*6,\vvv*1);
\draw[-] (\hhh*6,0) rectangle (\hhh*7,\vvv*1);
\draw[fill=black!30] (\hhh*2,-\vvv*1) rectangle (\hhh*3,\vvv*0);
\draw[fill=black!30] (\hhh*3,-\vvv*1) rectangle (\hhh*4,\vvv*0);
\draw[-] (\hhh*4,-\vvv*1) rectangle (\hhh*5,\vvv*0);
\draw[-] (\hhh*3,-\vvv*2) rectangle (\hhh*4,-\vvv*1);
\draw[-] (\hhh*4,-\vvv*2) rectangle (\hhh*5,-\vvv*1);
\node[] at (\hhh*3.5,\vvv*0.5) {$1'$};
\node[] at (\hhh*4.5,\vvv*0.5) {$1$};
\node[] at (\hhh*5.5,\vvv*0.5) {$1'$};
\node[] at (\hhh*6.5,\vvv*0.5) {$1$};
\node[] at (\hhh*2.5,-\vvv*0.5) {$1$};
\node[] at (\hhh*3.5,-\vvv*0.5) {$1$};
\node[] at (\hhh*4.5,-\vvv*0.5) {$1'$};
\node[] at (\hhh*3.5,-\vvv*1.5) {$1'$};
\node[] at (\hhh*4.5,-\vvv*1.5) {$1$};
\draw[-,thick,red] (\hhh*\aaa+\hhh*2,-\vvv*\bbb) to (\hhh*\aaa+\hhh*4,-\vvv*\bbb)
to (\hhh*\aaa+\hhh*4,-\vvv*\bbb-\vvv) to (\hhh*\aaa+\hhh*3,-\vvv*\bbb-\vvv)
to (\hhh*\aaa+\hhh*3,-\vvv*\bbb-\vvv*2) to (\hhh*\aaa+\hhh*2,-\vvv*\bbb-\vvv*2)
to (\hhh*\aaa+\hhh*2,-\vvv*\bbb-\vvv*0);
\end{tikzpicture}
\hskip 2mm
\begin{tikzpicture}
\def \hhh{4mm}    
\def \vvv{4mm}    
\def \hhhhh{30mm}  
\def \aaa{3}      
\def \bbb{-1}      
\draw[|->] (\hhh*0,-\vvv*0.5) -- (\hhh*1,-\vvv*0.5) node[midway,above] {\tiny (S5)};
\draw[fill=black!30] (\hhh*3,0) rectangle (\hhh*4,\vvv*1);
\draw[-] (\hhh*4,0) rectangle (\hhh*5,\vvv*1);
\draw[fill=black!30] (\hhh*5,0) rectangle (\hhh*6,\vvv*1);
\draw[-] (\hhh*6,0) rectangle (\hhh*7,\vvv*1);
\draw[fill=black!30] (\hhh*2,-\vvv*1) rectangle (\hhh*3,\vvv*0);
\draw[fill=black!30] (\hhh*3,-\vvv*1) rectangle (\hhh*4,\vvv*0);
\draw[-] (\hhh*4,-\vvv*1) rectangle (\hhh*5,\vvv*0);
\draw[-] (\hhh*3,-\vvv*2) rectangle (\hhh*4,-\vvv*1);
\draw[-] (\hhh*4,-\vvv*2) rectangle (\hhh*5,-\vvv*1);
\node[] at (\hhh*3.5,\vvv*0.5) {$1'$};
\node[] at (\hhh*4.5,\vvv*0.5) {$1'$};
\node[] at (\hhh*5.5,\vvv*0.5) {$1$};
\node[] at (\hhh*6.5,\vvv*0.5) {$1$};
\node[] at (\hhh*2.5,-\vvv*0.5) {$1$};
\node[] at (\hhh*3.5,-\vvv*0.5) {$1$};
\node[] at (\hhh*4.5,-\vvv*0.5) {$1'$};
\node[] at (\hhh*3.5,-\vvv*1.5) {$1'$};
\node[] at (\hhh*4.5,-\vvv*1.5) {$1$};
\draw[-,thick,red] (\hhh*\aaa+\hhh*2,-\vvv*\bbb)
to (\hhh*\aaa+\hhh*4,-\vvv*\bbb)
to (\hhh*\aaa+\hhh*4,-\vvv*\bbb-\vvv)
to (\hhh*\aaa+\hhh*2,-\vvv*\bbb-\vvv)
to (\hhh*\aaa+\hhh*2,-\vvv*\bbb-\vvv*0);
\end{tikzpicture}
\hskip 2mm
\begin{tikzpicture}
\def \hhh{4mm}    
\def \vvv{4mm}    
\def \hhhhh{30mm}  
\def \aaa{1}      
\def \bbb{0}      
\draw[|->] (\hhh*0,-\vvv*0.5) -- (\hhh*1,-\vvv*0.5) node[midway,above] {\tiny (S1)};
\draw[fill=black!30] (\hhh*3,0) rectangle (\hhh*4,\vvv*1);
\draw[-] (\hhh*4,0) rectangle (\hhh*5,\vvv*1);
\draw[-] (\hhh*5,0) rectangle (\hhh*6,\vvv*1);
\draw[fill=black!30] (\hhh*6,0) rectangle (\hhh*7,\vvv*1);
\draw[fill=black!30] (\hhh*2,-\vvv*1) rectangle (\hhh*3,\vvv*0);
\draw[fill=black!30] (\hhh*3,-\vvv*1) rectangle (\hhh*4,\vvv*0);
\draw[-] (\hhh*4,-\vvv*1) rectangle (\hhh*5,\vvv*0);
\draw[-] (\hhh*3,-\vvv*2) rectangle (\hhh*4,-\vvv*1);
\draw[-] (\hhh*4,-\vvv*2) rectangle (\hhh*5,-\vvv*1);
\node[] at (\hhh*3.5,\vvv*0.5) {$1'$};
\node[] at (\hhh*4.5,\vvv*0.5) {$1'$};
\node[] at (\hhh*5.5,\vvv*0.5) {$1$};
\node[] at (\hhh*6.5,\vvv*0.5) {$1$};
\node[] at (\hhh*2.5,-\vvv*0.5) {$1$};
\node[] at (\hhh*3.5,-\vvv*0.5) {$1$};
\node[] at (\hhh*4.5,-\vvv*0.5) {$1'$};
\node[] at (\hhh*3.5,-\vvv*1.5) {$1'$};
\node[] at (\hhh*4.5,-\vvv*1.5) {$1$};
\draw[-,thick,red] (\hhh*\aaa+\hhh*2,-\vvv*\bbb)
to (\hhh*\aaa+\hhh*4,-\vvv*\bbb)
to (\hhh*\aaa+\hhh*4,-\vvv*\bbb-\vvv)
to (\hhh*\aaa+\hhh*3,-\vvv*\bbb-\vvv)
to (\hhh*\aaa+\hhh*3,-\vvv*\bbb-\vvv*2)
to (\hhh*\aaa+\hhh*2,-\vvv*\bbb-\vvv*2)
to (\hhh*\aaa+\hhh*2,-\vvv*\bbb-\vvv*0);
\end{tikzpicture}
\hskip 2mm
\begin{tikzpicture}
\def \hhh{4mm}    
\def \vvv{4mm}    
\def \hhhhh{30mm}  
\def \aaa{2}      
\def \bbb{0}      
\draw[|->] (\hhh*0,-\vvv*0.5) -- (\hhh*1,-\vvv*0.5) node[midway,above] {\tiny (S5)};
\draw[fill=black!30] (\hhh*3,0) rectangle (\hhh*4,\vvv*1);
\draw[-] (\hhh*4,0) rectangle (\hhh*5,\vvv*1);
\draw[-] (\hhh*5,0) rectangle (\hhh*6,\vvv*1);
\draw[fill=black!30] (\hhh*6,0) rectangle (\hhh*7,\vvv*1);
\draw[fill=black!30] (\hhh*2,-\vvv*1) rectangle (\hhh*3,\vvv*0);
\draw[-] (\hhh*3,-\vvv*1) rectangle (\hhh*4,\vvv*0);
\draw[fill=black!30] (\hhh*4,-\vvv*1) rectangle (\hhh*5,\vvv*0);
\draw[-] (\hhh*3,-\vvv*2) rectangle (\hhh*4,-\vvv*1);
\draw[-] (\hhh*4,-\vvv*2) rectangle (\hhh*5,-\vvv*1);
\node[] at (\hhh*3.5,\vvv*0.5) {$1'$};
\node[] at (\hhh*4.5,\vvv*0.5) {$1'$};
\node[] at (\hhh*5.5,\vvv*0.5) {$1$};
\node[] at (\hhh*6.5,\vvv*0.5) {$1$};
\node[] at (\hhh*2.5,-\vvv*0.5) {$1$};
\node[] at (\hhh*3.5,-\vvv*0.5) {$1'$};
\node[] at (\hhh*4.5,-\vvv*0.5) {$1$};
\node[] at (\hhh*3.5,-\vvv*1.5) {$1'$};
\node[] at (\hhh*4.5,-\vvv*1.5) {$1$};
\draw[-,thick,red] (\hhh*\aaa+\hhh*2,-\vvv*\bbb)
to (\hhh*\aaa+\hhh*3,-\vvv*\bbb)
to (\hhh*\aaa+\hhh*3,-\vvv*\bbb-\vvv*2)
to (\hhh*\aaa+\hhh*2,-\vvv*\bbb-\vvv*2)
to (\hhh*\aaa+\hhh*2,-\vvv*\bbb-\vvv*0);
\end{tikzpicture}
\vskip 2mm
\begin{tikzpicture}
\def \hhh{4mm}    
\def \vvv{4mm}    
\def \hhhhh{30mm}  
\def \aaa{0}      
\def \bbb{0}      
\draw[|->] (\hhh*0,-\vvv*0.5) -- (\hhh*1,-\vvv*0.5) node[midway,above] {\tiny (S2)};
\draw[fill=black!30] (\hhh*3,0) rectangle (\hhh*4,\vvv*1);
\draw[-] (\hhh*4,0) rectangle (\hhh*5,\vvv*1);
\draw[-] (\hhh*5,0) rectangle (\hhh*6,\vvv*1);
\draw[fill=black!30] (\hhh*6,0) rectangle (\hhh*7,\vvv*1);
\draw[fill=black!30] (\hhh*2,-\vvv*1) rectangle (\hhh*3,\vvv*0);
\draw[-] (\hhh*3,-\vvv*1) rectangle (\hhh*4,\vvv*0);
\draw[-] (\hhh*4,-\vvv*1) rectangle (\hhh*5,\vvv*0);
\draw[-] (\hhh*3,-\vvv*2) rectangle (\hhh*4,-\vvv*1);
\draw[fill=black!30] (\hhh*4,-\vvv*2) rectangle (\hhh*5,-\vvv*1);
\node[] at (\hhh*3.5,\vvv*0.5) {$1'$};
\node[] at (\hhh*4.5,\vvv*0.5) {$1'$};
\node[] at (\hhh*5.5,\vvv*0.5) {$1$};
\node[] at (\hhh*6.5,\vvv*0.5) {$1$};
\node[] at (\hhh*2.5,-\vvv*0.5) {$1$};
\node[] at (\hhh*3.5,-\vvv*0.5) {$1'$};
\node[] at (\hhh*4.5,-\vvv*0.5) {$1$};
\node[] at (\hhh*3.5,-\vvv*1.5) {$1'$};
\node[] at (\hhh*4.5,-\vvv*1.5) {$1$};
\draw[-,thick,red] (\hhh*\aaa+\hhh*2,-\vvv*\bbb) to (\hhh*\aaa+\hhh*4,-\vvv*\bbb)
to (\hhh*\aaa+\hhh*4,-\vvv*\bbb-\vvv*2) to (\hhh*\aaa+\hhh*3,-\vvv*\bbb-\vvv*2)
to (\hhh*\aaa+\hhh*3,-\vvv*\bbb-\vvv*1) to (\hhh*\aaa+\hhh*2,-\vvv*\bbb-\vvv*1)
to (\hhh*\aaa+\hhh*2,-\vvv*\bbb-\vvv*0);
\end{tikzpicture}
\hskip 2mm
\begin{tikzpicture}
\def \hhh{4mm}    
\def \vvv{4mm}    
\def \hhhhh{30mm}  
\def \aaa{1}      
\def \bbb{-1}      
\draw[|->] (\hhh*0,-\vvv*0.5) -- (\hhh*1,-\vvv*0.5) node[midway,above] {\tiny (S3)};
\draw[fill=black!30] (\hhh*3,0) rectangle (\hhh*4,\vvv*1);
\draw[-] (\hhh*4,0) rectangle (\hhh*5,\vvv*1);
\draw[-] (\hhh*5,0) rectangle (\hhh*6,\vvv*1);
\draw[fill=black!30] (\hhh*6,0) rectangle (\hhh*7,\vvv*1);
\draw[-] (\hhh*2,-\vvv*1) rectangle (\hhh*3,\vvv*0);
\draw[-] (\hhh*3,-\vvv*1) rectangle (\hhh*4,\vvv*0);
\draw[-] (\hhh*4,-\vvv*1) rectangle (\hhh*5,\vvv*0);
\draw[fill=black!30] (\hhh*3,-\vvv*2) rectangle (\hhh*4,-\vvv*1);
\draw[fill=black!30] (\hhh*4,-\vvv*2) rectangle (\hhh*5,-\vvv*1);
\node[] at (\hhh*3.5,\vvv*0.5) {$1'$};
\node[] at (\hhh*4.5,\vvv*0.5) {$1'$};
\node[] at (\hhh*5.5,\vvv*0.5) {$1$};
\node[] at (\hhh*6.5,\vvv*0.5) {$1$};
\node[] at (\hhh*2.5,-\vvv*0.5) {$1'$};
\node[] at (\hhh*3.5,-\vvv*0.5) {$1$};
\node[] at (\hhh*4.5,-\vvv*0.5) {$1$};
\node[] at (\hhh*3.5,-\vvv*1.5) {$1$};
\node[] at (\hhh*4.5,-\vvv*1.5) {$1$};
\draw[-,thick,red] (\hhh*\aaa+\hhh*2,-\vvv*\bbb) to (\hhh*\aaa+\hhh*4,-\vvv*\bbb)
to (\hhh*\aaa+\hhh*4,-\vvv*\bbb-\vvv) to (\hhh*\aaa+\hhh*3,-\vvv*\bbb-\vvv)
to (\hhh*\aaa+\hhh*3,-\vvv*\bbb-\vvv*2) to (\hhh*\aaa+\hhh*2,-\vvv*\bbb-\vvv*2)
to (\hhh*\aaa+\hhh*2,-\vvv*\bbb-\vvv*0);
\end{tikzpicture}
\hskip 2mm
\begin{tikzpicture}
\def \hhh{4mm}    
\def \vvv{4mm}    
\def \hhhhh{30mm}  
\def \aaa{2}      
\def \bbb{-1}      
\draw[|->] (\hhh*0,-\vvv*0.5) -- (\hhh*1,-\vvv*0.5) node[midway,above] {\tiny (S5)};
\draw[-] (\hhh*3,0) rectangle (\hhh*4,\vvv*1);
\draw[fill=black!30] (\hhh*4,0) rectangle (\hhh*5,\vvv*1);
\draw[-] (\hhh*5,0) rectangle (\hhh*6,\vvv*1);
\draw[fill=black!30] (\hhh*6,0) rectangle (\hhh*7,\vvv*1);
\draw[-] (\hhh*2,-\vvv*1) rectangle (\hhh*3,\vvv*0);
\draw[-] (\hhh*3,-\vvv*1) rectangle (\hhh*4,\vvv*0);
\draw[-] (\hhh*4,-\vvv*1) rectangle (\hhh*5,\vvv*0);
\draw[fill=black!30] (\hhh*3,-\vvv*2) rectangle (\hhh*4,-\vvv*1);
\draw[fill=black!30] (\hhh*4,-\vvv*2) rectangle (\hhh*5,-\vvv*1);
\node[] at (\hhh*3.5,\vvv*0.5) {$1'$};
\node[] at (\hhh*4.5,\vvv*0.5) {$1'$};
\node[] at (\hhh*5.5,\vvv*0.5) {$1$};
\node[] at (\hhh*6.5,\vvv*0.5) {$1$};
\node[] at (\hhh*2.5,-\vvv*0.5) {$1'$};
\node[] at (\hhh*3.5,-\vvv*0.5) {$1$};
\node[] at (\hhh*4.5,-\vvv*0.5) {$1$};
\node[] at (\hhh*3.5,-\vvv*1.5) {$1$};
\node[] at (\hhh*4.5,-\vvv*1.5) {$1$};
\draw[-,thick,red] (\hhh*\aaa+\hhh*2,-\vvv*\bbb) to (\hhh*\aaa+\hhh*4,-\vvv*\bbb)
to (\hhh*\aaa+\hhh*4,-\vvv*\bbb-\vvv) to (\hhh*\aaa+\hhh*3,-\vvv*\bbb-\vvv)
to (\hhh*\aaa+\hhh*3,-\vvv*\bbb-\vvv*2) to (\hhh*\aaa+\hhh*2,-\vvv*\bbb-\vvv*2)
to (\hhh*\aaa+\hhh*2,-\vvv*\bbb-\vvv*0);
\end{tikzpicture}
\hskip 2mm
\begin{tikzpicture}
\def \hhh{4mm}    
\def \vvv{4mm}    
\def \hhhhh{0mm}  
\draw[|->] (\hhh*0,-\vvv*0.5) -- (\hhh*1,-\vvv*0.5) node[midway,above] {\tiny (S6)};
\draw[-] (\hhh*3+\hhhhh,0) rectangle (\hhh*4+\hhhhh,\vvv*1);
\draw[-] (\hhh*4+\hhhhh,0) rectangle (\hhh*5+\hhhhh,\vvv*1);
\draw[-] (\hhh*5+\hhhhh,0) rectangle (\hhh*6+\hhhhh,\vvv*1);
\draw[fill=black!30] (\hhh*6+\hhhhh,0) rectangle (\hhh*7+\hhhhh,\vvv*1);
\draw[-] (\hhh*2+\hhhhh,-\vvv*1) rectangle (\hhh*3+\hhhhh,\vvv*0);
\draw[-] (\hhh*3+\hhhhh,-\vvv*1) rectangle (\hhh*4+\hhhhh,\vvv*0);
\draw[fill=black!30] (\hhh*4+\hhhhh,-\vvv*1) rectangle (\hhh*5+\hhhhh,\vvv*0);
\draw[fill=black!30] (\hhh*3+\hhhhh,-\vvv*2) rectangle (\hhh*4+\hhhhh,-\vvv*1);
\draw[fill=black!30] (\hhh*4+\hhhhh,-\vvv*2) rectangle (\hhh*5+\hhhhh,-\vvv*1);
\node[] at (\hhh*3.5+\hhhhh,\vvv*0.5) {$1'$};
\node[] at (\hhh*4.5+\hhhhh,\vvv*0.5) {$1$};
\node[] at (\hhh*5.5+\hhhhh,\vvv*0.5) {$1$};
\node[] at (\hhh*6.5+\hhhhh,\vvv*0.5) {$1$};
\node[] at (\hhh*2.5+\hhhhh,-\vvv*0.5) {$1'$};
\node[] at (\hhh*3.5+\hhhhh,-\vvv*0.5) {$1$};
\node[] at (\hhh*4.5+\hhhhh,-\vvv*0.5) {$1'$};
\node[] at (\hhh*3.5+\hhhhh,-\vvv*1.5) {$1$};
\node[] at (\hhh*4.5+\hhhhh,-\vvv*1.5) {$1$};
\end{tikzpicture}
\end{center}
\end{example}

\vskip 1mm
\begin{rem}\label{Remark-Diff}
(i)
Algorithm \ref{algorithm1} heavily depends on the order in which we choose $\bfa$-boxes.
For instance, in the below, if the switch (S6) is applied to the $a'$-box in the first row, then we have
\begin{center}
\begin{tikzpicture}
\def \hhh{5mm}    
\def \vvv{5mm}    
\def \hhhhh{35mm}  
\node[] at (-\hhh*2,-\vvv*0.4) {$A \cup B =$};
\draw[-,black!20] (\hhh*0+\hhhhh*0,0) rectangle (\hhh*1+\hhhhh*0,\vvv*1);
\draw[-,black!20] (\hhh*1+\hhhhh*0,0) rectangle (\hhh*2+\hhhhh*0,\vvv*1);
\draw[fill=black!30] (\hhh*2+\hhhhh*0,0) rectangle (\hhh*3+\hhhhh*0,\vvv*1);
\draw[-] (\hhh*3+\hhhhh*0,0) rectangle (\hhh*4+\hhhhh*0,\vvv*1);
\draw[fill=black!30] (\hhh*1+\hhhhh*0,-\vvv*1) rectangle (\hhh*2+\hhhhh*0,\vvv*0);
\draw[-] (\hhh*2+\hhhhh*0,-\vvv*1) rectangle (\hhh*3+\hhhhh*0,\vvv*0);
\draw[-] (\hhh*2+\hhhhh*0,-\vvv*2) rectangle (\hhh*3+\hhhhh*0,-\vvv*1);
\draw[-,thick,red] (\hhh*2,\vvv*1) -- (\hhh*4,\vvv*1) -- (\hhh*4,\vvv*0)
-- (\hhh*3,\vvv*0) -- (\hhh*3,-\vvv*1) -- (\hhh*2,-\vvv*1) -- (\hhh*2,\vvv*1);
\node[] at (\hhh*2.5+\hhhhh*0,\vvv*0.5) {$a'$};
\node[] at (\hhh*3.5+\hhhhh*0,\vvv*0.5) {$b$};
\node[] at (\hhh*1.5+\hhhhh*0,-\vvv*0.5) {$a$};
\node[] at (\hhh*2.5+\hhhhh*0,-\vvv*0.45) {$b'$};
\node[] at (\hhh*2.5+\hhhhh*0,-\vvv*1.5) {$b$};
\draw[|->] (\hhh*4.7,-\vvv*0.5) -- (-\hhh*0.7+\hhhhh*1,-\vvv*0.5)
node[midway,above] {\small (S6)};
\draw[-,black!20] (\hhh*0+\hhhhh*1,0) rectangle (\hhh*1+\hhhhh*1,\vvv*1);
\draw[-,black!20] (\hhh*1+\hhhhh*1,0) rectangle (\hhh*2+\hhhhh*1,\vvv*1);
\draw[-] (\hhh*2+\hhhhh*1,0) rectangle (\hhh*3+\hhhhh*1,\vvv*1);
\draw[-] (\hhh*3+\hhhhh*1,0) rectangle (\hhh*4+\hhhhh*1,\vvv*1);
\draw[fill=black!30] (\hhh*1+\hhhhh*1,-\vvv*1) rectangle (\hhh*2+\hhhhh*1,\vvv*0);
\draw[-,fill=black!30] (\hhh*2+\hhhhh*1,-\vvv*1) rectangle (\hhh*3+\hhhhh*1,\vvv*0);
\draw[-] (\hhh*2+\hhhhh*1,-\vvv*2) rectangle (\hhh*3+\hhhhh*1,-\vvv*1);
\node[] at (\hhh*2.5+\hhhhh*1,\vvv*0.5) {$b'$};
\node[] at (\hhh*3.5+\hhhhh*1,\vvv*0.5) {$b$};
\node[] at (\hhh*1.5+\hhhhh*1,-\vvv*0.5) {$a$};
\node[] at (\hhh*2.5+\hhhhh*1,-\vvv*0.4) {$a'$};
\node[] at (\hhh*2.5+\hhhhh*1,-\vvv*1.5) {$b$};
\end{tikzpicture}.
\end{center}
Notice that the resulting filling is not a shifted perforated $(\bfa,\bfb)$-pair.

\label{rem-stan-switch}
(ii) There is a slight difference between the slides in \cite{BKT} and
our switches.
For the clarity, denote the slides in \cite{BKT} by
\begin{tikzpicture}
\def \hhh{5mm}    
\draw[|->,dotted] (\hhh*3,0) to (\hhh*4,0);
\node at (\hhh*3,0.3mm) () {};
\end{tikzpicture}.
Let us consider the following shifted perforated $(\bfa,\bfb)$-pairs:
\begin{center}
\begin{tikzpicture}
\def \hhh{5mm}    
\def \vvv{5mm}    
\node[] at (-\hhh*1,\vvv*0) () {$U=$};
\draw[fill=black!30] (0,0) rectangle (\hhh*1,\vvv*1);
\draw[fill=black!30] (\hhh*1,0) rectangle (\hhh*2,\vvv*1);
\draw[-] (\hhh*1,-\vvv*1) rectangle (\hhh*2,0);
\node[] at (\hhh*0.5,\vvv*0.6) () {$a'$};
\node[] at (\hhh*1.5,\vvv*0.5) () {$a$};
\node[] at (\hhh*1.5,-\vvv*0.5) () {$\bfb$};
\end{tikzpicture}
\hskip 10mm
\begin{tikzpicture}
\def \hhh{5mm}    
\def \vvv{5mm}    
\node[] at (-\hhh*0,\vvv*0) () {$V=$};
\draw[fill=black!30] (\hhh*1,0) rectangle (\hhh*2,\vvv*1);
\draw[fill=black!30] (\hhh*2,0) rectangle (\hhh*3,\vvv*1);
\draw[-] (\hhh*3,0) rectangle (\hhh*4,\vvv*1);
\draw[-] (\hhh*2,-\vvv*1) rectangle (\hhh*3,\vvv*0);
\node[] at (\hhh*1.5,\vvv*0.6) {$a'$};
\node[] at (\hhh*2.5,\vvv*0.5) {$a$};
\node[] at (\hhh*3.5,\vvv*0.5) {$b$};
\node[] at (\hhh*2.5,-\vvv*0.5) {$\bfb$};
\end{tikzpicture}
\end{center}
Under the identification of ${\rm o}$-holes with $a$-boxes and $a'$-boxes with ${\rm o'}$-holes,
these pairs are transformed by the sliding process given in \cite{BKT} as follows:
\begin{center}
\begin{tikzpicture}
\def \hhh{5mm}    
\def \vvv{5mm}    
\node[] at (-\hhh*1,\vvv*0) () {$U'=$};
\draw[-] (0,0) rectangle (\hhh*1,\vvv*1);
\draw[fill=black!30] (\hhh*1,0) rectangle (\hhh*2,\vvv*1);
\draw[fill=black!30] (\hhh*1,-\vvv*1) rectangle (\hhh*2,0);
\node[] at (\hhh*0.5,\vvv*0.5) () {$\bfb$};
\node[] at (\hhh*1.5,\vvv*0.6) () {$a'$};
\node[] at (\hhh*1.5,-\vvv*0.5) () {$a$};
\end{tikzpicture}
\hskip 10mm
\begin{tikzpicture}
\def \hhh{5mm}    
\def \vvv{5mm}    
\node[] at (-\hhh*0,\vvv*0) () {$V'=$};
\draw[-] (\hhh*1,0) rectangle (\hhh*2,\vvv*1);
\draw[-] (\hhh*2,0) rectangle (\hhh*3,\vvv*1);
\draw[fill=black!30] (\hhh*3,0) rectangle (\hhh*4,\vvv*1);
\draw[fill=black!30] (\hhh*2,-\vvv*1) rectangle (\hhh*3,\vvv*0);
\node[] at (\hhh*1.5,\vvv*0.5) {$\bfb$};
\node[] at (\hhh*2.5,\vvv*0.5) {$b$};
\node[] at (\hhh*3.5,\vvv*0.6) {$a'$};
\node[] at (\hhh*2.5,-\vvv*0.5) {$a$};
\end{tikzpicture}
\end{center}
On the contrary,
$U$ and $V$ are transformed by our switching process in the following form:
\begin{center}
\begin{tikzpicture}
\def \hhh{5mm}    
\def \vvv{5mm}    
\node[] at (-\hhh*1,\vvv*0) () {$U''=$};
\draw[-] (0,0) rectangle (\hhh*1,\vvv*1);
\draw[fill=black!30] (\hhh*1,0) rectangle (\hhh*2,\vvv*1);
\draw[fill=black!30] (\hhh*1,-\vvv*1) rectangle (\hhh*2,0);
\node[] at (\hhh*0.5,\vvv*0.5) () {$\bfb$};
\node[] at (\hhh*1.5,\vvv*0.5) () {$a'$};
\node[] at (\hhh*1.5,-\vvv*0.5) () {$a'$};
\end{tikzpicture}
\hskip 10mm
\begin{tikzpicture}
\def \hhh{5mm}    
\def \vvv{5mm}    
\node[] at (-\hhh*0,\vvv*0) () {$V''=$};
\draw[-] (\hhh*1,0) rectangle (\hhh*2,\vvv*1);
\draw[-] (\hhh*2,0) rectangle (\hhh*3,\vvv*1);
\draw[fill=black!30] (\hhh*3,0) rectangle (\hhh*4,\vvv*1);
\draw[fill=black!30] (\hhh*2,-\vvv*1) rectangle (\hhh*3,\vvv*0);
\node[] at (\hhh*1.5,\vvv*0.5) {$\bfb$};
\node[] at (\hhh*2.5,\vvv*0.5) {$b$};
\node[] at (\hhh*3.5,\vvv*0.5) {$a'$};
\node[] at (\hhh*2.5,-\vvv*0.5) {$a'$};
\end{tikzpicture}
\end{center}
It should be pointed out that
$U$ and $V$ cannot be recovered from $U'$ and $V'$, respectively,
by applying
\begin{tikzpicture}
\def \hhh{5mm}    
\draw[|->,dotted] (\hhh*3,0) to (\hhh*4,0);
\node at (\hhh*3,0.3mm) () {};
\end{tikzpicture}
since
\begin{center}
\begin{tikzpicture}
\def \hhh{5mm}    
\def \vvv{5mm}    
\draw[-] (0,0) rectangle (\hhh*1,\vvv*1);
\draw[fill=black!30] (\hhh*1,0) rectangle (\hhh*2,\vvv*1);
\draw[fill=black!30] (\hhh*1,-\vvv*1) rectangle (\hhh*2,0);
\node[] at (\hhh*0.5,\vvv*0.5) () {$\bfb$};
\node[] at (\hhh*1.5,\vvv*0.5) () {$a'$};
\node[] at (\hhh*1.5,-\vvv*0.5) () {$a$};
\draw[|->,dotted] (\hhh*3,0) to (\hhh*4,0);
\draw[fill=black!30] (\hhh*5,0) rectangle (\hhh*6,\vvv*1);
\draw[fill=black!30] (\hhh*6,0) rectangle (\hhh*7,\vvv*1);
\draw[-] (\hhh*6,-\vvv*1) rectangle (\hhh*7,0);
\node[] at (\hhh*5.5,\vvv*0.5) () {$a$};
\node[] at (\hhh*6.5,\vvv*0.5) () {$a$};
\node[] at (\hhh*6.5,-\vvv*0.5) () {$\bfb$};
\node[] at (\hhh*9,-\vvv*0) () {and};
\draw[-] (\hhh*11,0) rectangle (\hhh*12,\vvv*1);
\draw[-] (\hhh*12,0) rectangle (\hhh*13,\vvv*1);
\draw[fill=black!30] (\hhh*13,0) rectangle (\hhh*14,\vvv*1);
\draw[-,fill=black!30] (\hhh*12,-\vvv*1) rectangle (\hhh*13,\vvv*0);
\node[] at (\hhh*11.5,\vvv*0.5) {$\bfb$};
\node[] at (\hhh*12.5,\vvv*0.5) {$b$};
\node[] at (\hhh*13.5,\vvv*0.5) {$a'$};
\node[] at (\hhh*12.5,-\vvv*0.5) {$a$};
\draw[|->,dotted] (\hhh*15,0) to (\hhh*16,0);
\draw[-] (\hhh*17,0) rectangle (\hhh*18,\vvv*1);
\draw[fill=black!30] (\hhh*18,0) rectangle (\hhh*19,\vvv*1);
\draw[-] (\hhh*19,0) rectangle (\hhh*20,\vvv*1);
\draw[-,fill=black!30] (\hhh*18,-\vvv*1) rectangle (\hhh*19,\vvv*0);
\node[] at (\hhh*17.5,\vvv*0.5) {$\bfb$};
\node[] at (\hhh*18.5,\vvv*0.5) {$a'$};
\node[] at (\hhh*19.5,\vvv*0.5) {$b$};
\node[] at (\hhh*18.5,-\vvv*0.5) {$a$};
\draw[|->,dotted] (\hhh*21,0) to (\hhh*22,0);
\draw[fill=black!30] (\hhh*23,0) rectangle (\hhh*24,\vvv*1);
\draw[fill=black!30] (\hhh*24,0) rectangle (\hhh*25,\vvv*1);
\draw[-] (\hhh*25,0) rectangle (\hhh*26,\vvv*1);
\draw[-] (\hhh*24,-\vvv*1) rectangle (\hhh*25,\vvv*0);
\node[] at (\hhh*23.5,\vvv*0.5) {$a$};
\node[] at (\hhh*24.5,\vvv*0.5) {$a$};
\node[] at (\hhh*25.5,\vvv*0.5) {$b$};
\node[] at (\hhh*24.5,-\vvv*0.5) {$\bfb$};
\end{tikzpicture}.
\end{center}
But, by applying our switches,
they can be recovered from $U''$ and $V''$, respectively.
In fact we will prove in Section \ref{sect-prop-ss} that our shifted tableau switching is an involution
(see Theorem \ref{thm-main-section4}).

\vskip 1mm
\noindent (iii)
Given a shifted perforated $(\bfa,\bfb)$-pair $(A,B)$ such that $B$ extends $A$,
let us identify $a'$-boxes and $a$-boxes with ${\rm o}'$-holes and ${\rm o}$-holes, respectively.
Let $(\overline{{}^AB},\overline{A_B})$ be the shifted perforated $(\bfb,\bfa)$-pair
obtained through the sliding process given in \cite{BKT}.
In case where neither $U$ nor $V$ presented in (ii) does not occur in the middle of the sliding process,
one can easily verify that $({}^AB,A_B)$ is exactly equal to $(\overline{{}^AB},\overline{A_B})$.
However, if
either $U$ or $V$ occurs, then
$A_B$ is different from $\overline{A_B}$ although ${}^AB$ is still equal to $\overline{{}^AB}$.
The difference between $A_B$ and $\overline{A_B}$
lies in that the former has a primed letter
in the southernmost box on the main diagonal,
whereas the latter has an unprimed letter therein.
It implies that $\stan(A_B) = \stan(\overline{A_B})$.

\vskip 1mm
\noindent (iv)
The definition of the sliding process presented in~\cite{BKT} tells us that
$\overline{{}^AB}=\overline{{}^{\stan(A)}B}$ and $\stan(A)_B=\overline{\stan(A)_B}$.
From this it follows that ${}^AB={}^{\stan(A)}B$ and $\stan(A_B)=\stan(A)_B$.
Similarly, one can observe that the standardization of $\bfb$ letters is compatible
with our switches (S1) through (S7). The below example shows that (S7) followed by
the standardization of $\bfb$ letters and the standardization of $\bfb$ letters followed by
(S4) produce the same result.
\vskip 2mm
\begin{center}
\begin{tikzpicture}[baseline=0mm]
\def \hhh{5mm}     
\def \vvv{5mm}       
\def \vvvv{-20mm}   
\def \hhhh{40mm}  
\draw[-,fill=black!20] (\hhh*0,\vvv*0) rectangle (\hhh*1,\vvv*1);
\draw[-,fill=black!20] (\hhh*1,\vvv*0) rectangle (\hhh*2,\vvv*1);
\draw[-] (\hhh*2,\vvv*0) rectangle (\hhh*3,\vvv*1);
\draw[-] (\hhh*1,-\vvv*1) rectangle (\hhh*2,\vvv*0);
\node at (\hhh*0.5,\vvv*0.5) {$\bfa$};
\node at (\hhh*1.5,\vvv*0.5) {$a$};
\node at (\hhh*2.5,\vvv*0.5) {$b$};
\node at (\hhh*1.5,-\vvv*0.5) {$\bfb$};
\draw[|->] (\hhh*4,\vvv*0) -- (-\hhh*1+\hhhh,\vvv*0) node[above,midway] {\tiny (S7)};
\draw[-] (\hhhh+\hhh*0,\vvv*0) rectangle (\hhhh+\hhh*1,\vvv*1);
\draw[-,fill=black!20] (\hhhh+\hhh*1,\vvv*0) rectangle (\hhhh+\hhh*2,\vvv*1);
\draw[-] (\hhhh+\hhh*2,\vvv*0) rectangle (\hhhh+\hhh*3,\vvv*1);
\draw[-,fill=black!20] (\hhhh+\hhh*1,-\vvv*1) rectangle (\hhhh+\hhh*2,\vvv*0);
\node at (\hhhh+\hhh*0.5,\vvv*0.5) {$\bfb$};
\node at (\hhhh+\hhh*1.5,\vvv*0.5) {$a'$};
\node at (\hhhh+\hhh*2.5,\vvv*0.5) {$b$};
\node at (\hhhh+\hhh*1.5,-\vvv*0.5) {$\bfa$};
\draw[-,fill=black!20] (\hhh*0,\vvvv+\vvv*0) rectangle (\hhh*1,\vvvv+\vvv*1);
\draw[-,fill=black!20] (\hhh*1,\vvvv+\vvv*0) rectangle (\hhh*2,\vvvv+\vvv*1);
\draw[-] (\hhh*2,\vvvv+\vvv*0) rectangle (\hhh*3,\vvvv+\vvv*1);
\draw[-] (\hhh*1,\vvvv-\vvv*1) rectangle (\hhh*2,\vvvv+\vvv*0);
\node at (\hhh*0.5,\vvvv+\vvv*0.5) {$\bfa$};
\node at (\hhh*1.5,\vvvv+\vvv*0.5) {$a$};
\node at (\hhh*2.5,\vvvv+\vvv*0.5) {$2$};
\node at (\hhh*1.5,\vvvv-\vvv*0.5) {$1$};
\draw[|->] (\hhh*4,\vvvv+\vvv*0) -- (-\hhh*1+\hhhh,\vvvv+\vvv*0) node[above,midway] {\tiny (S4)};
\draw[-] (\hhhh+\hhh*0,\vvvv+\vvv*0) rectangle (\hhhh+\hhh*1,\vvvv+\vvv*1);
\draw[-,fill=black!20] (\hhhh+\hhh*1,\vvvv+\vvv*0) rectangle (\hhhh+\hhh*2,\vvvv+\vvv*1);
\draw[-] (\hhhh+\hhh*2,\vvvv+\vvv*0) rectangle (\hhhh+\hhh*3,\vvvv+\vvv*1);
\draw[-,fill=black!20] (\hhhh+\hhh*1,\vvvv-\vvv*1) rectangle (\hhhh+\hhh*2,\vvvv+\vvv*0);
\node at (\hhhh+\hhh*0.5,\vvvv+\vvv*0.5) {$1$};
\node at (\hhhh+\hhh*1.5,\vvvv+\vvv*0.5) {$a'$};
\node at (\hhhh+\hhh*2.5,\vvvv+\vvv*0.5) {$2$};
\node at (\hhhh+\hhh*1.5,\vvvv-\vvv*0.5) {$\bfa$};
\draw[|->] (\hhh*1.5,-\vvv*1.5) -- (\hhh*1.5,\vvvv+\vvv*1.5) node[left,midway] {\tiny \stan(B)};
\draw[|->] (\hhh*1.5+\hhhh,-\vvv*1.5) -- (\hhh*1.5+\hhhh,\vvvv+\vvv*1.5) node[right,midway] {\tiny \stan(B)};
\end{tikzpicture}
\end{center}
As a consequence, we can derive that $A_B=A_{\stan(B)}$ and $\stan({}^AB)={}^A\stan(B)$.

\vskip 1mm
\noindent (v)
Suppose that $(A',B)$ is a shifted perforated $(\bfa,\bfb)$-pair
such that $B$ extends $A'$.
If $\stan(A) = \stan(A')$,
then $\stan(A)_B = \stan(A')_B$
and therefore $\stan({A}_B)=\stan({A'}_B)$ by (iv).
\end{rem}

Now we are ready to present the shifted switching process.

\begin{algo}\label{algorithm2}
(The shifted switching process on pairs of SYTs) \
Let $(S,T)$ be a pair of SYTs such that $T$ extends $S$.

\begin{description}
\item[Step 1] Let $m = {\rm max}(S)$ and $n ={\rm max}(T)$.

\item[Step 2] Set $i:=m$ and $j:=1$.

\item[Step 3] Apply Algorithm \ref{algorithm1} to $(S^{(i)}, T^{(j)})$.

\item[Step 4] If $j = n$, then go to {\bf Step 5}.
Otherwise, go to {\bf Step 3} after setting $j:=j+1$.

\item[Step 5] If $i > 1$, then go to {\bf Step 3} after setting $j:=1$ and $i:=i-1$.
Otherwise, finish the algorithm.
\end{description}
\end{algo}

Let $(S, T)$ be a pair of SYTs such that $T$ extends $S$.
Suppose that ${\rm max}(S) = m$ and ${\rm max}(T)=n.$
By abuse of notation, we write $S \cup T$ as
$$
S^{(1)} \cdots S^{(m)} T^{(1)} \cdots  T^{(n)}
$$
and the filling obtained by applying $\sigma^k$ to
$S^{(m)} \cup T^{(1)}$ in $S \cup T$
as
$$
S^{(1)} \cdots S^{(m-1)} {\sigma^k}\left(S^{(m)} T^{(1)}\right) T^{(2)} \cdots T^{(n)}.
$$
When $S^{(m)} \cup T^{(1)}$ gets fully switched after applying a sequence of switches,
we write the resulting filling as
$$
S^{(1)} \cdots S^{(m-1)} T^{(1)} S^{(m)} T^{(2)} \cdots T^{(n)}.
$$
In the inductive manner we can write
$$
S^{(1)} \cdots S^{(r-1)}
{\sigma^k}\left( S^{(r)} T^{(1)} \right)
T^{(2)} \cdots T^{(n)} S^{(r+1)} \cdots  S^{(m)}
$$
for the filling obtained from
$S^{(1)} \cdots S^{(r)} T^{(1)} \cdots T^{(n)}  S^{(r+1)}  \cdots  S^{(m)}$
by applying $\sigma^k$ to
$S^{(r)} \cup T^{(1)}$.
When $S^{(1)} \cup T^{(n)}$ gets fully switched after applying a sequence of switches,
the process is finished.
We write $S_T$ and ${}^ST$ for the fillings that
$S$ and $T$ respectively become when the above process is applied to $S \cup T$.
With this notation,
we define the {\it shifted (tableau) switching} $\Sigma$
by
the map sending $(S,T)$ to $({}^ST, S_T)$.
When $S\cup T$ has multiple components,
each of them is switched independently
according to Algorithm \ref{algorithm2}.
The next theorem is the main result of this section.

\begin{thm}\label{Thm-Sec3-SW}
Let $(S, T)$ be a pair of SYTs such that $T$ extends $S$.
The shifted tableau switching $(S ,T) \mapsto ({}^ST, S_T)$ satisfies the following.
\begin{itemize}
\item[{\rm (a)}] ${}^S T \cup S_T$ has the same shape as $S \cup T$.
In addition,
$S_T$ (resp., ${}^ST$) has the same weight as $S$ (resp., $T$).

\item[{\rm (b)}] $S_T$ extends ${}^S T$. In addition, $S_T$ and ${}^S T$ are SYTs.
\end{itemize}
\end{thm}

\begin{proof}
(a) The assertion follows from Theorem \ref{Thm-Perforated}(a) and Algorithm \ref{algorithm2}.

(b) In the same way as in the proof of Theorem \ref{Thm-Perforated}(b),
Remark \ref{remark-switch=jdt}(i) implies that
${}^ST$ is an SYT of skew shape and $S_T$ extends ${}^ST$.
Set $m := \max(S)$ and $n := \max(T)$.
Without loss of generality, we may assume that $S^{(i)} \neq \emptyset$ for all $1\leq i \leq m$.
Now we will prove by induction on $m$ that $S_T$ is an SYT.

When $m=1$, $S^{(1)} (= S)$ is an SYT of border strip shape.
Therefore by applying Theorem \ref{Thm-Perforated}(b) repeatedly,
we can see that $S^{(1)}(=S_T)$ in $T^{(1)}\cdots T^{(n)}S^{(1)}$ is an SYT of border strip shape.
Next we assume that the desired result holds for all SYTs $S$ with $\max(S) \leq k$
which is extended by $T$.
Let $(S',T)$ be a pair of SYTs such that $T$ extends $S'$ and $\max(S')=k+1$.
Applying the same argument as in case where $m=1$, 
one can see that $S'^{(k+1)}$ in $S'^{(1)} \cdots S'^{(k)} T^{(1)} \cdots T^{(n)} S'^{(k+1)}$ is an SYT of border strip shape.
Then our induction hypothesis implies that
$S'^{(1)} \cdots S'^{(k)}$ in $T^{(1)} \cdots T^{(n)} S'^{(1)} \cdots S'^{(k)} S'^{(k+1)}$
is an SYT which is extended by $S'^{(k+1)}$.
Consequently $S'^{(1)} \cdots S'^{(k)} S'^{(k+1)} (= S'_T)$ is an SYT.
\end{proof}

\begin{example}\label{example-section3}
Let
\vskip 1mm
\begin{center}
.
\end{center}
Applying Algorithm \ref{algorithm1} to $(S^{(2)}, T^{(1)})$ in
yields that
\vskip 2mm
\begin{center}
%
.
\end{center}
Similarly, applying Algorithm \ref{algorithm1} to
$(S^{(2)}, T^{(2)})$
in $S^{(1)} T^{(1)} S^{(2)} T^{(2)}$ yields that
\vskip 2mm
\begin{center}
%
.
\end{center}
Next,
applying Algorithm \ref{algorithm1} to
$(S^{(1)}, T^{(1)})$
in $S^{(1)} T^{(1)} T^{(2)} S^{(2)} $ yields that
\vskip 2mm
\begin{center}
%
\end{center}
(see Example \ref{example-sw-perforated-pair}).
Finally,
applying Algorithm \ref{algorithm1} to
$(S^{(1)}, T^{(2)})$
in $T^{(1)} S^{(1)} T^{(2)} S^{(2)}$ yields that
\vskip 2mm
\begin{center}
%
.
\end{center}
This shows that
\begin{center}
%
.
\end{center}
\end{example}

\vskip 2mm
\begin{rem}\label{Remark-Diff2}
(i)
In our shifted switching process, the order in which we choose the boxes to apply switches is very important.
For instance,
given a box,
we may apply switches to that box as many times as possible.
The resulting filling obtained in this way need not be a pair of SYTs.
Refer to the following example:
\vskip 2mm
\begin{center}
%
\end{center}

\noindent
(ii)
When we are given a filling of the form
\begin{displaymath}
\cdots S^{(i)} T^{(j)}  \cdots S^{(k)} T^{(l)} \cdots,
\end{displaymath}
the shifted switching process on $(S^{(i)}, T^{(j)})$ and that on $(S^{(k)}, T^{(l)})$
do not affect on each other, thus the order in which we apply Algorithm \ref{algorithm1} to them
does not matter.
As a consequence, the following algorithm produces the same result as Algorithm \ref{algorithm2}.
\begin{algo}\label{algorithm2-1}
\vskip 1mm
Let $(S,T)$ be a pair of SYTs such that $T$ extends $S$.
\begin{description}
\item[Step 1] Let $m = {\rm max}(S)$ and $n ={\rm max}(T)$.

\item[Step 2] Set $i:=m$ and $j:=1$.

\item[Step 3] Apply Algorithm \ref{algorithm1} to $(S^{(i)}, T^{(j)})$.

\item[Step 4] If $i = 1$, then go to {\bf Step 5}.
Otherwise, go to {\bf Step 3} after setting $i:=i-1$.

\item[Step 5] If $j < n$, then go to {\bf Step 3} after setting $j:=j+1$ and $i:=m$.
Otherwise, finish the algorithm.
\end{description}
\end{algo}
\end{rem}

\subsection{Proof of Proposition \ref{prop-preserving-sp}}
\label{subsection-proof-prop}
We provide two auxiliary lemmas for the proof of Proposition \ref{prop-preserving-sp}.
To do this we need the notion of the shifted switching path.
Suppose that $(A, B)$ is a shifted perforated $(\bfa,\bfb)$-pair such that
$B$ extends $A$.
For each $\bfa$-box in $A \cup B$,
define its {\it shifted switching path} by
the set of boxes it passes through during the shifted switching process.
Particularly when the switch (S3) is applied to a box on the main diagonal, say the $(i,i)$ box, in the middle of the shifted switching process,
the corresponding shifted switching path is understood as $\{ \ldots, (i,i), (i,i+1), (i+1,i+1), \ldots \}$
(for the preciseness, refer to Lemma \ref{lem-sw-path}).
Denote by $\overrightarrow{\rm P}_{(i,j)}$ the shifted switching path starting from the $\bfa$-box positioned at $(i,j)$
and ${\rm End}(\overrightarrow{\rm P}_{(i,j)})$ the endpoint of $\overrightarrow{\rm P}_{(i,j)}$.
For instance, when
\begin{tikzpicture}[baseline=0pt]
\def \hhh{4mm}    
\def \vvv{4mm}    
\node[] at (-\hhh*2,\vvv*0.2) {$A \cup B =$};
\draw[fill=black!30] (\hhh*0,0) rectangle (\hhh*1,\vvv*1);
\draw[fill=black!30] (\hhh*1,0) rectangle (\hhh*2,\vvv*1);
\draw[-] (\hhh*2,0) rectangle (\hhh*3,\vvv*1);
\draw[-] (\hhh*3,0) rectangle (\hhh*4,\vvv*1);
\draw[-] (\hhh*4,0) rectangle (\hhh*5,\vvv*1);
\draw[-] (\hhh*5,0) rectangle (\hhh*6,\vvv*1);
\draw[-] (\hhh*1,-\vvv*1) rectangle (\hhh*2,\vvv*0);
\node[] at (\hhh*0.5,\vvv*0.5) {$1$};
\node[] at (\hhh*1.5,\vvv*0.5) {$1$};
\node[] at (\hhh*2.5,\vvv*0.5) {$1'$};
\node[] at (\hhh*3.5,\vvv*0.5) {$1$};
\node[] at (\hhh*4.5,\vvv*0.5) {$1$};
\node[] at (\hhh*5.5,\vvv*0.5) {$1$};
\node[] at (\hhh*1.5,-\vvv*0.5) {$1'$};
\end{tikzpicture},
we can obtain the fully switched pair
\begin{tikzpicture}[baseline=0mm]
\def \hhh{4mm}    
\def \vvv{4mm}    
\node[] at (-\hhh*3,\vvv*0.2) {$\sigma^5(A \cup B) =$};
\draw[-] (\hhh*0,0) rectangle (\hhh*1,\vvv*1);
\draw[-] (\hhh*1,0) rectangle (\hhh*2,\vvv*1);
\draw[-] (\hhh*2,0) rectangle (\hhh*3,\vvv*1);
\draw[-] (\hhh*3,0) rectangle (\hhh*4,\vvv*1);
\draw[-] (\hhh*4,0) rectangle (\hhh*5,\vvv*1);
\draw[fill=black!30] (\hhh*5,0) rectangle (\hhh*6,\vvv*1);
\draw[fill=black!30] (\hhh*1,-\vvv*1) rectangle (\hhh*2,\vvv*0);
\node[] at (\hhh*0.5,\vvv*0.5) {$1'$};
\node[] at (\hhh*1.5,\vvv*0.5) {$1$};
\node[] at (\hhh*2.5,\vvv*0.5) {$1$};
\node[] at (\hhh*3.5,\vvv*0.5) {$1$};
\node[] at (\hhh*4.5,\vvv*0.5) {$1$};
\node[] at (\hhh*5.5,\vvv*0.5) {$1$};
\node[] at (\hhh*1.5,-\vvv*0.5) {$1$};
\end{tikzpicture}.
Thus
$\overrightarrow{{\rm P}}_{(1,2)}=\left\{(1,2),(1,3),\right.$
$\left.(1,4),(1,5),(1,6)\right\}$ and
$\overrightarrow{{\rm P}}_{(1,1)}=\left\{(1,1),(1,2),(2,2)\right\}.$
Moreover, ${\rm End}(\overrightarrow{{\rm P}}_{(1,2)})=(1,6)$ and ${\rm End}(\overrightarrow{{\rm P}}_{(1,1)})=(2,2)$.
\vskip 1mm

\begin{lem}\label{lem-sw-path}
Let $(A,B)$ be a shifted perforated $(\bfa,\bfb)$-pair such that $B$ extends $A$.
Then the shifted switching path $\overrightarrow{{\rm P}}_{(i,j)}$
corresponding to the $(i,j)$ box in $A$
has one of the following forms$:$
\begin{description}
\item[T1] $\left\{(i,j),\ldots,(i,j+y)\right\}$ for some $y \geq 0$

\item[T2] $\left\{(i,j),\ldots,(i+x,j)\right\}$ for some $x \geq 1$

\item[T3] $\left\{(i,j), (i,j+1), (i+1,j+1)\right\}$
\end{description}
\end{lem}
\begin{proof}
For our purpose it suffices to show that
no paths contain a subset of the following types:
\begin{description}
\item[E1] $\{(s,t), (s+1,t), (s+1,t+1)\}$

\item[E2] $\{(s,t), (s,t+1), (s,t+2), (s+1,t+2)\}$

\item[E3] $\{(s,t), (s,t+1), (s+1,t+1), (s+2,t+1)\}$
\end{description}

We first consider the case where the $(i,j)$ box in $A$ is $\Box(A \cup B)$.
When $\overrightarrow{{\rm P}}_{(i,j)}$ contains {\bf E1}, the $(s,t+1)$ box is contained in $B$.
If it is filled with $b'$, then $(s,t+1)$ should be in $\overrightarrow{{\rm P}}_{(i,j)}$
regardless of the letter in the $(s+1,t)$ box.
But this cannot occur since $\overrightarrow{{\rm P}}_{(i,j)}$ contains {\bf E1}.
Therefore the $(s,t+1)$ box is filled with $b$, which is also a contradiction since
the $(s+1,t+1)$ box in $A \cup B$ is filled with $\bfb$.
When $\overrightarrow{{\rm P}}_{(i,j)}$ contains {\bf E2},
the $(s,t+1),(s,t+2),(s+1,t),(s+1,t+1)$ and $(s+1,t+2)$ boxes should be contained in $B$.
This cannot occur since $B$ is an SYT of border strip shape.
Finally the case where $\overrightarrow{{\rm P}}_{(i,j)}$ contains {\bf E3}
can be treated in the same way as above.

Next we consider the case where the $(i,j)$ box in $A$ is not $\Box(A \cup B)$.
We can treat the case where the $(i,j)$ box in $A$ is $\Box(\sigma^k(A\cup B))$ by
applying the above argument to $\sigma^k(A\cup B)$ instead of $A\cup B$.
On the other hand, when the $(i,i)$ box in $A$ is moved to the $(i,i+1)$ position by either (S4) or (S7),
it cannot be $\Box(\sigma^k(A\cup B))$.
But in this case the $(i,i+2)$ box is empty or filled with $b$, which implies that
the $(i+1)$st row consists of only one $\bfb$-box.
It follows that $\overrightarrow{{\rm P}}_{(i,i)}$ should be horizontal.
\end{proof}
\vskip 2mm

Let $P$ be a set of boxes in the shifted diagram,
and let $A$ be a box in this diagram.
We say that $P$ lies to the {\it north} of $A$
if each box in $P$ is strictly to the north of $A$.
The other directions can be defined in the same manner.

\begin{lem}\label{lem-path-intersect}
Let $(A,B)$ be a shifted perforated $(\bfa,\bfb)$-pair such that $B$ extends $A$.
Suppose that the shifted switching path $\overrightarrow{{\rm P}}_{(s,t)}$
occurs directly after the shifted switching path $\overrightarrow{{\rm P}}_{(i,j)}$.
Then the following hold.
\vskip 1mm
{\rm (a)} Suppose that $A_{i,j}=A_{s,t} = a$.
When both ${\rm End}(\overrightarrow{{\rm P}}_{(i,j)})$ and  ${\rm End}(\overrightarrow{{\rm P}}_{(s,t)})$ are $a$-boxes,
$\overrightarrow{{\rm P}}_{(s,t)}$ is to the west of ${\rm End}(\overrightarrow{{\rm P}}_{(i,j)})$.
When ${\rm End}(\overrightarrow{{\rm P}}_{(i,j)})$ and ${\rm End}(\overrightarrow{{\rm P}}_{(s,t)})$ are an $a$-box and an $a'$-box, respectively, we have the following two cases:
\vskip 1mm
$\bullet$ If either $(S4)$ or $(S7)$ occurs in the shifted switching processes
corresponding to $\overrightarrow{{\rm P}}_{(i,j)}$ and  $\overrightarrow{{\rm P}}_{(s,t)}$ simultaneously,
then $\overrightarrow{{\rm P}}_{(s,t)}$ is to the north of ${\rm End}(\overrightarrow{{\rm P}}_{(i,j)})$.
\vskip 1mm
$\bullet$ If either $(S4)$ or $(S7)$ occurs only in the shifted switching process
corresponding to $\overrightarrow{{\rm P}}_{(s,t)}$,
then $\overrightarrow{{\rm P}}_{(s,t)}$ is to the west of
${\rm End}(\overrightarrow{{\rm P}}_{(i,j)})$.
\vskip 1mm
{\rm (b)} Suppose that $A_{i,j} =a$ and $A_{s,t}= a'$.
When ${\rm End}(\overrightarrow{{\rm P}}_{(i,j)})$ is an $a$-box
and ${\rm End}(\overrightarrow{{\rm P}}_{(s,t)})$
is an $a'$-box,
$\overrightarrow{{\rm P}}_{(s,t)}$
is to the west or north of
${\rm End}(\overrightarrow{{\rm P}}_{(i,j)})$.
On the other hand,
when both ${\rm End}(\overrightarrow{{\rm P}}_{(i,j)})$ and ${\rm End}(\overrightarrow{{\rm P}}_{(s,t)})$ are $a'$-boxes,
$\overrightarrow{{\rm P}}_{(s,t)}$
is to the north of
${\rm End}(\overrightarrow{{\rm P}}_{(i,j)})$.
\vskip 1mm
{\rm (c)} Suppose that $A_{i,j} = A_{s,t} = a'$.
Then
$\overrightarrow{{\rm P}}_{(s,t)}$
is to the north of
${\rm End}(\overrightarrow{{\rm P}}_{(i,j)})$.
\end{lem}
\begin{proof}
When
$\overrightarrow{{\rm P}}_{(i,j)}$ or $\overrightarrow{{\rm P}}_{(s,t)}$
consists of a single box,
the desired results are immediately from Lemma \ref{lem-sw-path}.
Hence we will only deal with the case where
both of them consist of at least two boxes.

{\rm (a)}
Identify $(s,t)$ with ${\rm o}_2$ and $(i,j)$ with ${\rm o}_1$.
Then the first assertion immediately follows from Lemma \ref{lem-sw-path} together with \cite[Lemma 2 (a)]{BKT}.
For the second assertion, we have to consider the following two cases.

If either $(S4)$ or $(S7)$ occurs in the shifted switching process
corresponding to $\overrightarrow{{\rm P}}_{(i,j)}$ and in $\overrightarrow{{\rm P}}_{(s,t)}$ simultaneously,
$(A \cup B)_{{s,t}}$ and $(A\cup B)_{{i,j}}$
are on the main diagonal and on the next column of $(A \cup B)_{{s,t}}$, respectively.
It says that $i=t=s$ and $j=s+1$.
By the given assumption,
either $(S4)$ or $(S7)$ is applied to $\sigma^k(A \cup B)$ for some $k \geq 1$,
thus $\{(s,s+1),(s+1,s+1)\}$ and $\{(s,s),(s,s+1)\}$
should be contained in $\overrightarrow{{\rm P}}_{(i,j)}$ and $\overrightarrow{{\rm P}}_{(s,t)}$,
respectively.
Therefore, by Lemma \ref{lem-sw-path},
${\rm End}(\overrightarrow{{\rm P}}_{(s,t)})$ is to
the north of ${\rm End}(\overrightarrow{{\rm P}}_{(i,j)})$.
See the following figure:
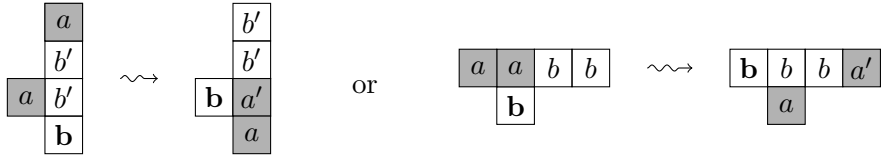
\begin{center}
\begin{tikzpicture}
\def \hhh{5mm}    
\def \vvv{5mm}    
\def \hhhhh{25mm}  
\draw[-] (\hhh*1,-\vvv*1) rectangle (\hhh*2,-\vvv*0);
\draw[-] (\hhh*1,-\vvv*2) rectangle (\hhh*2,-\vvv*1);
\draw[-] (\hhh*1,-\vvv*3) rectangle (\hhh*2,-\vvv*2);
\draw[-] (\hhh*1,-\vvv*4) rectangle (\hhh*2,-\vvv*3);
\draw[fill=black!30] (\hhh*1,\vvv*0) rectangle (\hhh*2,-\vvv*1);
\draw[fill=black!30] (\hhh*0,-\vvv*3) rectangle (\hhh*1,-\vvv*2);
\node[] at (\hhh*1.5,-\vvv*0.5) {$a$};
\node[] at (\hhh*1.5,-\vvv*1.5) {$b'$};
\node[] at (\hhh*0.5,-\vvv*2.5) {$a$};
\node[] at (\hhh*1.5,-\vvv*2.5) {$b'$};
\node[] at (\hhh*1.5,-\vvv*3.5) {$\bfb$};
\draw[->,decorate,decoration={snake,amplitude=.4mm,segment length=2mm,post length=1mm}] (\hhh*3,-\vvv*2) -- (-\hhh*1+\hhhhh*1,-\vvv*2);
\draw[-] (\hhh*1+\hhhhh,-\vvv*1) rectangle (\hhh*2+\hhhhh,-\vvv*0);
\draw[-] (\hhh*1+\hhhhh,-\vvv*2) rectangle (\hhh*2+\hhhhh,-\vvv*1);
\draw[-] (\hhh*1+\hhhhh,-\vvv*3) rectangle (\hhh*2+\hhhhh,-\vvv*2);
\draw[-] (\hhh*0+\hhhhh,-\vvv*3) rectangle (\hhh*1+\hhhhh,-\vvv*2);
\draw[fill=black!30] (\hhh*1+\hhhhh,-\vvv*3) rectangle (\hhh*2+\hhhhh,-\vvv*2);
\draw[fill=black!30] (\hhh*1+\hhhhh,-\vvv*4) rectangle (\hhh*2+\hhhhh,-\vvv*3);
\node[] at (\hhh*1.5+\hhhhh,-\vvv*0.5) {$b'$};
\node[] at (\hhh*1.5+\hhhhh,-\vvv*1.5) {$b'$};
\node[] at (\hhh*0.5+\hhhhh,-\vvv*2.5) {$\bfb$};
\node[] at (\hhh*1.5+\hhhhh,-\vvv*2.5) {$a'$};
\node[] at (\hhh*1.5+\hhhhh,-\vvv*3.5) {$a$};
\end{tikzpicture}
\hskip 8mm
\begin{tikzpicture}[baseline=-19mm]
\def \hhh{5mm}    
\def \vvv{5mm}    
\def \hhhhh{36mm}  
\node[] at (-\hhh*2.5,-\vvv*2) {or};
\draw[-] (\hhh*1,-\vvv*3) rectangle (\hhh*2,-\vvv*2);
\draw[-] (\hhh*2,-\vvv*2) rectangle (\hhh*3,-\vvv*1);
\draw[-] (\hhh*3,-\vvv*2) rectangle (\hhh*4,-\vvv*1);
\draw[fill=black!30] (\hhh*1,-\vvv*2) rectangle (\hhh*2,-\vvv*1);
\draw[fill=black!30] (\hhh*0,-\vvv*2) rectangle (\hhh*1,-\vvv*1);
\node[] at (\hhh*0.5,-\vvv*1.5) {$a$};
\node[] at (\hhh*1.5,-\vvv*1.5) {$a$};
\node[] at (\hhh*2.5,-\vvv*1.5) {$b$};
\node[] at (\hhh*3.5,-\vvv*1.5) {$b$};
\node[] at (\hhh*1.5,-\vvv*2.5) {$\bfb$};
\draw[->,decorate,decoration={snake,amplitude=.4mm,segment length=2mm,post length=1mm}] (\hhh*5,-\vvv*1.5) -- (-\hhh*1+\hhhhh*1,-\vvv*1.5);
\draw[-] (\hhh*0+\hhhhh,-\vvv*2) rectangle (\hhh*1+\hhhhh,-\vvv*1);
\draw[-] (\hhh*1+\hhhhh,-\vvv*2) rectangle (\hhh*2+\hhhhh,-\vvv*1);
\draw[-] (\hhh*2+\hhhhh,-\vvv*2) rectangle (\hhh*3+\hhhhh,-\vvv*1);
\draw[fill=black!30] (\hhh*1+\hhhhh,-\vvv*3) rectangle (\hhh*2+\hhhhh,-\vvv*2);
\draw[fill=black!30] (\hhh*3+\hhhhh,-\vvv*2) rectangle (\hhh*4+\hhhhh,-\vvv*1);
\node[] at (\hhh*0.5+\hhhhh,-\vvv*1.5) {$\bfb$};
\node[] at (\hhh*1.5+\hhhhh,-\vvv*1.5) {$b$};
\node[] at (\hhh*2.5+\hhhhh,-\vvv*1.5) {$b$};
\node[] at (\hhh*3.5+\hhhhh,-\vvv*1.5) {$a'$};
\node[] at (\hhh*1.5+\hhhhh,-\vvv*2.5) {$a$};
\end{tikzpicture}
\end{center}

If either $(S4)$ or $(S7)$ occurs only in the shifted switching process
corresponding to $\overrightarrow{{\rm P}}_{(s,t)}$,
$\overrightarrow{{\rm P}}_{(s,t)}$ and $\overrightarrow{\rm P}_{(i,j)}$ do not intersect
since neither $(S4)$ nor $(S7)$ occurs in $\overrightarrow{\rm P}_{(i,j)}$.
Therefore, by Lemma \ref{lem-sw-path},
$\overrightarrow{\rm P}_{(s,t)}$ is given by $\{(s,t),(s+1,t),\ldots,(t,t) \}$.
It follows that
$\overrightarrow{{\rm P}}_{(s,t)}$ is to
the west of ${\rm End}(\overrightarrow{{\rm P}}_{(i,j)})$.
\vskip 2mm

{\rm (b)}
In the first case,
neither (S4) nor (S7) does not occur
in the shifted switching processes
corresponding to
$\overrightarrow{\rm P}_{(s,t)}$ and $\overrightarrow{\rm P}_{(i,j)}$.
It can be easily verified that
the assertion can be deduced from Lemma \ref{lem-sw-path} together with \cite[Lemma 2]{BKT}
by identifying $(s,t)$ with ${\rm o}_2'$ and $(i,j)$ with ${\rm o}_1$.

For the second case, we need to consider the following two subcases.
If either $(S4)$ or $(S7)$ occurs in the shifted switching process
corresponding to $\overrightarrow{{\rm P}}_{(i,j)}$ and  $\overrightarrow{{\rm P}}_{(s,t)}$ simultaneously,
we obtain the desired result in the same way as in {\rm (a)}.
If either $(S4)$ or $(S7)$ occurs only in the shifted switching processes
corresponding to $\overrightarrow{{\rm P}}_{(i,j)}$,
then
$(A\cup B)_{i,j}$ is on the main diagonal and $(A \cup B)_{s,t}$
is to the northeast of $(A\cup B)_{i,j}$.
So $j=i$ and $s < i$.
Therefore, by Lemma \ref{lem-sw-path},
$\overrightarrow{{\rm P}}_{(i,j)}$ is given by $\{(i,i),(i,i+1),\ldots,(i,i+r)\}$ for some $r\geq 1$.
It follows that $(A \cup B)_{i,i+r}$ is filled with $b$.
Furthermore, $t$ should be greater than or equal to $i+r-1$.
Otherwise $(A \cup B)_{s,i+r-1}$ and $(A \cup B)_{s,i+r}$
are filled with $b'$, which is a contradiction.
If $t = i+r-1$, then $(A \cup B)_{s,i+r}$ is filled with $b'$.
So $\overrightarrow{{\rm P}}_{(s,t)}$ contains $\{(s,t),(s,t+1) \}$.
It means that $\overrightarrow{{\rm P}}_{(s,t)}$ is to
the north of ${\rm End}(\overrightarrow{{\rm P}}_{(i,j)})$.
Refer to the following figure:
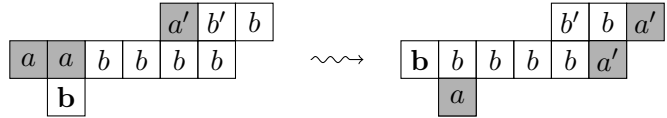
\begin{center}
\begin{tikzpicture}
\def \hhh{5mm}    
\def \vvv{5mm}    
\def \hhhhh{52mm}  
\draw[-] (\hhh*5,-\vvv*1) rectangle (\hhh*6,-\vvv*0);
\draw[-] (\hhh*6,-\vvv*1) rectangle (\hhh*7,-\vvv*0);
\draw[-] (\hhh*2,-\vvv*2) rectangle (\hhh*3,-\vvv*1);
\draw[-] (\hhh*3,-\vvv*2) rectangle (\hhh*4,-\vvv*1);
\draw[-] (\hhh*4,-\vvv*2) rectangle (\hhh*5,-\vvv*1);
\draw[-] (\hhh*5,-\vvv*2) rectangle (\hhh*6,-\vvv*1);
\draw[-] (\hhh*1,-\vvv*3) rectangle (\hhh*2,-\vvv*2);
\draw[fill=black!30] (\hhh*4,-\vvv*1) rectangle (\hhh*5,-\vvv*0);
\draw[fill=black!30] (\hhh*0,-\vvv*2) rectangle (\hhh*1,-\vvv*1);
\draw[fill=black!30] (\hhh*1,-\vvv*2) rectangle (\hhh*2,-\vvv*1);
\node[] at (\hhh*4.5,-\vvv*0.5) {$a'$};
\node[] at (\hhh*5.5,-\vvv*0.5) {$b'$};
\node[] at (\hhh*6.5,-\vvv*0.5) {$b$};
\node[] at (\hhh*0.5,-\vvv*1.5) {$a$};
\node[] at (\hhh*1.5,-\vvv*1.5) {$a$};
\node[] at (\hhh*2.5,-\vvv*1.5) {$b$};
\node[] at (\hhh*3.5,-\vvv*1.5) {$b$};
\node[] at (\hhh*4.5,-\vvv*1.5) {$b$};
\node[] at (\hhh*5.5,-\vvv*1.5) {$b$};
\node[] at (\hhh*1.5,-\vvv*2.5) {$\bfb$};
\draw[->,decorate,decoration={snake,amplitude=.4mm,segment length=2mm,post length=1mm}] (\hhh*8,-\vvv*1.5) -- (-\hhh*1+\hhhhh*1,-\vvv*1.5);
\draw[-] (\hhh*4+\hhhhh,-\vvv*1) rectangle (\hhh*5+\hhhhh,-\vvv*0);
\draw[-] (\hhh*5+\hhhhh,-\vvv*1) rectangle (\hhh*6+\hhhhh,-\vvv*0);
\draw[-] (\hhh*0+\hhhhh,-\vvv*2) rectangle (\hhh*1+\hhhhh,-\vvv*1);
\draw[-] (\hhh*1+\hhhhh,-\vvv*2) rectangle (\hhh*2+\hhhhh,-\vvv*1);
\draw[-] (\hhh*2+\hhhhh,-\vvv*2) rectangle (\hhh*3+\hhhhh,-\vvv*1);
\draw[-] (\hhh*3+\hhhhh,-\vvv*2) rectangle (\hhh*4+\hhhhh,-\vvv*1);
\draw[-] (\hhh*4+\hhhhh,-\vvv*2) rectangle (\hhh*5+\hhhhh,-\vvv*1);
\draw[-] (\hhh*5+\hhhhh,-\vvv*2) rectangle (\hhh*6+\hhhhh,-\vvv*1);
\draw[-] (\hhh*1+\hhhhh,-\vvv*3) rectangle (\hhh*2+\hhhhh,-\vvv*2);
\draw[fill=black!30] (\hhh*1+\hhhhh,-\vvv*3) rectangle (\hhh*2+\hhhhh,-\vvv*2);
\draw[fill=black!30] (\hhh*5+\hhhhh,-\vvv*2) rectangle (\hhh*6+\hhhhh,-\vvv*1);
\draw[fill=black!30] (\hhh*6+\hhhhh,-\vvv*1) rectangle (\hhh*7+\hhhhh,-\vvv*0);
\node[] at (\hhh*4.5+\hhhhh,-\vvv*0.5) {$b'$};
\node[] at (\hhh*5.5+\hhhhh,-\vvv*0.5) {$b$};
\node[] at (\hhh*6.5+\hhhhh,-\vvv*0.5) {$a'$};
\node[] at (\hhh*0.5+\hhhhh,-\vvv*1.5) {$\bfb$};
\node[] at (\hhh*1.5+\hhhhh,-\vvv*1.5) {$b$};
\node[] at (\hhh*2.5+\hhhhh,-\vvv*1.5) {$b$};
\node[] at (\hhh*3.5+\hhhhh,-\vvv*1.5) {$b$};
\node[] at (\hhh*4.5+\hhhhh,-\vvv*1.5) {$b$};
\node[] at (\hhh*5.5+\hhhhh,-\vvv*1.5) {$a'$};
\node[] at (\hhh*1.5+\hhhhh,-\vvv*2.5) {$a$};
\node[] at (\hhh*7.5+\hhhhh,-\vvv*2.5) {};
\end{tikzpicture}
\end{center}
If $t> i+r-1$, then $\overrightarrow{{\rm P}}_{(s,t)}$ is also to
the north of ${\rm End}(\overrightarrow{{\rm P}}_{(i,j)})$
since the $i$th row contains $\bfb$-boxes from
$(i,i+2)$ to $(i,i+r)$.
Refer to the following figure:
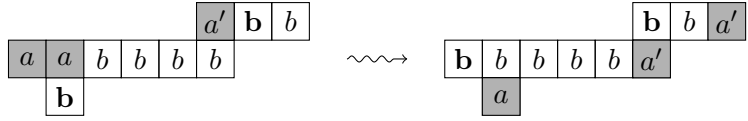
\begin{center}
\begin{tikzpicture}
\def \hhh{5mm}    
\def \vvv{5mm}    
\def \hhhhh{58mm}  
\draw[-] (\hhh*5,-\vvv*1) rectangle (\hhh*6,-\vvv*0);
\draw[-] (\hhh*6,-\vvv*1) rectangle (\hhh*7,-\vvv*0);
\draw[-] (\hhh*7,-\vvv*1) rectangle (\hhh*8,-\vvv*0);
\draw[-] (\hhh*2,-\vvv*2) rectangle (\hhh*3,-\vvv*1);
\draw[-] (\hhh*3,-\vvv*2) rectangle (\hhh*4,-\vvv*1);
\draw[-] (\hhh*4,-\vvv*2) rectangle (\hhh*5,-\vvv*1);
\draw[-] (\hhh*5,-\vvv*2) rectangle (\hhh*6,-\vvv*1);
\draw[-] (\hhh*1,-\vvv*3) rectangle (\hhh*2,-\vvv*2);
\draw[fill=black!30] (\hhh*5,-\vvv*1) rectangle (\hhh*6,-\vvv*0);
\draw[fill=black!30] (\hhh*0,-\vvv*2) rectangle (\hhh*1,-\vvv*1);
\draw[fill=black!30] (\hhh*1,-\vvv*2) rectangle (\hhh*2,-\vvv*1);
\node[] at (\hhh*1.5,-\vvv*1.5) {$a$};
\node[] at (\hhh*0.5,-\vvv*1.5) {$a$};
\node[] at (\hhh*5.5,-\vvv*0.5) {$a'$};
\node[] at (\hhh*1.5,-\vvv*2.5) {$\bfb$};
\node[] at (\hhh*2.5,-\vvv*1.5) {$b$};
\node[] at (\hhh*3.5,-\vvv*1.5) {$b$};
\node[] at (\hhh*4.5,-\vvv*1.5) {$b$};
\node[] at (\hhh*5.5,-\vvv*1.5) {$b$};
\node[] at (\hhh*6.5,-\vvv*0.5) {$\bfb$};
\node[] at (\hhh*7.5,-\vvv*0.5) {$b$};
\draw[->,decorate,decoration={snake,amplitude=.4mm,segment length=2mm,post length=1mm}] (\hhh*9,-\vvv*1.5) -- (-\hhh*1+\hhhhh*1,-\vvv*1.5);
\draw[-] (\hhh*5+\hhhhh,-\vvv*1) rectangle (\hhh*6+\hhhhh,-\vvv*0);
\draw[-] (\hhh*6+\hhhhh,-\vvv*1) rectangle (\hhh*7+\hhhhh,-\vvv*0);
\draw[-] (\hhh*0+\hhhhh,-\vvv*2) rectangle (\hhh*1+\hhhhh,-\vvv*1);
\draw[-] (\hhh*1+\hhhhh,-\vvv*2) rectangle (\hhh*2+\hhhhh,-\vvv*1);
\draw[-] (\hhh*2+\hhhhh,-\vvv*2) rectangle (\hhh*3+\hhhhh,-\vvv*1);
\draw[-] (\hhh*3+\hhhhh,-\vvv*2) rectangle (\hhh*4+\hhhhh,-\vvv*1);
\draw[-] (\hhh*4+\hhhhh,-\vvv*2) rectangle (\hhh*5+\hhhhh,-\vvv*1);
\draw[-] (\hhh*5+\hhhhh,-\vvv*2) rectangle (\hhh*6+\hhhhh,-\vvv*1);
\draw[-] (\hhh*1+\hhhhh,-\vvv*3) rectangle (\hhh*2+\hhhhh,-\vvv*2);
\draw[fill=black!30] (\hhh*1+\hhhhh,-\vvv*3) rectangle (\hhh*2+\hhhhh,-\vvv*2);
\draw[fill=black!30] (\hhh*5+\hhhhh,-\vvv*2) rectangle (\hhh*6+\hhhhh,-\vvv*1);
\draw[fill=black!30] (\hhh*7+\hhhhh,-\vvv*1) rectangle (\hhh*8+\hhhhh,-\vvv*0);
\node[] at (\hhh*1.5+\hhhhh,-\vvv*2.5) {$a$};
\node[] at (\hhh*5.5+\hhhhh,-\vvv*1.5) {$a'$};
\node[] at (\hhh*7.5+\hhhhh,-\vvv*0.5) {$a'$};
\node[] at (\hhh*0.5+\hhhhh,-\vvv*1.5) {$\bfb$};
\node[] at (\hhh*1.5+\hhhhh,-\vvv*1.5) {$b$};
\node[] at (\hhh*2.5+\hhhhh,-\vvv*1.5) {$b$};
\node[] at (\hhh*3.5+\hhhhh,-\vvv*1.5) {$b$};
\node[] at (\hhh*4.5+\hhhhh,-\vvv*1.5) {$b$};
\node[] at (\hhh*5.5+\hhhhh,-\vvv*0.5) {$\bfb$};
\node[] at (\hhh*6.5+\hhhhh,-\vvv*0.5) {$b$};
\end{tikzpicture}
\end{center}

(c) Identify $(s,t)$ with ${\rm o}_2$ and $(i,j)$ with ${\rm o}_1$.
Then this case immediately follows from
Lemma \ref{lem-sw-path} together with \cite[Lemma 2 (a)]{BKT}.
\end{proof}

\vskip 2mm
Now we are ready to present the proof of Proposition \ref{prop-preserving-sp}.
\vskip 2mm \noindent
{\bf Proof of Proposition \ref{prop-preserving-sp}.}
The only nontrivial part of the proof is to see that
$\sigma^{m}(A \cup B)^{(a)}$ and $\sigma^{m}(A \cup B)^{(b)}$\
are shifted perforated tableaux for all $m = 1, \ldots, n$.
For a fixed inner corner of $B$ let us move this box until it becomes a removable corner
by applying a sequence of shifted slides.
In each step one can see that the corresponding filling satisfies the four conditions in Definition \ref{Def-Perforated}.
Therefore the final filling obtained by discarding the new removable corner in the above
is an SYT of border strip shape.
By repeating this process, one can see that
$\sigma^{m}(A \cup B)^{(b)}$ is a shifted perforated $\bfb$-tableau.

Next we will show that $\sigma^{m}(A \cup B)^{(a)}$ is a shifted perforated $\bfa$-tableau.
Set $\Box(A\cup B) = (s,t)$
and let $k_1$ be the smallest positive integer such that
the $(s,t)$ box is moved to a removable corner, which is nothing but ${\rm End}(\overrightarrow{{\rm P}}_{(s,t)})$, after applying $\sigma^{k_1}$.
As the first step, we will show that $\sigma^m(A \cup B)^{(a)}$ is a shifted perforated tableau for all $m=1,\ldots,k_1$.
When $\overrightarrow{{\rm P}}_{(s,t)}$ is horizontal or vertical, the perforateness is obviously being preserved
in the corresponding switching process.
When $\overrightarrow{{\rm P}}_{(s,t)} = \{(s,t),(s,t+1),(s+1,t+1)\}$,
if $\Box(A\cup B)$ is filled with $a$ (resp., $a'$), then the perforateness is being preserved
since $\Box(A\cup B)$ is the easternmost $a$-box (resp., the southernmost $a'$-box) in $A\cup B$.

For $1 \leq l \leq |A|$,
let $A_l$
denote
the boxes in $A$ occupied by $l$ in $\stan^*(A)$.
Let
$k_{l}$ denote the nonnegative integer such that
$A_{l}$ is moved to a removable corner
after applying $\sigma^{k_1+ \cdots + k_{l}}$ to $A \cup B$.
Hereafter we abbreviate $k_1+k_2+\cdots+k_{l}$ to $\epsilon(l)$.
If $|A|=1$,
$\sigma^{m}(A \cup B)^{(a)}$ is a shifted perforated $\bfa$-tableau
for all $m = 1,\ldots, \epsilon(1)$
by the above paragraph.
Now assume that $|A| > 1$.
For $1 \leq l \leq |A|-1$,
we will show that
$\sigma^{m}(A \cup B)^{(a)}$ is a shifted perforated $\bfa$-tableau
for all $m = \epsilon(l)+1, \ldots, \epsilon(l)+k_{l+1}$
under the assumption that $\sigma^{m}(A \cup B)^{(a)}$ is a shifted perforated $\bfa$-tableau
for all $m = 1, \ldots, \epsilon(l)$.
Let $A_l$ be positioned at $(i,j)$
and $A_{l+1}$ at $(s,t)$.

First, assume that both $A_{s,t}$
and
${\rm End}(\overrightarrow{\rm P}_{(s,t)})$ are $a$-boxes.
Then $A_{i,j}=a$ and therefore, by Lemma \ref{lem-path-intersect},
$\overrightarrow{\rm P}_{(s,t)}$ is to the west of
${\rm End}(\overrightarrow{\rm P}_{(i,j)})$.
Second, assume that
$A_{s,t}$ is an $a$-box
and
${\rm End}(\overrightarrow{\rm P}_{(s,t)})$ is an $a'$-box.
If either (S4) or (S7) occurs only in
the shifted switching process corresponding to $\overrightarrow{\rm P}_{(s,t)}$,
then $\overrightarrow{\rm P}_{(s,t)}$ is vertical and
so the perforateness is being preserved.
If either (S4) or (S7) occurs in the shifted switching processes
corresponding to $\overrightarrow{\rm P}_{(i,j)}$ and $\overrightarrow{\rm P}_{(s,t)}$ simultaneously,
then $\overrightarrow{{\rm P}}_{(s,t)}$ is horizontal
and thus $\overrightarrow{{\rm P}}_{(s,t)}$ is to the strictly north
of ${\rm End}(\overrightarrow{{\rm P}}_{(i,j)})$.
Finally,
assume that both $A_{s,t}$
and
${\rm End}(\overrightarrow{\rm P}_{(s,t)})$ are $a'$-boxes.
If either (S4) or (S7) occurs in the shifted switching processes corresponding to $\overrightarrow{\rm P}_{(i,j)}$ and $\overrightarrow{\rm P}_{(s,t)}$ simultaneously,
then $\overrightarrow{{\rm P}}_{(s,t)}$ is horizontal
and thus
$\overrightarrow{{\rm P}}_{(s,t)}$ is to the strictly north of
${\rm End}(\overrightarrow{{\rm P}}_{(i,j)})$.
If neither (S4) nor (S7) occurs in the shifted switching process corresponding to $\overrightarrow{\rm P}_{(s,t)}$,
then $\overrightarrow{\rm P}_{(s,t)}$ is to the strictly north of
the $a'$-box ${\rm End}(\overrightarrow{\rm P}_{(i,j)})$,
or $\overrightarrow{\rm P}_{(s,t)}$ is to the strictly north or west of
the $a$-box ${\rm End}(\overrightarrow{\rm P}_{(i,j)})$.
This completes the proof.
\qed

\section{Properties and applications of the shifted tableau switching}
\label{sect-prop-ss}
The primary purpose of this section is show that
our shifted tableau switching is involutive and behave very nicely with respect to
tableaux whose reading words satisfy the lattice property.
To do this let us review (skew) Schur $P$- and $Q$-functions very briefly.

For $\ld, \mu \in \Ld^+$ with $\mu \subset \ld$, let us
define (skew) Schur $P$- and $Q$-functions by
$$
P_{\lambda/\mu}(x) = \sum_{T \in \mathcal{Y}(\lambda/\mu)} x^{T} 
\qquad\text{and}\qquad 
Q_{\lambda/\mu}(x) = \sum_{T \in \widetilde{\mathcal{Y}}(\lambda/\mu)} x^{T}.
$$
When $\mu=\emptyset$, we simply write them as $P_\lambda(x)$ and $Q_\lambda(x)$.
It is well known that $P_\lambda(x)$ and $Q_\lambda(x)$ indexed by $\lambda \in \Lambda^+$ 
form a basis of the subring they generate in the ring of symmetric functions, respectively.
See \cite{Mac} for more details.
For $\ld, \mu, \nu \in \Ld^+$, {\it shifted Littlewood-Richardson (LR) coefficients}
$f_{\ld \mu}^{\nu}$ are the structure constants appearing in the expansion of the product of two $P_{\ld}$ and $P_{\mu}$, that is,
\begin{equation*}
P_{\ld}(x)\cdot P_{\mu}(x) = \underset{\nu}{\sum}f_{\lambda  \mu}^{\nu}P_{\nu}(x).
\end{equation*}
To prove the positivity of $f_{\lambda  \mu}^{\nu}$ in a combinatorial way
has been an important and interesting problem in the study of combinatorics of symmetric functions.
For instance, \cite{Sa, Ste1, Wo}
approach to this problem with SYTs
and very recently
\cite{Cho, CNO, GJKKK, LLS, Se} 
with another combinatorial objects called `semistandard decomposition tableaux'.
We here focus on the former approach.

It is due to Sagan and Worley that $f_{\ld \mu}^{\nu}$ counts the standard shifted Young tableaux $S$ of shape  $\nu / \ld$ with
$\Rect(S) = T$ for an arbitrarily chosen standard shifted Young tableau $T$ of shape $\mu$.
Moreover, they proved that $\Rect(S)$
depends not on the shape of $S$ but on its reading word $w(S)$.
Later, Stembridge \cite{Ste1} gave a characterization for the SYTs
that rectify to the special SYT, denoted by $R_{\mu}$, whose $i$th row consists of only $i$'s for all $i=1, \ldots, \ell(\mu)$.
By virtue of their seminal works one can interpret the identity
\begin{equation}\label{eqn-Schur-P-and-Q-ftns}
Q_{\nu/\lambda}(x) = \sum_{\mu} f^{\nu}_{\lambda \mu} Q_\mu(x)
\end{equation}
in a combinatorial way by using the shifted jeu de taquin.
We here give another combinatorial interpretation of \eqref{eqn-Schur-P-and-Q-ftns}
by using our shifted tableau switching.
To do this, we need to review Stembridge's combinatorial interpretation of $f^{\nu}_{\lambda \mu}$.
Let us first introduce necessary definitions and notation.

Given any word $w=w_1w_2 \cdots w_n$ in $X$,
let $\hat{w}$ be the word obtained from $w$ by reading $w$
from right to left, and then
replacing each $k$ by $(k+1)'$ and each $k'$ by $k$ for all $k \geq 1$.
Let $|w|$ denote the word obtained by erasing the primes of $w$.
For example, if $w=1 1' 2' 1 2' 2$, then
$\hat{w} = 3' 2 2' 2 1 2'$
and $|w|=1 1 2 1 2 2$.

Let $w \hat{w} =a_1 \cdots a_{2n}$ be the concatenation of $w$ and $\hat{w}$,
and then define a sequence of statistics $m_i(j)$ ($0 \leq j \leq 2n$, $i \geq 1$) depending on $w$
by the multiplicity of $i$ among $a_1 a_2 \cdots a_j$ when $1 \leq j \leq 2n$,
or $0$ when $j=0$.
The word $w$ is said to satisfy the {\it lattice property} if, whenever $m_i(j)=m_{i-1}(j)$, we have that $a_{j+1} $ is neither $i$ nor $i'$.
Let us call a word satisfying the lattice property
a {\it lattice word}.

\begin{thm}\label{Stembridge's result}{\rm (\cite[Theorem 8.3]{Ste1})}
The shifted Littlewood-Richardson coefficient $f_{\ld \mu}^{\nu}$ is
the number of SYTs $T$ of shape $\sh(T) = \nu / \ld$ and $\wt(T) = \mu$
such that
\begin{itemize}
\item[{\rm (a)}] $w=\w(T)$ satisfies the lattice property, and

\item[{\rm (b)}] the rightmost $i$ of $|w|$ is unprimed in $w$
 $(1\leq i \leq \ell(\mu))$.
\end{itemize}
\end{thm}

Notice that each SYT satisfying the condition (b) in Theorem \ref{Stembridge's result}
belongs to $\mathcal{Y}(\nu/\lambda)$.
Let $\mathcal{Y}(\nu/\lambda,\mu):=\{T \in \mathcal{Y}(\nu/\lambda) : \wt(T)=\mu \} $
and let
$$\mathcal{F}^\nu_{\lambda \mu}
:=\{ T \in \mathcal{Y}(\nu/\lambda,\mu) :
\w(T) \text{ satisfies (a) and (b) in Theorem \ref{Stembridge's result}}\}.$$
We call an SYT $T$ in $\mathcal{F}^\nu_{\lambda \mu}$
a {\it Littlewood-Richardson-Stembridge (LRS) tableau}
and
a word satisfying both the lattice property and
Theorem \ref{Stembridge's result}(b)
an {\it LRS word}.

\begin{example}\label{LRS exm}
(a) Let $\ld = (4,2), \mu =(4,3,1)$ and $\nu =(6,5,2,1)$.
Then $f_{\ld \mu}^{\nu}=4$
since there are four LRS tableaux of shape $\nu / \ld$ and weight $\mu$
as follows:
\vskip 2mm
\begin{center}
\begin{tikzpicture}
\def\hhh{5mm}
\def\vvv{5mm}
\draw[-] (\hhh*5,\vvv*0) rectangle (\hhh*6,\vvv*1);
\draw[-] (\hhh*6,\vvv*0) rectangle (\hhh*7,\vvv*1);
\draw[-] (\hhh*4,-\vvv*1) rectangle (\hhh*5,\vvv*0);
\draw[-] (\hhh*5,-\vvv*1) rectangle (\hhh*6,\vvv*0);
\draw[-] (\hhh*6,-\vvv*1) rectangle (\hhh*7,\vvv*0);
\draw[-] (\hhh*3,-\vvv*2) rectangle (\hhh*4,-\vvv*1);
\draw[-] (\hhh*4,-\vvv*2) rectangle (\hhh*5,-\vvv*1);
\draw[-] (\hhh*4,-\vvv*3) rectangle (\hhh*5,-\vvv*2);
\node at (\hhh*5.5,\vvv*0.5) () {$1'$};
\node at (\hhh*6.5,\vvv*0.5) () {$1$};
\node at (\hhh*4.5,-\vvv*0.5) () {$1$};
\node at (\hhh*5.5,-\vvv*0.5) () {$1$};
\node at (\hhh*6.5,-\vvv*0.5) () {$2'$};
\node at (\hhh*3.5,-\vvv*1.5) () {$2$};
\node at (\hhh*4.5,-\vvv*1.5) () {$2$};
\node at (\hhh*4.5,-\vvv*2.5) () {$3$};
\end{tikzpicture}
\hskip 10mm
\begin{tikzpicture}
\def\hhh{5mm}
\def\vvv{5mm}
\draw[-] (\hhh*5,\vvv*0) rectangle (\hhh*6,\vvv*1);
\draw[-] (\hhh*6,\vvv*0) rectangle (\hhh*7,\vvv*1);
\draw[-] (\hhh*4,-\vvv*1) rectangle (\hhh*5,\vvv*0);
\draw[-] (\hhh*5,-\vvv*1) rectangle (\hhh*6,\vvv*0);
\draw[-] (\hhh*6,-\vvv*1) rectangle (\hhh*7,\vvv*0);
\draw[-] (\hhh*3,-\vvv*2) rectangle (\hhh*4,-\vvv*1);
\draw[-] (\hhh*4,-\vvv*2) rectangle (\hhh*5,-\vvv*1);
\draw[-] (\hhh*4,-\vvv*3) rectangle (\hhh*5,-\vvv*2);
\node at (\hhh*5.5,\vvv*0.5) () {$1$};
\node at (\hhh*6.5,\vvv*0.5) () {$1$};
\node at (\hhh*4.5,-\vvv*0.5) () {$1$};
\node at (\hhh*5.5,-\vvv*0.5) () {$2$};
\node at (\hhh*6.5,-\vvv*0.5) () {$2$};
\node at (\hhh*3.5,-\vvv*1.5) () {$1$};
\node at (\hhh*4.5,-\vvv*1.5) () {$2$};
\node at (\hhh*4.5,-\vvv*2.5) () {$3$};
\end{tikzpicture}
\hskip 10mm
\begin{tikzpicture}
\def\hhh{5mm}
\def\vvv{5mm}
\draw[-] (\hhh*5,\vvv*0) rectangle (\hhh*6,\vvv*1);
\draw[-] (\hhh*6,\vvv*0) rectangle (\hhh*7,\vvv*1);
\draw[-] (\hhh*4,-\vvv*1) rectangle (\hhh*5,\vvv*0);
\draw[-] (\hhh*5,-\vvv*1) rectangle (\hhh*6,\vvv*0);
\draw[-] (\hhh*6,-\vvv*1) rectangle (\hhh*7,\vvv*0);
\draw[-] (\hhh*3,-\vvv*2) rectangle (\hhh*4,-\vvv*1);
\draw[-] (\hhh*4,-\vvv*2) rectangle (\hhh*5,-\vvv*1);
\draw[-] (\hhh*4,-\vvv*3) rectangle (\hhh*5,-\vvv*2);
\node at (\hhh*5.5,\vvv*0.5) () {$1$};
\node at (\hhh*6.5,\vvv*0.5) () {$1$};
\node at (\hhh*4.5,-\vvv*0.5) () {$1$};
\node at (\hhh*5.5,-\vvv*0.5) () {$2'$};
\node at (\hhh*6.5,-\vvv*0.5) () {$2$};
\node at (\hhh*3.5,-\vvv*1.5) () {$1$};
\node at (\hhh*4.5,-\vvv*1.5) () {$2$};
\node at (\hhh*4.5,-\vvv*2.5) () {$3$};
\end{tikzpicture}
\hskip 10mm
\begin{tikzpicture}
\def\hhh{5mm}
\def\vvv{5mm}
\draw[-] (\hhh*5,\vvv*0) rectangle (\hhh*6,\vvv*1);
\draw[-] (\hhh*6,\vvv*0) rectangle (\hhh*7,\vvv*1);
\draw[-] (\hhh*4,-\vvv*1) rectangle (\hhh*5,\vvv*0);
\draw[-] (\hhh*5,-\vvv*1) rectangle (\hhh*6,\vvv*0);
\draw[-] (\hhh*6,-\vvv*1) rectangle (\hhh*7,\vvv*0);
\draw[-] (\hhh*3,-\vvv*2) rectangle (\hhh*4,-\vvv*1);
\draw[-] (\hhh*4,-\vvv*2) rectangle (\hhh*5,-\vvv*1);
\draw[-] (\hhh*4,-\vvv*3) rectangle (\hhh*5,-\vvv*2);
\node at (\hhh*5.5,\vvv*0.5) () {$1$};
\node at (\hhh*6.5,\vvv*0.5) () {$1$};
\node at (\hhh*4.5,-\vvv*0.5) () {$1'$};
\node at (\hhh*5.5,-\vvv*0.5) () {$2'$};
\node at (\hhh*6.5,-\vvv*0.5) () {$2$};
\node at (\hhh*3.5,-\vvv*1.5) () {$1$};
\node at (\hhh*4.5,-\vvv*1.5) () {$2$};
\node at (\hhh*4.5,-\vvv*2.5) () {$3$};
\end{tikzpicture}
\end{center}
\vskip 2mm \noindent
(b) For $\nu \in \Lambda^+$,
$f^{\nu}_{\emptyset \, \nu}=1$
and the corresponding LRS tableau is $R_\nu$.
\end{example}
\vskip 2mm

The following theorem is the main result of this section.
\begin{thm}\label{thm-main-section4}
Suppose that $S$ and $T$ are SYTs such that $T$ extends $S$.
If the shifted tableau switching transforms $(S,T)$ into $({}^ST, S_T)$,
then we have the following.
\begin{itemize}
  \item[{\rm (a)}] The shifted tableau switching transforms
 $({}^ST, S_T)$ into $(S,T)$.
  \item[{\rm (b)}] When $T$ is an LRS tableau, so is ${}^ST$.
  \item[{\rm (c)}] When $S$ is an LRS tableau, so is $S_T$.
\end{itemize}
\end{thm}

In order to prove this theorem,
we first introduce necessary lemmas.
Let $A$ be an SYT consisting of $(k+1)$ $\bfa$-boxes.
Define $A^{[k]}$ to be the subtableau of $A$ occupied by $k$ in $\rstan(A)$
and
\begin{displaymath}
A^{[\leq k]}:=A^{[k]} \cup A^{[k-1]} \cup \cdots \cup A^{[1]}.
\end{displaymath}
Then, by definition, $A = A^{[k+1]} \cup A^{[\leq k]}$
and $A^{[\leq k]}$ extends $A^{[k+1]}$.
If two shifted perforated $(\bfa,\bfb)$-pairs $(A,B)$ and $(A',B')$
are same when we ignore the $(i,j)$ box,
then we use the following expression
\begin{displaymath}
(A,B) \approx (A',B') \quad \text{up to $(i,j)$.}
\end{displaymath}

\begin{lem}\label{lem1-sw-involution}
Let $(A,B)$ be a shifted perforated $(\bfa,\bfb)$-pair
such that $|A|=k+1$ $(k \geq 1)$ and $B$ extends $A$.
Let $A^{[k+1]} = A_{i,j}$.
If neither {\rm (S4)} nor {\rm (S7)} occurs in the shifted switching process corresponding to
$\overrightarrow{\rm P}_{(i,j)}$,
then
\begin{displaymath}
\Sigma({}^AB,A_B)\approx \Sigma({}^{A^{[\leq k]}}B, {A^{[\leq k]}}_B)
\quad \text{up to $(i,j)$}.
\end{displaymath}
\end{lem}
\begin{proof}
Suppose that $u$ is the smallest integer such that
$\sigma^u(A^{[\leq k]},B)$ gets fully switched.
Since $A^{[k+1]}$ remains unchanged in applying $\sigma^{u}$ to $(A, B)$,
it follows that
$\Box(\sigma^u(A, B))= A^{[k+1]}=A_{i,j}$.
When $A^{[k+1]}$ is fully switched in $\sigma^u(A, B)$,
there is nothing to prove.
Hence we assume that $A^{[k+1]}$ is not fully switched.
Then, by Lemma \ref{lem-sw-path}, $\overrightarrow{\rm P}_{(i,j)}$ has one of the following types:
\begin{description}
\item [T1] $\{(i,j),\ldots,(i,j+y)\}$ for some $y \geq 1$
\item [T2] $ \{(i,j),\ldots,(i+x,j)\}$ for some $x \geq 1$
\item [T3] $\{(i,j),(i,j+1),(i+1,j+1)\}$
\end{description}
In case of ${\bf T1}$,
$\sigma^{u+y}(A,B)$ is fully switched and $\left(\sigma^{u+y}(A,B)\right)_{i-1,j+y}$ is a $b'$-box or an empty box
since $(\sigma^u(A,B))_{i,j+y}$ is a $\bfb$-box.
For $1 \leq s \leq y$,
let $v_s$ be the integer uniquely determined by the the condition
$$
\Box(\sigma^{v_s}({}^AB, A_B))=(i,j+y-s).
$$
Then we have
$$\sigma^{v_1+1}({}^AB,A_B) \approx \sigma^{v_1}({}^{A^{[\leq k]}}B, {A^{[\leq k]}}_B)
\quad \text{up to $(i,j), \ldots, (i,j+y-1)$.} $$
More generally, for $1 \leq s \leq y$ and $v_s \leq r \leq v_{s+1}-1$,
it is not difficult to see that
$$
\sigma^{r+1}({}^AB,A_B) \approx \sigma^{r+1-s}({}^{A^{[\leq k]}}B, {A^{[\leq k]}}_B)
\quad \text{up to $(i,j), \ldots, (i,j+y-s)$}.
$$
It should be noticed that we are regarding $v_{y+1}$ as a sufficiently large integer.
This implies that $\left(\Sigma({}^AB,A_B)\right)_{i,j} = A^{[k+1]}$
and
$$
\Sigma\left({}^AB,A_B\right)\approx \Sigma\left({}^{A^{[\leq k]}}B, {A^{[\leq k]}}_B\right)
\quad \text{up to $(i,j)$}
$$
because
the $(i,j)$ box is not in
$\Sigma({}^{A^{[\leq k]}}B, {A^{[\leq k]}}_B)$.

In case of ${\bf T2}$,
$\sigma^{u+x}(A,B)$ is fully switched.
For $1 \leq s \leq x$,
let $v_s$ be the integer uniquely determined by the condition
$$
\Box(\sigma^{v_s}({}^AB, A_B))=(i+x-s,j).
$$
For $1 \leq s \leq x$ and $v_s \leq r \leq v_{s+1}-1$,
$$
\sigma^{r+1}({}^AB,A_B) \approx \sigma^{r+1-s}({}^{A^{[\leq k]}}B, {A^{[\leq k]}}_B)
\quad
\text{up to $(i,j), \ldots, (i+x-s,j)$}.
$$
Here we are regarding $v_{x+1}$ as a sufficiently large integer.
This equivalence holds since
$\Box(\sigma^{v_1}({}^AB,A_B)) = (i+x-1,j)$
is switched with an $\bfa$-box placed at $(i+x,j)$
regardless of the $(i+x,j-1)$ box.
As in case of ${\bf T1}$, it holds that
$\left(\Sigma({}^AB,A_B)\right)_{i,j} = A^{[k+1]}$
and
$$
\Sigma\left({}^AB,A_B\right) \approx \Sigma\left({}^{A^{[\leq k]}}B, {A^{[\leq k]}}_B\right)
\quad
\text{up to $(i,j)$.}
$$

Finally, in case of ${\bf T3}$,
$\sigma^u(A, B)$ must contain one of the following subfillings:
\begin{center}
\begin{tikzpicture}
\def \hhh{5mm}    
\def \vvv{5mm}    
\def \hhhhh{30mm}  
\draw[fill=black!30] (\hhh*0,0) rectangle (\hhh*1,\vvv*1);
\draw[-] (\hhh*1,0) rectangle (\hhh*2,\vvv*1);
\draw[-] (\hhh*1,-\vvv*1) rectangle (\hhh*2,\vvv*0);
\node[] at (\hhh*0.5,\vvv*0.5) {$\bfa$};
\node[] at (\hhh*1.5,\vvv*0.5) {$b'$};
\node[] at (\hhh*1.5,-\vvv*0.5) {$b'$};
\node[] at (\hhh*1+\hhhhh*0.5,-\vvv*0) {or};
\draw[fill=black!30] (\hhh*0+\hhhhh,0) rectangle (\hhh*1+\hhhhh,\vvv*1);
\draw[-] (\hhh*1+\hhhhh,0) rectangle (\hhh*2+\hhhhh,\vvv*1);
\draw[-] (\hhh*1+\hhhhh,-\vvv*1) rectangle (\hhh*2+\hhhhh,\vvv*0);
\node[] at (\hhh*0.5+\hhhhh,\vvv*0.5) {$\bfa$};
\node[] at (\hhh*1.5+\hhhhh,\vvv*0.5) {$b'$};
\node[] at (\hhh*1.5+\hhhhh,-\vvv*0.5) {$b$};
\node[] at (\hhh*1+\hhhhh*1.5,-\vvv*0) {or};
\draw[fill=black!30] (\hhh*0+\hhhhh*2,0) rectangle (\hhh*1+\hhhhh*2,\vvv*1);
\draw[-] (\hhh*1+\hhhhh*2,0) rectangle (\hhh*2+\hhhhh*2,\vvv*1);
\draw[-] (\hhh*0+\hhhhh*2,-\vvv*1) rectangle (\hhh*2+\hhhhh*2,\vvv*0);
\draw[-] (\hhh*1+\hhhhh*2,-\vvv*1) rectangle (\hhh*2+\hhhhh*2,\vvv*0);
\node[] at (\hhh*0.5+\hhhhh*2,\vvv*0.5) {$\bfa$};
\node[] at (\hhh*1.5+\hhhhh*2,\vvv*0.5) {$b'$};
\node[] at (\hhh*0.5+\hhhhh*2,-\vvv*0.5) {$b'$};
\node[] at (\hhh*1.5+\hhhhh*2,-\vvv*0.5) {$b$};
\node[] at (\hhh*1+\hhhhh*2.5,-\vvv*0) {or};
\draw[fill=black!30] (\hhh*0+\hhhhh*3,0) rectangle (\hhh*1+\hhhhh*3,\vvv*1);
\draw[-] (\hhh*1+\hhhhh*3,0) rectangle (\hhh*2+\hhhhh*3,\vvv*1);
\draw[-] (\hhh*0+\hhhhh*3,-\vvv*1) rectangle (\hhh*1+\hhhhh*3,\vvv*0);
\draw[-] (\hhh*1+\hhhhh*3,-\vvv*1) rectangle (\hhh*2+\hhhhh*3,\vvv*0);
\node[] at (\hhh*0.5+\hhhhh*3,\vvv*0.5) {$\bfa$};
\node[] at (\hhh*1.5+\hhhhh*3,\vvv*0.5) {$b'$};
\node[] at (\hhh*0.5+\hhhhh*3,-\vvv*0.5) {$b$};
\node[] at (\hhh*1.5+\hhhhh*3,-\vvv*0.5) {$b$};
\end{tikzpicture}
\end{center}
It implies that $({}^AB,A_B)$
must contain one of the following subfillings:
\begin{center}
\begin{tikzpicture}
\def \hhh{5mm}    
\def \vvv{5mm}    
\def \hhhhh{30mm}  
\draw[-] (\hhh*0,0) rectangle (\hhh*1,\vvv*1);
\draw[-] (\hhh*1,0) rectangle (\hhh*2,\vvv*1);
\draw[fill=black!30] (\hhh*1,-\vvv*1) rectangle (\hhh*2,\vvv*0);
\node[] at (\hhh*0.5,\vvv*0.5) {$b'$};
\node[] at (\hhh*1.5,\vvv*0.5) {$b$};
\node[] at (\hhh*1.5,-\vvv*0.5) {$\bfa$};
\node[] at (\hhh*1+\hhhhh*0.5,-\vvv*0) {or};
\draw[-] (\hhh*0+\hhhhh,0) rectangle (\hhh*1+\hhhhh,\vvv*1);
\draw[-] (\hhh*1+\hhhhh,0) rectangle (\hhh*2+\hhhhh,\vvv*1);
\draw[fill=black!30] (\hhh*1+\hhhhh,-\vvv*1) rectangle (\hhh*2+\hhhhh,\vvv*0);
\node[] at (\hhh*0.5+\hhhhh,\vvv*0.5) {$b$};
\node[] at (\hhh*1.5+\hhhhh,\vvv*0.5) {$b$};
\node[] at (\hhh*1.5+\hhhhh,-\vvv*0.5) {$\bfa$};
\node[] at (\hhh*1+\hhhhh*1.5,-\vvv*0) {or};
\draw[-] (\hhh*0+\hhhhh*2,0) rectangle (\hhh*1+\hhhhh*2,\vvv*1);
\draw[-] (\hhh*1+\hhhhh*2,0) rectangle (\hhh*2+\hhhhh*2,\vvv*1);
\draw[-] (\hhh*0+\hhhhh*2,-\vvv*1) rectangle (\hhh*2+\hhhhh*2,\vvv*0);
\draw[fill=black!30] (\hhh*1+\hhhhh*2,-\vvv*1) rectangle (\hhh*2+\hhhhh*2,\vvv*0);
\node[] at (\hhh*0.5+\hhhhh*2,\vvv*0.5) {$b'$};
\node[] at (\hhh*1.5+\hhhhh*2,\vvv*0.5) {$b$};
\node[] at (\hhh*0.5+\hhhhh*2,-\vvv*0.5) {$b'$};
\node[] at (\hhh*1.5+\hhhhh*2,-\vvv*0.5) {$\bfa$};
\node[] at (\hhh*1+\hhhhh*2.5,-\vvv*0) {or};
\draw[-] (\hhh*0+\hhhhh*3,0) rectangle (\hhh*1+\hhhhh*3,\vvv*1);
\draw[-] (\hhh*1+\hhhhh*3,0) rectangle (\hhh*2+\hhhhh*3,\vvv*1);
\draw[-] (\hhh*0+\hhhhh*3,-\vvv*1) rectangle (\hhh*1+\hhhhh*3,\vvv*0);
\draw[fill=black!30] (\hhh*1+\hhhhh*3,-\vvv*1) rectangle (\hhh*2+\hhhhh*3,\vvv*0);
\node[] at (\hhh*0.5+\hhhhh*3,\vvv*0.5) {$b'$};
\node[] at (\hhh*1.5+\hhhhh*3,\vvv*0.5) {$b$};
\node[] at (\hhh*0.5+\hhhhh*3,-\vvv*0.5) {$b$};
\node[] at (\hhh*1.5+\hhhhh*3,-\vvv*0.5) {$\bfa$};
\end{tikzpicture}
\end{center}
Note that in the first two subfillings,
the upper leftmost box denotes the $(i,i)$ box
in $\sigma^u(A, B)$ and $\left({}^AB,A_B\right)$, respectively.
In particular, $A^{[k+1]} = A_{i,i}$.
Let $v$ be the integer uniquely determined by the condition
$$
\Box\left(\sigma^{v}\left({}^AB,A_B\right)\right)= (i,i+1).
$$
Applying (S4) to $\Box(\sigma^{v}({}^AB,A_B))$, we can deduce that
$$\sigma^{v+1}\left({}^AB,A_B\right) \approx
\sigma^{v}\left({}^{A^{[\leq k]}}B, {A^{[\leq k]}}_B\right)
\quad \text{up to $(i,i)$.}$$
Furthermore,
since the $(i,i)$ $\bfa$-box in $\sigma^{v+1}\left({}^AB,A_B\right)$
cannot be switched with
any $\bfb$-box in $\sigma^{v+1}\left({}^AB,A_B\right)$,
it follows that for $r \geq 1$,
$$
\sigma^{v+1+r}\left({}^AB,A_B\right) \approx
\sigma^{v+r}\left({}^{A^{[k]}}B, {A^{[k]}}_B\right)
\quad \text{up to $(i,i)$.}$$
For the remaining two subfillings,
we let $v_1$ and $v_2$ be the integer uniquely determined by
$$
\Box\left(\sigma^{v_1}\left({}^AB,A_B\right)\right)= (i,j+1)
$$
and
$$
\Box\left(\sigma^{v_2}\left({}^AB,A_B\right)\right)= (i,j),
$$
respectively.
Here the $(i,j)$ box is the upper leftmost box in the each subfilling.
Then
$$
\sigma^{v_1+1}\left({}^AB,A_B\right) \approx
\sigma^{v_1}({}^{A^{[\leq k]}}B, {A^{[\leq k]}}_B)
\quad \text{up to $(i,j)$ and $(i,j+1)$.}
$$
Now consider any $\bfb$-box on the southwest of $(i,j)$.
The shifted switching path of the $(i,j)$ box does not contain $(i,j+1)$.
Moreover, the $(i,j)$ $\bfa$-box in $\sigma^{v_2+1}({}^AB,A_B)$
cannot be switched with any $\bfb$-box in $\sigma^{v_2+1}({}^AB,A_B)$.
Therefore, for $s=1,2$ and $v_s \leq r \leq v_{s+1}-1$,
$$
\sigma^{r+1}\left({}^AB,A_B\right) \approx
\sigma^{r+1-s}({}^{A^{[\leq k]}}B, {A^{[\leq k]}}_B)
\quad
\text{up to $(i,j)$ and $(i,j+2-s)$.}
$$
Here we are regarding $v_3$ as a sufficiently large integer.
Hence we showed for all subfillings above that
$$
\Sigma\left({}^AB,A_B\right) \approx \Sigma\left({}^{A^{[\leq k]}}B, {A^{[\leq k]}}_B\right)
\quad
\text{up to $(i,j)$,}$$
where $A^{[k+1]} = A_{i,j}$.
\end{proof}
\vskip 2mm

\begin{lem}\label{lem2-sw-involution}
Let $(A,B)$ be a shifted perforated $(\bfa,\bfb)$-pair
such that $|A|=k+1$ $(k \geq 1)$ and $B$ extends $A$.
Let $A^{[k+1]} = A_{i,i}$.
If either {\rm (S4)} or {\rm (S7)} occurs in the shifted switching process corresponding to
$\overrightarrow{\rm P}_{(i,i)}$,
then
\begin{displaymath}
\Sigma\left({}^AB,A_B\right)\approx \Sigma\left({}^{A^{[\leq k]}}B, {A^{[\leq k]}}_B\right)
\quad \text{up to $(i,i)$}.
\end{displaymath}
\end{lem}
\begin{proof}
Let $u$ be the smallest integer such that
$\sigma^u(A^{[\leq k]},B)$ is fully switched.
Then the $u$th switch
should be (S4) or (S7) and by which
$A^{[k+1]}$ is moved from $(i,i)$ to $(i,i+1)$.

When (S4) is applied to $\sigma^{u-1}(A,B)$,
the $(i,i+1)$ box in $\sigma^{u-1}(A,B)$ is filled with $a$
and this box is in the shifted switching path of $A^{[k]}$.
In particular, the $(i,i+2)$ box in $\sigma^{u-1}(A,B)$ is
either an $a$-box or an empty box
and
the $i$th column in $A\cup B$ has only a single box at $(i,i)$.
Refer to the following figure:
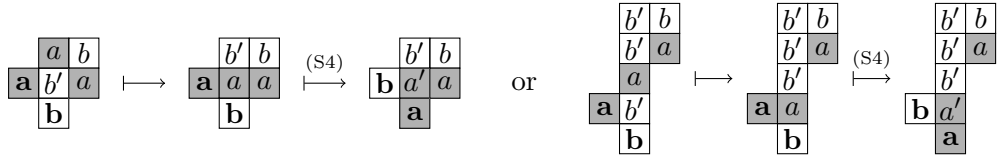
\begin{center}
\begin{tikzpicture}[baseline=-16mm]
\def \hhh{4mm}    
\def \vvv{4mm}    
\def \hhhhh{24mm}  
\draw[fill=black!30] (\hhh*0,-\vvv*2) rectangle (\hhh*1,-\vvv*1);
\draw[fill=black!30] (\hhh*1,-\vvv*1) rectangle (\hhh*2,\vvv*0);
\draw[-] (\hhh*1,-\vvv*2) rectangle (\hhh*2,-\vvv*1);
\draw[-] (\hhh*1,-\vvv*3) rectangle (\hhh*2,-\vvv*2);
\draw[-] (\hhh*2,-\vvv*1) rectangle (\hhh*3,\vvv*0);
\draw[fill=black!30] (\hhh*2,-\vvv*2) rectangle (\hhh*3,-\vvv*1);
\node[] at (\hhh*0.5,-\vvv*1.5) {$\bfa$};
\node[] at (\hhh*1.5,-\vvv*0.5) {$a$};
\node[] at (\hhh*1.5,-\vvv*1.5) {$b'$};
\node[] at (\hhh*1.5,-\vvv*2.5) {$\bfb$};
\node[] at (\hhh*2.5,-\vvv*0.5) {$b$};
\node[] at (\hhh*2.5,-\vvv*1.5) {$a$};
\draw[|->] (\hhh*3.8,-\vvv*1.5) -- (-\hhh*0.8+\hhhhh*1,-\vvv*1.5);
\draw[fill=black!30] (\hhh*0+\hhhhh,-\vvv*2) rectangle (\hhh*1+\hhhhh,-\vvv*1);
\draw[-] (\hhh*1+\hhhhh,-\vvv*1) rectangle (\hhh*2+\hhhhh,-\vvv*0);
\draw[fill=black!30] (\hhh*1+\hhhhh,-\vvv*2) rectangle (\hhh*2+\hhhhh,-\vvv*1);
\draw[-] (\hhh*1+\hhhhh,-\vvv*3) rectangle (\hhh*2+\hhhhh,-\vvv*2);
\draw[-] (\hhh*2+\hhhhh,-\vvv*1) rectangle (\hhh*3+\hhhhh,-\vvv*0);
\draw[fill=black!30] (\hhh*2+\hhhhh,-\vvv*2) rectangle (\hhh*3+\hhhhh,-\vvv*1);
\node[] at (\hhh*0.5+\hhhhh,-\vvv*1.5) {$\bfa$};
\node[] at (\hhh*1.5+\hhhhh,-\vvv*0.5) {$b'$};
\node[] at (\hhh*1.5+\hhhhh,-\vvv*1.5) {$a$};
\node[] at (\hhh*1.5+\hhhhh,-\vvv*2.5) {$\bfb$};
\node[] at (\hhh*2.5+\hhhhh,-\vvv*0.5) {$b$};
\node[] at (\hhh*2.5+\hhhhh,-\vvv*1.5) {$a$};
\draw[|->] (\hhh*3.8+\hhhhh*1,-\vvv*1.5) -- (-\hhh*0.8+\hhhhh*2,-\vvv*1.5)
node[above,midway] {\tiny (S4)};
\draw[-] (\hhh*0+\hhhhh*2,-\vvv*2) rectangle (\hhh*1+\hhhhh*2,-\vvv*1);
\draw[-] (\hhh*1+\hhhhh*2,-\vvv*1) rectangle (\hhh*2+\hhhhh*2,-\vvv*0);
\draw[fill=black!30] (\hhh*1+\hhhhh*2,-\vvv*2) rectangle (\hhh*2+\hhhhh*2,-\vvv*1);
\draw[fill=black!30] (\hhh*1+\hhhhh*2,-\vvv*3) rectangle (\hhh*2+\hhhhh*2,-\vvv*2);
\draw[-] (\hhh*2+\hhhhh*2,-\vvv*1) rectangle (\hhh*3+\hhhhh*2,-\vvv*0);
\draw[fill=black!30] (\hhh*2+\hhhhh*2,-\vvv*2) rectangle (\hhh*3+\hhhhh*2,-\vvv*1);
\node[] at (\hhh*0.5+\hhhhh*2,-\vvv*1.5) {$\bfb$};
\node[] at (\hhh*1.5+\hhhhh*2,-\vvv*0.5) {$b'$};
\node[] at (\hhh*1.5+\hhhhh*2,-\vvv*1.4) {$a'$};
\node[] at (\hhh*1.5+\hhhhh*2,-\vvv*2.5) {$\bfa$};
\node[] at (\hhh*2.5+\hhhhh*2,-\vvv*0.5) {$b$};
\node[] at (\hhh*2.5+\hhhhh*2,-\vvv*1.5) {$a$};
\end{tikzpicture}
\hskip 4mm
\begin{tikzpicture}[baseline=-10mm]
\node[] at (0,0) {or};
\end{tikzpicture}
\hskip 4mm
\begin{tikzpicture}
\def \hhh{4mm}    
\def \vvv{4mm}    
\def \hhhhh{21mm}  
\draw[fill=black!30] (\hhh*0,-\vvv*4) rectangle (\hhh*1,-\vvv*3);
\draw[-] (\hhh*1,-\vvv*1) rectangle (\hhh*2,\vvv*0);
\draw[-] (\hhh*1,-\vvv*2) rectangle (\hhh*2,-\vvv*1);
\draw[fill=black!30] (\hhh*1,-\vvv*3) rectangle (\hhh*2,-\vvv*2);
\draw[-] (\hhh*1,-\vvv*4) rectangle (\hhh*2,-\vvv*3);
\draw[-] (\hhh*1,-\vvv*5) rectangle (\hhh*2,-\vvv*4);
\draw[-] (\hhh*2,-\vvv*1) rectangle (\hhh*3,\vvv*0);
\draw[fill=black!30] (\hhh*2,-\vvv*2) rectangle (\hhh*3,-\vvv*1);
\node[] at (\hhh*0.5,-\vvv*3.5) {$\bfa$};
\node[] at (\hhh*1.5,-\vvv*0.5) {$b'$};
\node[] at (\hhh*1.5,-\vvv*1.5) {$b'$};
\node[] at (\hhh*1.5,-\vvv*2.5) {$a$};
\node[] at (\hhh*1.5,-\vvv*3.5) {$b'$};
\node[] at (\hhh*1.5,-\vvv*4.5) {$\bfb$};
\node[] at (\hhh*2.5,-\vvv*0.5) {$b$};
\node[] at (\hhh*2.5,-\vvv*1.5) {$a$};
\draw[|->] (\hhh*3.5,-\vvv*2.5) -- (-\hhh*0.5+\hhhhh*1,-\vvv*2.5);
\draw[fill=black!30] (\hhh*0+\hhhhh,-\vvv*4) rectangle (\hhh*1+\hhhhh,-\vvv*3);
\draw[-] (\hhh*1+\hhhhh,-\vvv*1) rectangle (\hhh*2+\hhhhh,-\vvv*0);
\draw[-] (\hhh*1+\hhhhh,-\vvv*2) rectangle (\hhh*2+\hhhhh,-\vvv*1);
\draw[-] (\hhh*1+\hhhhh,-\vvv*3) rectangle (\hhh*2+\hhhhh,-\vvv*2);
\draw[fill=black!30] (\hhh*1+\hhhhh,-\vvv*4) rectangle (\hhh*2+\hhhhh,-\vvv*3);
\draw[-] (\hhh*1+\hhhhh,-\vvv*5) rectangle (\hhh*2+\hhhhh,-\vvv*4);
\draw[-] (\hhh*2+\hhhhh,-\vvv*1) rectangle (\hhh*3+\hhhhh,-\vvv*0);
\draw[fill=black!30] (\hhh*2+\hhhhh,-\vvv*2) rectangle (\hhh*3+\hhhhh,-\vvv*1);
\node[] at (\hhh*0.5+\hhhhh,-\vvv*3.5) {$\bfa$};
\node[] at (\hhh*1.5+\hhhhh,-\vvv*0.5) {$b'$};
\node[] at (\hhh*1.5+\hhhhh,-\vvv*1.5) {$b'$};
\node[] at (\hhh*1.5+\hhhhh,-\vvv*2.5) {$b'$};
\node[] at (\hhh*1.5+\hhhhh,-\vvv*3.5) {$a$};
\node[] at (\hhh*1.5+\hhhhh,-\vvv*4.5) {$\bfb$};
\node[] at (\hhh*2.5+\hhhhh,-\vvv*0.5) {$b$};
\node[] at (\hhh*2.5+\hhhhh,-\vvv*1.5) {$a$};
\draw[|->] (\hhh*3.5+\hhhhh*1,-\vvv*2.5) -- (-\hhh*0.5+\hhhhh*2,-\vvv*2.5)
node[above,midway] {\tiny (S4)};
\draw[-] (\hhh*0+\hhhhh*2,-\vvv*4) rectangle (\hhh*1+\hhhhh*2,-\vvv*3);
\draw[-] (\hhh*1+\hhhhh*2,-\vvv*1) rectangle (\hhh*2+\hhhhh*2,-\vvv*0);
\draw[-] (\hhh*1+\hhhhh*2,-\vvv*2) rectangle (\hhh*2+\hhhhh*2,-\vvv*1);
\draw[-] (\hhh*1+\hhhhh*2,-\vvv*3) rectangle (\hhh*2+\hhhhh*2,-\vvv*2);
\draw[fill=black!30] (\hhh*1+\hhhhh*2,-\vvv*4) rectangle (\hhh*2+\hhhhh*2,-\vvv*3);
\draw[fill=black!30] (\hhh*1+\hhhhh*2,-\vvv*5) rectangle (\hhh*2+\hhhhh*2,-\vvv*4);
\draw[-] (\hhh*2+\hhhhh*2,-\vvv*1) rectangle (\hhh*3+\hhhhh*2,-\vvv*0);
\draw[fill=black!30] (\hhh*2+\hhhhh*2,-\vvv*2) rectangle (\hhh*3+\hhhhh*2,-\vvv*1);
\node[] at (\hhh*0.5+\hhhhh*2,-\vvv*3.5) {$\bfb$};
\node[] at (\hhh*1.5+\hhhhh*2,-\vvv*0.5) {$b'$};
\node[] at (\hhh*1.5+\hhhhh*2,-\vvv*1.5) {$b'$};
\node[] at (\hhh*1.5+\hhhhh*2,-\vvv*2.5) {$b'$};
\node[] at (\hhh*1.5+\hhhhh*2,-\vvv*3.5) {$a'$};
\node[] at (\hhh*1.5+\hhhhh*2,-\vvv*4.5) {$\bfa$};
\node[] at (\hhh*2.5+\hhhhh*2,-\vvv*0.5) {$b$};
\node[] at (\hhh*2.5+\hhhhh*2,-\vvv*1.5) {$a$};
\end{tikzpicture}
\end{center}
Hence $\sigma^u(A ,B)$ is fully switched
and
\begin{displaymath}
\left({}^AB,A_B\right) \approx \left({}^{A^{[\leq k]}}B, {A^{[\leq k]}}_B\right)
\quad \text{up to $(i,i)$, $(i,i+1)$ and $(i+1,i+1)$.}
\end{displaymath}
Let $v$ be the integer such that
$\Box(\sigma^v({}^{A}B, {A}_B)) = (i,i).$
We see that $\Box(\sigma^v({}^{A^{[\leq k]}}B, {A^{[\leq k]}}_B))$ $ = (i,i+1)$
and for $0 \leq r \leq v$,
$$
\sigma^r({}^AB,A_B) \approx \sigma^r({}^{A^{[\leq k]}}B, {A^{[\leq k]}}_B)
\quad \text{up to $(i,i)$, $(i,i+1)$ and $(i+1,i+1)$.}$$
Since (S2) and (S3) are applied to
$\Box(\sigma^{v}({}^{A^{[\leq k]}}B, {A^{[\leq k]}}_B))$
and
$\Box(\sigma^{v}({}^AB,A_B))$, respectively,
it follows that
$$
\sigma^{v+1}({}^AB,A_B) \approx \sigma^{v+1}({}^{A^{[\leq k]}}B, {A^{[\leq k]}}_B)
\quad \text{up to $(i,i)$.}$$
This implies that
whenever $r \geq v+1$, the switch to be applied to
$\sigma^{r}({}^AB,A_B)$
should be same to
that to be applied to
$\sigma^{r}({}^{A^{[\leq k]}}B, {A^{[\leq k]}}_B)$.
Thus we see that
$$
\Sigma\left({}^AB,A_B\right)\approx \Sigma\left({}^{A^{[\leq k]}}B, {A^{[\leq k]}}_B\right)
\quad \text{up to $(i,i)$}.
$$

When (S7) is applied to $\sigma^{u-1}(A,B)$,
then $\Box(\sigma^{u-1}(A,B)) = (i,i+1)$.
Moreover,
$(i+1,i+1)$ and $(i,i+2)$ in $\sigma^{u-1}(A,B)$ respectively
are filled with $\bfb$ and $b$.
We claim that $A^{[k]} = A_{i,i+1}$.
Suppose that $A^{[k]} \neq A_{i,i+1}$.
Then
$(i-1,i+1)$ is in the shifted switching path of $A^{[k]}$
and filled with $a$ in $\sigma^{u-2}(A, B)$.
Now let us investigate the $(i-1,i+2)$ box
in $\sigma^{u-2}(A, B)$.
If it is filled with $a$,
then $\Box(\sigma^{u-2}(A,B)) = (i-1,i+2)$
and thus $\Box(\sigma^{u-1}(A,B)) = (i-1,i+1)$ or $(i-1,i+3)$.
If it is filled with $b'$,
then the $(u-1)$st switch is (S1) and
$\Box(\sigma^{u-1}(A,B)) = (i-1,i+2)$.
This shows that
the $(i-1,i+2)$ box should be filled with $b$ in $\sigma^{u-2}(A, B)$,
thus in $\sigma^{u-1}(A, B)$.
It is obviously a contradiction to the fact that
the $(i,i+2)$ box in $\sigma^{u-1}(A, B)$ is filled with $b$.
Hence we verified the claim.
In addition, by using this property,
we see that
the $i$th column in $A\cup B$ has only a single box at $(i,i)$.
Refer to the following figure:
\vskip 2mm
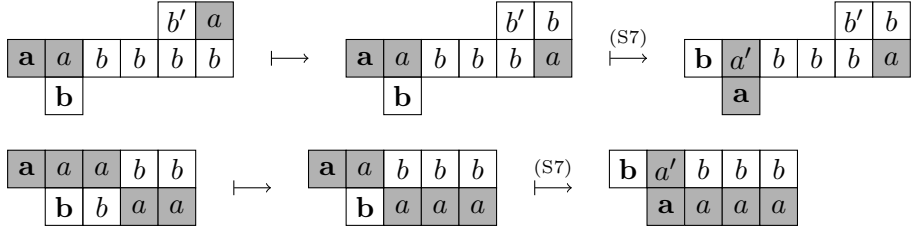
\begin{center}
\begin{tikzpicture}
\def \hhh{5mm}    
\def \vvv{5mm}    
\def \hhhhh{45mm}  
\draw[-] (\hhh*4,-\vvv*1) rectangle (\hhh*5,-\vvv*0);
\draw[fill=black!30] (\hhh*5,-\vvv*1) rectangle (\hhh*6,-\vvv*0);
\draw[fill=black!30] (\hhh*0,-\vvv*2) rectangle (\hhh*1,-\vvv*1);
\draw[fill=black!30] (\hhh*1,-\vvv*2) rectangle (\hhh*2,-\vvv*1);
\draw[-] (\hhh*2,-\vvv*2) rectangle (\hhh*3,-\vvv*1);
\draw[-] (\hhh*3,-\vvv*2) rectangle (\hhh*4,-\vvv*1);
\draw[-] (\hhh*4,-\vvv*2) rectangle (\hhh*5,-\vvv*1);
\draw[-] (\hhh*5,-\vvv*2) rectangle (\hhh*6,-\vvv*1);
\draw[-] (\hhh*1,-\vvv*3) rectangle (\hhh*2,-\vvv*2);
\node[] at (\hhh*4.5,-\vvv*0.5) {$b'$};
\node[] at (\hhh*5.5,-\vvv*0.5) {$a$};
\node[] at (\hhh*0.5,-\vvv*1.5) {$\bfa$};
\node[] at (\hhh*1.5,-\vvv*1.5) {$a$};
\node[] at (\hhh*2.5,-\vvv*1.5) {$b$};
\node[] at (\hhh*3.5,-\vvv*1.5) {$b$};
\node[] at (\hhh*4.5,-\vvv*1.5) {$b$};
\node[] at (\hhh*5.5,-\vvv*1.5) {$b$};
\node[] at (\hhh*1.5,-\vvv*2.5) {$\bfb$};
\draw[|->] (\hhh*7,-\vvv*1.5) -- (-\hhh*1+\hhhhh*1,-\vvv*1.5);
\draw[-] (\hhh*4+\hhhhh,-\vvv*1) rectangle (\hhh*5+\hhhhh,-\vvv*0);
\draw[-] (\hhh*5+\hhhhh,-\vvv*1) rectangle (\hhh*6+\hhhhh,-\vvv*0);
\draw[fill=black!30] (\hhh*0+\hhhhh,-\vvv*2) rectangle (\hhh*1+\hhhhh,-\vvv*1);
\draw[fill=black!30] (\hhh*1+\hhhhh,-\vvv*2) rectangle (\hhh*2+\hhhhh,-\vvv*1);
\draw[-] (\hhh*2+\hhhhh,-\vvv*2) rectangle (\hhh*3+\hhhhh,-\vvv*1);
\draw[-] (\hhh*3+\hhhhh,-\vvv*2) rectangle (\hhh*4+\hhhhh,-\vvv*1);
\draw[-] (\hhh*4+\hhhhh,-\vvv*2) rectangle (\hhh*5+\hhhhh,-\vvv*1);
\draw[fill=black!30] (\hhh*5+\hhhhh,-\vvv*2) rectangle (\hhh*6+\hhhhh,-\vvv*1);
\draw[-] (\hhh*1+\hhhhh,-\vvv*3) rectangle (\hhh*2+\hhhhh,-\vvv*2);
\node[] at (\hhh*4.5+\hhhhh,-\vvv*0.5) {$b'$};
\node[] at (\hhh*5.5+\hhhhh,-\vvv*0.5) {$b$};
\node[] at (\hhh*0.5+\hhhhh,-\vvv*1.5) {$\bfa$};
\node[] at (\hhh*1.5+\hhhhh,-\vvv*1.5) {$a$};
\node[] at (\hhh*2.5+\hhhhh,-\vvv*1.5) {$b$};
\node[] at (\hhh*3.5+\hhhhh,-\vvv*1.5) {$b$};
\node[] at (\hhh*4.5+\hhhhh,-\vvv*1.5) {$b$};
\node[] at (\hhh*5.5+\hhhhh,-\vvv*1.5) {$a$};
\node[] at (\hhh*1.5+\hhhhh,-\vvv*2.5) {$\bfb$};
\draw[|->] (\hhh*7+\hhhhh,-\vvv*1.5) -- (-\hhh*1+\hhhhh*2,-\vvv*1.5)
node[above,midway] {\tiny (S7)};
\draw[-] (\hhh*4+\hhhhh*2,-\vvv*1) rectangle (\hhh*5+\hhhhh*2,-\vvv*0);
\draw[-] (\hhh*5+\hhhhh*2,-\vvv*1) rectangle (\hhh*6+\hhhhh*2,-\vvv*0);
\draw[-] (\hhh*0+\hhhhh*2,-\vvv*2) rectangle (\hhh*1+\hhhhh*2,-\vvv*1);
\draw[fill=black!30] (\hhh*1+\hhhhh*2,-\vvv*2) rectangle (\hhh*2+\hhhhh*2,-\vvv*1);
\draw[-] (\hhh*2+\hhhhh*2,-\vvv*2) rectangle (\hhh*3+\hhhhh*2,-\vvv*1);
\draw[-] (\hhh*3+\hhhhh*2,-\vvv*2) rectangle (\hhh*4+\hhhhh*2,-\vvv*1);
\draw[-] (\hhh*4+\hhhhh*2,-\vvv*2) rectangle (\hhh*5+\hhhhh*2,-\vvv*1);
\draw[fill=black!30] (\hhh*5+\hhhhh*2,-\vvv*2) rectangle (\hhh*6+\hhhhh*2,-\vvv*1);
\draw[fill=black!30] (\hhh*1+\hhhhh*2,-\vvv*3) rectangle (\hhh*2+\hhhhh*2,-\vvv*2);
\node[] at (\hhh*4.5+\hhhhh*2,-\vvv*0.5) {$b'$};
\node[] at (\hhh*5.5+\hhhhh*2,-\vvv*0.5) {$b$};
\node[] at (\hhh*0.5+\hhhhh*2,-\vvv*1.5) {$\bfb$};
\node[] at (\hhh*1.5+\hhhhh*2,-\vvv*1.5) {$a'$};
\node[] at (\hhh*2.5+\hhhhh*2,-\vvv*1.5) {$b$};
\node[] at (\hhh*3.5+\hhhhh*2,-\vvv*1.5) {$b$};
\node[] at (\hhh*4.5+\hhhhh*2,-\vvv*1.5) {$b$};
\node[] at (\hhh*5.5+\hhhhh*2,-\vvv*1.5) {$a$};
\node[] at (\hhh*1.5+\hhhhh*2,-\vvv*2.5) {$\bfa$};
\end{tikzpicture}
\vskip 4mm
\begin{tikzpicture}
\def \hhh{5mm}    
\def \vvv{5mm}    
\def \hhhhh{40mm}  
\draw[fill=black!30] (\hhh*0,-\vvv*2) rectangle (\hhh*1,-\vvv*1);
\draw[fill=black!30] (\hhh*1,-\vvv*2) rectangle (\hhh*2,-\vvv*1);
\draw[fill=black!30] (\hhh*2,-\vvv*2) rectangle (\hhh*3,-\vvv*1);
\draw[-] (\hhh*3,-\vvv*2) rectangle (\hhh*4,-\vvv*1);
\draw[-] (\hhh*4,-\vvv*2) rectangle (\hhh*5,-\vvv*1);
\draw[-] (\hhh*1,-\vvv*3) rectangle (\hhh*2,-\vvv*2);
\draw[-] (\hhh*2,-\vvv*3) rectangle (\hhh*3,-\vvv*2);
\draw[fill=black!30] (\hhh*3,-\vvv*3) rectangle (\hhh*4,-\vvv*2);
\draw[fill=black!30] (\hhh*4,-\vvv*3) rectangle (\hhh*5,-\vvv*2);
\node[] at (\hhh*0.5,-\vvv*1.5) {$\bfa$};
\node[] at (\hhh*1.5,-\vvv*1.5) {$a$};
\node[] at (\hhh*2.5,-\vvv*1.5) {$a$};
\node[] at (\hhh*3.5,-\vvv*1.5) {$b$};
\node[] at (\hhh*4.5,-\vvv*1.5) {$b$};
\node[] at (\hhh*1.5,-\vvv*2.5) {$\bfb$};
\node[] at (\hhh*2.5,-\vvv*2.5) {$b$};
\node[] at (\hhh*3.5,-\vvv*2.5) {$a$};
\node[] at (\hhh*4.5,-\vvv*2.5) {$a$};
\draw[|->] (\hhh*6,-\vvv*2) -- (-\hhh*1+\hhhhh*1,-\vvv*2);
\draw[fill=black!30] (\hhh*0+\hhhhh,-\vvv*2) rectangle (\hhh*1+\hhhhh,-\vvv*1);
\draw[fill=black!30] (\hhh*1+\hhhhh,-\vvv*2) rectangle (\hhh*2+\hhhhh,-\vvv*1);
\draw[-] (\hhh*2+\hhhhh,-\vvv*2) rectangle (\hhh*3+\hhhhh,-\vvv*1);
\draw[-] (\hhh*3+\hhhhh,-\vvv*2) rectangle (\hhh*4+\hhhhh,-\vvv*1);
\draw[-] (\hhh*4+\hhhhh,-\vvv*2) rectangle (\hhh*5+\hhhhh,-\vvv*1);
\draw[-] (\hhh*1+\hhhhh,-\vvv*3) rectangle (\hhh*2+\hhhhh,-\vvv*2);
\draw[fill=black!30] (\hhh*2+\hhhhh,-\vvv*3) rectangle (\hhh*3+\hhhhh,-\vvv*2);
\draw[fill=black!30] (\hhh*3+\hhhhh,-\vvv*3) rectangle (\hhh*4+\hhhhh,-\vvv*2);
\draw[fill=black!30] (\hhh*4+\hhhhh,-\vvv*3) rectangle (\hhh*5+\hhhhh,-\vvv*2);
\node[] at (\hhh*0.5+\hhhhh,-\vvv*1.5) {$\bfa$};
\node[] at (\hhh*1.5+\hhhhh,-\vvv*1.5) {$a$};
\node[] at (\hhh*2.5+\hhhhh,-\vvv*1.5) {$b$};
\node[] at (\hhh*3.5+\hhhhh,-\vvv*1.5) {$b$};
\node[] at (\hhh*4.5+\hhhhh,-\vvv*1.5) {$b$};
\node[] at (\hhh*1.5+\hhhhh,-\vvv*2.5) {$\bfb$};
\node[] at (\hhh*2.5+\hhhhh,-\vvv*2.5) {$a$};
\node[] at (\hhh*3.5+\hhhhh,-\vvv*2.5) {$a$};
\node[] at (\hhh*4.5+\hhhhh,-\vvv*2.5) {$a$};
\draw[|->] (\hhh*6+\hhhhh*1,-\vvv*2) -- (-\hhh*1+\hhhhh*2,-\vvv*2)
node[above,midway] {\tiny (S7)};
\draw[-] (\hhh*0+\hhhhh*2,-\vvv*2) rectangle (\hhh*1+\hhhhh*2,-\vvv*1);
\draw[fill=black!30] (\hhh*1+\hhhhh*2,-\vvv*2) rectangle (\hhh*2+\hhhhh*2,-\vvv*1);
\draw[-] (\hhh*2+\hhhhh*2,-\vvv*2) rectangle (\hhh*3+\hhhhh*2,-\vvv*1);
\draw[-] (\hhh*3+\hhhhh*2,-\vvv*2) rectangle (\hhh*4+\hhhhh*2,-\vvv*1);
\draw[-] (\hhh*4+\hhhhh*2,-\vvv*2) rectangle (\hhh*5+\hhhhh*2,-\vvv*1);
\draw[fill=black!30] (\hhh*1+\hhhhh*2,-\vvv*3) rectangle (\hhh*2+\hhhhh*2,-\vvv*2);
\draw[fill=black!30] (\hhh*2+\hhhhh*2,-\vvv*3) rectangle (\hhh*3+\hhhhh*2,-\vvv*2);
\draw[fill=black!30] (\hhh*3+\hhhhh*2,-\vvv*3) rectangle (\hhh*4+\hhhhh*2,-\vvv*2);
\draw[fill=black!30] (\hhh*4+\hhhhh*2,-\vvv*3) rectangle (\hhh*5+\hhhhh*2,-\vvv*2);
\draw[black!0] (\hhh*7+\hhhhh*2,-\vvv*2) rectangle (\hhh*8+\hhhhh*2,-\vvv*1);
\node[] at (\hhh*0.5+\hhhhh*2,-\vvv*1.5) {$\bfb$};
\node[] at (\hhh*1.5+\hhhhh*2,-\vvv*1.5) {$a'$};
\node[] at (\hhh*2.5+\hhhhh*2,-\vvv*1.5) {$b$};
\node[] at (\hhh*3.5+\hhhhh*2,-\vvv*1.5) {$b$};
\node[] at (\hhh*4.5+\hhhhh*2,-\vvv*1.5) {$b$};
\node[] at (\hhh*1.5+\hhhhh*2,-\vvv*2.5) {$\bfa$};
\node[] at (\hhh*2.5+\hhhhh*2,-\vvv*2.5) {$a$};
\node[] at (\hhh*3.5+\hhhhh*2,-\vvv*2.5) {$a$};
\node[] at (\hhh*4.5+\hhhhh*2,-\vvv*2.5) {$a$};
\end{tikzpicture}
\end{center}
\vskip 2mm
Suppose that $\overrightarrow{\rm P}_{(i,i)} = \{(i,i),\ldots,(i,i+y)\}$
for some $y \geq 2$.
Then $\sigma^{u+y-1}(A,B)$ is fully switched and
$$
({}^AB,A_B) \approx ({}^{A^{[\leq k]}}B,{A^{[\leq k]}}_B)
\quad \text{up to $(i,i),\ldots,(i,i+y)$ and $(i+1,i+1)$.}$$
For $1 \leq s \leq y$,
let $v_s$ be the integer determined by
the condition
$$
\Box(\sigma^{v_s}({}^{A}B,{A}_B)) = (i,i+y-s).
$$
Then, for $1\leq s \leq y-1$ and $v_s \leq r \leq v_{s+1}-1$,
$$
\sigma^{r+1}({}^AB,A_B) \approx \sigma^{r+1-s}({}^{A^{[\leq k]}}B,{A^{[\leq k]}}_B)
\text{ up to $(i,i),\ldots,(i,i+y-s)$ and $(i+1,i+1)$.}$$
Applying $\sigma$ to the case where $s=y-1$ and $r=v_y-1$,
one can see that
\begin{displaymath}
\sigma^{v_{y}+1}({}^AB,A_B) \approx \sigma^{v_y-y+2}({}^{A^{[\leq k]}}B,{A^{[\leq k]}}_B)
\quad \text{up to $(i,i)$.}
\end{displaymath}
This equivalence implies that for each $r \geq 1$,
the switch to be applied to $\sigma^{v_y+r}({}^AB,A_B)$
is equal to that to be applied to $\sigma^{v_y-y+r+1}({}^{A^{[\leq k]}}B, {A^{[\leq k]}}_B)$.
Thus we see that
\begin{displaymath}
\Sigma\left({}^AB,A_B\right)\approx \Sigma\left({}^{A^{[\leq k]}}B, {A^{[\leq k]}}_B\right)
\quad \text{up to $(i,i)$}.
\end{displaymath}
\end{proof}

\vskip 2mm
\begin{lem}\label{lem-sw-involution}
Let $(A,B)$ be a shifted perforated $(\bfa,\bfb)$-pair
such that $B$ extends $A$. Then $\Sigma({}^AB, A_B) = (A ,B).$
\end{lem}
\begin{proof}
We prove the assertion by induction on the number of boxes in $A$.
To begin with, consider the case where $A$ consists of a single box.
Let $(i,j)$ be its position.
As in the proof of Lemma \ref{lem1-sw-involution}, the type of $\overrightarrow{\rm P}_{(i,j)}$ is one of {\bf T1}, {\bf T2} and {\bf T3}.
Let $u$ be the smallest integer such that $\sigma^u(A,B)$ is fully switched.

In case of ${\bf T1}$,
when the $(i+1,j)$ box is filled with $\bfb$ in $B$,
the $(i,j+1)$ box is filled with $b'$ and
all boxes $(i,j+2)$ through $(i,j+y)$
should be filled with $b$.
On the other hand, when the $(i+1,j)$ box is not in $B$,
the $(i,j+1)$ box is filled with $\bfb$ and
all boxes $(i,j+2)$ through $(i,j+y)$
should be filled with $b$.
In either case, since the $(i-1,j+y)$ box is a $b'$-box or an empty box, it follows that
$\sigma^i({}^AB,A_B) = \sigma^{u-i}(A,B)$ for all $1 \leq i \leq u$.
Thus $\sigma^u({}^AB, A_B) = (A, B)$.

In case of ${\bf T2}$,
all boxes $(i+1,j)$ through $(i+x-1,j)$ are filled with $b'$,
whereas the $(i+x,j)$ box is filled with $\bfb$.
When the $(i+x,j-1)$ box is filled with $\bfb$ in $B$, the $(i+x,j)$ box is filled with $b$.
On the other hand, when the $(i+x,j-1)$ box is not in $B$, it is filled with $\bfb$.
In either case, since $\Box({}^AB,A_B) = (i+x-1,j)$,
it follows that
$\sigma^i({}^AB,A_B) = \sigma^{u-i}(A,B)$ for all $1 \leq i \leq u$.
Thus $\sigma^u({}^AB, A_B) = (A, B)$.

Finally, one can easily see that the ${\bf T3}$ case
can occur only in the following subfillings:
\vskip 2mm
\begin{center}
\begin{tikzpicture}
\def \hhh{5mm}    
\def \vvv{5mm}    
\def \hhhhh{18mm}  
\draw[fill=black!30] (\hhh*0,0) rectangle (\hhh*1,\vvv*1);
\draw[-] (\hhh*1,0) rectangle (\hhh*2,\vvv*1);
\draw[-] (\hhh*1,-\vvv*1) rectangle (\hhh*2,\vvv*0);
\node[] at (\hhh*0.5,\vvv*0.5) {$\bfa$};
\node[] at (\hhh*1.5,\vvv*0.5) {$b'$};
\node[] at (\hhh*1.5,-\vvv*0.5) {$b$};
\draw[|->] (\hhh*2.5,-\vvv*0) -- (-\hhh*0.5+\hhhhh,-\vvv*0);
\draw[-] (\hhh*0+\hhhhh,0) rectangle (\hhh*1+\hhhhh,\vvv*1);
\draw[-] (\hhh*1+\hhhhh,0) rectangle (\hhh*2+\hhhhh,\vvv*1);
\draw[fill=black!30] (\hhh*1+\hhhhh,-\vvv*1) rectangle (\hhh*2+\hhhhh,\vvv*0);
\node[] at (\hhh*0.5+\hhhhh,\vvv*0.5) {$b$};
\node[] at (\hhh*1.5+\hhhhh,\vvv*0.5) {$b$};
\node[] at (\hhh*1.5+\hhhhh,-\vvv*0.5) {$\bfa$};
\end{tikzpicture}
\hskip 7mm
\begin{tikzpicture}
\def \hhh{5mm}    
\def \vvv{5mm}    
\def \hhhhh{18mm}  
\draw[fill=black!30] (\hhh*0,0) rectangle (\hhh*1,\vvv*1);
\draw[-] (\hhh*1,0) rectangle (\hhh*2,\vvv*1);
\draw[-] (\hhh*1,-\vvv*1) rectangle (\hhh*2,\vvv*0);
\node[] at (\hhh*0.5,\vvv*0.5) {$\bfa$};
\node[] at (\hhh*1.5,\vvv*0.5) {$b'$};
\node[] at (\hhh*1.5,-\vvv*0.5) {$b'$};
\draw[|->] (\hhh*2.5,-\vvv*0) -- (-\hhh*0.5+\hhhhh,-\vvv*0);
\draw[-] (\hhh*0+\hhhhh,0) rectangle (\hhh*1+\hhhhh,\vvv*1);
\draw[-] (\hhh*1+\hhhhh,0) rectangle (\hhh*2+\hhhhh,\vvv*1);
\draw[fill=black!30] (\hhh*1+\hhhhh,-\vvv*1) rectangle (\hhh*2+\hhhhh,\vvv*0);
\node[] at (\hhh*0.5+\hhhhh,\vvv*0.5) {$b'$};
\node[] at (\hhh*1.5+\hhhhh,\vvv*0.5) {$b$};
\node[] at (\hhh*1.5+\hhhhh,-\vvv*0.5) {$\bfa$};
\end{tikzpicture}
\hskip 7mm
\begin{tikzpicture}
\def \hhh{5mm}    
\def \vvv{5mm}    
\def \hhhhh{18mm}  
\draw[fill=black!30] (\hhh*0,0) rectangle (\hhh*1,\vvv*1);
\draw[-] (\hhh*1,0) rectangle (\hhh*2,\vvv*1);
\draw[-] (\hhh*0,-\vvv*1) rectangle (\hhh*2,\vvv*0);
\draw[-] (\hhh*1,-\vvv*1) rectangle (\hhh*2,\vvv*0);
\node[] at (\hhh*0.5,\vvv*0.5) {$\bfa$};
\node[] at (\hhh*1.5,\vvv*0.5) {$b'$};
\node[] at (\hhh*0.5,-\vvv*0.5) {$b$};
\node[] at (\hhh*1.5,-\vvv*0.5) {$b$};
\draw[->,decorate,decoration={snake,amplitude=.4mm,segment length=2mm,post length=1mm}] (\hhh*2.5,-\vvv*0) -- (-\hhh*0.5+\hhhhh,-\vvv*0);
\draw[-] (\hhh*0+\hhhhh,0) rectangle (\hhh*1+\hhhhh,\vvv*1);
\draw[-] (\hhh*1+\hhhhh,0) rectangle (\hhh*2+\hhhhh,\vvv*1);
\draw[-] (\hhh*0+\hhhhh,-\vvv*1) rectangle (\hhh*1+\hhhhh,\vvv*0);
\draw[fill=black!30] (\hhh*1+\hhhhh,-\vvv*1) rectangle (\hhh*2+\hhhhh,\vvv*0);
\node[] at (\hhh*0.5+\hhhhh,\vvv*0.5) {$b'$};
\node[] at (\hhh*1.5+\hhhhh,\vvv*0.5) {$b$};
\node[] at (\hhh*0.5+\hhhhh,-\vvv*0.5) {$b$};
\node[] at (\hhh*1.5+\hhhhh,-\vvv*0.5) {$\bfa$};
\end{tikzpicture}
\hskip 7mm
\begin{tikzpicture}
\def \hhh{5mm}    
\def \vvv{5mm}    
\def \hhhhh{18mm}  
\draw[fill=black!30] (\hhh*0,0) rectangle (\hhh*1,\vvv*1);
\draw[-] (\hhh*1,0) rectangle (\hhh*2,\vvv*1);
\draw[-] (\hhh*0,-\vvv*1) rectangle (\hhh*2,\vvv*0);
\draw[-] (\hhh*1,-\vvv*1) rectangle (\hhh*2,\vvv*0);
\node[] at (\hhh*0.5,\vvv*0.5) {$\bfa$};
\node[] at (\hhh*1.5,\vvv*0.5) {$b'$};
\node[] at (\hhh*0.5,-\vvv*0.5) {$b'$};
\node[] at (\hhh*1.5,-\vvv*0.5) {$b$};
\draw[->,decorate,decoration={snake,amplitude=.4mm,segment length=2mm,post length=1mm}] (\hhh*2.5,-\vvv*0) -- (-\hhh*0.5+\hhhhh,-\vvv*0);
\draw[-] (\hhh*0+\hhhhh,0) rectangle (\hhh*1+\hhhhh,\vvv*1);
\draw[-] (\hhh*1+\hhhhh,0) rectangle (\hhh*2+\hhhhh,\vvv*1);
\draw[-] (\hhh*0+\hhhhh,-\vvv*1) rectangle (\hhh*1+\hhhhh,\vvv*0);
\draw[fill=black!30] (\hhh*1+\hhhhh,-\vvv*1) rectangle (\hhh*2+\hhhhh,\vvv*0);
\node[] at (\hhh*0.5+\hhhhh,\vvv*0.5) {$b'$};
\node[] at (\hhh*1.5+\hhhhh,\vvv*0.5) {$b$};
\node[] at (\hhh*0.5+\hhhhh,-\vvv*0.5) {$b'$};
\node[] at (\hhh*1.5+\hhhhh,-\vvv*0.5) {$\bfa$};
\end{tikzpicture}
\end{center}
For all subfillings above, it is straightforward to see that
$\Sigma({}^AB, A_B) = (A, B)$.

We now assume that the desired result holds for all SYTs of border strip shape
with $k$ boxes.
Suppose that $A$ has $k+1$ boxes.
Letting $A^{[k+1]} = A_{i,j}$,
by Lemma \ref{lem1-sw-involution} and Lemma \ref{lem2-sw-involution}
we have
$$
\Sigma\left({}^AB,A_B\right) \approx \Sigma\left({}^{A^{[\leq k]}}B, {A^{[\leq k]}}_B\right)
\quad
\text{up to $(i,j)$.}$$
Moreover, from induction hypothesis it follows that
$$
\Sigma \left({}^{A^{[\leq k]}}B, {A^{[\leq k]}}_B\right)
=\left(A^{[\leq k]}, B\right).
$$
It follows that
$$
\begin{aligned}
\Sigma\left({}^AB,A_B\right) &= A^{[k+1]} \cup \Sigma\left({}^{A^{[\leq k]}}B, {A^{[\leq k]}}_B\right) \\
&= A^{[k+1]} \cup \left(A^{[\leq k]}, B\right) \\ &= (A,B) \, ,
\end{aligned}
$$
as required.
\end{proof}

\vskip 2mm
\begin{lem}\label{lem-BKT}{\rm (\cite[Lemma 1]{BKT})}
Let two words $u$ and $v$ be shifted Knuth equivalent.
Then $u$ is an LRS word if and only if $v$ is an LRS word.
\end{lem}

\vskip 2mm
Now we are ready to provide the proof of Theorem \ref{thm-main-section4}.
\vskip 2mm \noindent
\textbf {Proof of Theorem \ref{thm-main-section4}.}
(a)
Lemma \ref{lem-sw-involution} says that
$\Sigma \circ \Sigma\left(A \cup B\right) = A \cup B$
for any perforated $(\bfa,\bfb)$-pair $(A,B)$ such that $B$ extends $A$.
So the assertion follows from Remark \ref{Remark-Diff2}(ii).

(b) Let $T$ be an LRS tableau and $T'$ any filling obtained by
a sequence of shifted slides to $T$.
By applying Lemma \ref{lem-BKT} in iteration, we can deduce that
$T$ is an LRS tableau if and only if $\w(T')$ is an LRS word.
In particular, by Remark \ref{remark-switch=jdt}(i),
$\w({}^ST)$ is an LRS word.
Consequently ${}^ST$ is an LRS tableau.

(c) Let $S$ be an LRS tableau.
By (b), $S_T$ is an LRS tableau if and only if ${}^{{}^{{}^ST}}S_T$
is an LRS tableau.
Now our assertion is straightforward since ${}^{{}^{{}^ST}}S_T=S$
by (a).
\qed

\begin{example}
\label{example-section4}
Let $\lambda, \mu, \nu \in \Lambda^+$ with $\lambda,\mu \subseteq \nu$.

(a) (The shifted Littlewood-Richardson rule)
By Theorem \ref{thm-main-section4} we obtain a bijection
\begin{displaymath}
\Sigma: \mathcal{F}^\ld_{\emptyset \ld} \times \widetilde{\mathcal{Y}}(\nu/\lambda)
\rightarrow
\bigcup_{\mu \subseteq \nu} \widetilde{\mathcal{Y}}(\mu) \times
\mathcal{F}^\nu_{\mu \lambda},
\quad
(R_\lambda, T) \mapsto \left({}^{R_\lambda}T, ({R_\lambda})_T\right).
\end{displaymath}
This immediately gives a combinatorial interpretation of
$Q_{\nu/\lambda}(x) = \sum_{\mu} f^{\nu}_{\mu \lambda} Q_\mu(x)$.

(b) (Symmetry of the shifted Littlewood-Richardson coefficients)
By Theorem \ref{thm-main-section4} we obtain a bijection
\begin{displaymath}
\Sigma :
\mathcal{F}^\ld_{\emptyset \ld} \times \mathcal{F}^\nu_{\lambda \mu}
\rightarrow
\mathcal{F}^\mu_{\emptyset \mu} \times \mathcal{F}^\nu_{\mu \lambda},
\quad
(R_\lambda,T) \mapsto \left(R_\mu,({R_\lambda})_T\right).
\end{displaymath}
Hence, not only have we proven $f^{\nu}_{\ld \mu} = f^{\nu}_{\mu \ld},$
we have constructed an explicit involution that interchanges the inner shape with
the weight of an LRS tableau.
\end{example}

\section{Worley's shifted $J$-operation and the shifted Generalized Evacuation}
\label{sect-Worley-oper-J}
Worley \cite{Wo} introduced a shifted analogue of
Sch\"{u}tzenberger's $J$-operation \cite{Sc2}
in order to explain the identities related with shifted Littlewood-Richardson coefficients
in a combinatorial way.
The purpose of this section is to demonstrate the relationship between
Worley's shifted $J$-operation and our shifted tableau switching.

Let $-X$ be the set $\{\cdots < -2' < -2 < -1' < -1\}$  of ordered alphabets
and
$\overline{X}$ be the set
$\{\cdots < -1' < -1 < 0' < 0 < 1' < 1 < \cdots\}$.
Throughout this section,
we use the following notation:
\begin{itemize}
\item $\widetilde{\mathcal{Y}}_{-}(\ld/\mu)$: the set of all SYTs of shape $\lambda/\mu$ with letters in $-X$

\item $\widetilde{\mathcal{Y}}_{\pm}(\ld/\mu)$: the set of all SYTs of shape $\lambda/\mu$ with letters in $\overline{X}$

\end{itemize}
It should be particularly noticed that, in the present section,
an SYT always denotes an element in $\widetilde{\mathcal{Y}}_{\pm}(\ld/\mu)$.
Recall that $T \in \widetilde{\mathcal{Y}}_{\pm}(\lambda/\mu)$ is of normal shape
if $\mu = \emptyset.$
For an unprimed letter $a\in \overline{X}$,
set $a^*:= -a'$ and $({a'})^* := -a$.
Conventionally, $-(-\bfa)$ is set to be $\bfa$.

\begin{df}\label{def-*}(\cite{Wo})
Given $\lambda,\mu \in \Lambda^+$ with $\mu \subset \lambda$,
let $T$ be an SYT of skew shape $\lambda/\mu$.
Then $T^*$ is defined by the SYT obtained by
transporting each box $(i,j)$ in $T$ to $(\lambda_1-j+1,\lambda_1-i+1)$
and then by filling it with $(T_{i,j})^*$.
Equivalently,
${\rm sh}(T^*)=\{(\lambda_1-j+1,\lambda_1-i+1):(i,j) \in {\rm sh}(T)\}$
and $(T^*)_{\lambda_1-j+1,\lambda_1-i+1}=(T_{i,j})^*$.
\end{df}

Given $\lambda,\mu \in \Lambda^+$ with $\mu \subset \lambda$,
let us consider the map
\begin{displaymath}
J: \widetilde{\mathcal{Y}}_{\pm}(\lambda/\mu) \longrightarrow
\bigcup_{\nu \in \Lambda^+}\widetilde{\mathcal{Y}}_{\pm}(\nu),
\quad
T \mapsto \Rect(T^*).
\end{displaymath}
For simplicity, we write $T^J$ for $J(T)$.
Then the following result is well known.
\begin{lem}{\rm (\cite[Lemma 7.1.4]{Wo})}\label{lem-property-J}
Let $T$ be an SYT. Then we have the following.
\begin{itemize}
\item[(a)] If $T$ is of normal shape, then $\sh(T^J) = \sh(T)$
and $(T^J)^J = T$.

\item[(b)]
If $T$ is of skew shape, then
$\w(\Rect(T^*)) \equiv_{\rm \tiny SK} \w(\Rect(T)^*)$,
that is, $T^J   =  (\Rect(T))^J$.
\end{itemize}
\end{lem}
For instance, if
\begin{center}
\begin{tikzpicture}
\def \hhh{5mm}    
\def \vvv{5mm}    
\node[] at (-\hhh*1,-\vvv*0.4) {$T=$};
\draw[-] (\hhh*0,0) rectangle (\hhh*1,\vvv*1);
\draw[-] (\hhh*1,0) rectangle (\hhh*2,\vvv*1);
\draw[-] (\hhh*2,0) rectangle (\hhh*3,\vvv*1);
\draw[-] (\hhh*3,0) rectangle (\hhh*4,\vvv*1);
\draw[-] (\hhh*1,-\vvv*1) rectangle (\hhh*2,\vvv*0);
\draw[-] (\hhh*2,-\vvv*1) rectangle (\hhh*3,\vvv*0);
\draw[-] (\hhh*2,-\vvv*2) rectangle (\hhh*3,-\vvv*1);
\node[] at (\hhh*0.5,\vvv*0.5) {$1'$};
\node[] at (\hhh*1.5,\vvv*0.5) {$1$};
\node[] at (\hhh*2.5,\vvv*0.5) {$1$};
\node[] at (\hhh*3.5,\vvv*0.5) {$2$};
\node[] at (\hhh*1.5,-\vvv*0.5) {$2$};
\node[] at (\hhh*2.5,-\vvv*0.5) {$3'$};
\node[] at (\hhh*2.5,-\vvv*1.5) {$3$};
\end{tikzpicture},
\end{center}
then one can easily see that
\vskip 2mm
\begin{center}
\begin{tikzpicture}
\def \hhh{5.6mm}    
\def \vvv{5.3mm}    
\def \hhhhh{60mm}  
\node[] at (-\hhh*1,-\vvv*0.5) {$T^*=$};
\draw[black!20] (\hhh*0,0) rectangle (\hhh*1,\vvv*1);
\draw[black!20] (\hhh*1,0) rectangle (\hhh*2,\vvv*1);
\draw[black!20] (\hhh*2,0) rectangle (\hhh*3,\vvv*1);
\draw[-] (\hhh*3,0) rectangle (\hhh*4,\vvv*1);
\draw[-] (\hhh*1,-\vvv*1) rectangle (\hhh*2,\vvv*0);
\draw[-] (\hhh*2,-\vvv*1) rectangle (\hhh*3,\vvv*0);
\draw[-] (\hhh*3,-\vvv*1) rectangle (\hhh*4,\vvv*0);
\draw[-] (\hhh*2,-\vvv*2) rectangle (\hhh*3,-\vvv*1);
\draw[-] (\hhh*3,-\vvv*2) rectangle (\hhh*4,-\vvv*1);
\draw[-] (\hhh*3,-\vvv*3) rectangle (\hhh*4,-\vvv*2);
\node[] at (\hhh*3.5,\vvv*0.5) {\footnotesize $-2'$};
\node[] at (\hhh*1.5,-\vvv*0.5) {\footnotesize $-3'$};
\node[] at (\hhh*2.5,-\vvv*0.5) {\footnotesize $-3$};
\node[] at (\hhh*3.5,-\vvv*0.5) {\footnotesize $-1'$};
\node[] at (\hhh*2.5,-\vvv*1.5) {\footnotesize $-2'$};
\node[] at (\hhh*3.5,-\vvv*1.5) {\footnotesize $-1'$};
\node[] at (\hhh*3.5,-\vvv*2.5) {\footnotesize $-1$};
\node[] at (\hhh*1.3+\hhhhh*0.5,-\vvv*0.5) {and};
\node[] at (-\hhh*1.1+\hhhhh,-\vvv*0.5) {$T^J=$};
\draw[-] (\hhh*0+\hhhhh,-\vvv*0) rectangle (\hhh*1+\hhhhh,\vvv*1);
\draw[-] (\hhh*1+\hhhhh,-\vvv*0) rectangle (\hhh*2+\hhhhh,\vvv*1);
\draw[-] (\hhh*2+\hhhhh,-\vvv*0) rectangle (\hhh*3+\hhhhh,\vvv*1);
\draw[-] (\hhh*3+\hhhhh,-\vvv*0) rectangle (\hhh*4+\hhhhh,\vvv*1);
\draw[-] (\hhh*1+\hhhhh,-\vvv*1) rectangle (\hhh*2+\hhhhh,-\vvv*0);
\draw[-] (\hhh*2+\hhhhh,-\vvv*1) rectangle (\hhh*3+\hhhhh,-\vvv*0);
%
\draw[-] (\hhh*2+\hhhhh,-\vvv*2) rectangle (\hhh*3+\hhhhh,-\vvv*1);
\node[] at (\hhh*0.5+\hhhhh,\vvv*0.5) {\footnotesize $-3'$};
\node[] at (\hhh*1.5+\hhhhh,\vvv*0.5) {\footnotesize $-3$};
\node[] at (\hhh*2.5+\hhhhh,\vvv*0.5) {\footnotesize $-2'$};
\node[] at (\hhh*3.5+\hhhhh,\vvv*0.5) {\footnotesize $-1$};
\node[] at (\hhh*1.5+\hhhhh,-\vvv*0.5) {\footnotesize $-2'$};
\node[] at (\hhh*2.5+\hhhhh,-\vvv*0.5) {\footnotesize $-1'$};
\node[] at (\hhh*2.5+\hhhhh,-\vvv*1.5) {\footnotesize $-1$};
\end{tikzpicture}.
\end{center}

In Example \ref{example-section4}(b)
we showed
$f^\nu_{\lambda \mu} = f^\nu_{\mu \lambda}$
by using the shifted tableau switching.
Worley had, however, obtained this result much earlier
by using the shifted $J$-operation.
His argument can be restated in the following manner.
\begin{thm}{\rm (\cite[Theorem 7.2.2]{Wo})}
\label{thm-wo-symmtries}
For each $T \in \mathcal{F}_{\lambda \mu}^\nu$
define $\rho(T)$ to be the SYT uniquely determined by the condition that
$((R_\lambda)^J \cup T)^J = (R_\mu)^J \cup \rho(T).$
Then the function
$$
\rho:
\mathcal{F}^\nu_{\lambda \mu} \rightarrow \mathcal{F}^\nu_{\mu \lambda},
\
T \mapsto \rho(T)
$$
is a bijection.
\end{thm}

Hence it would be very natural to search for
the relationship between Worley's shifted $J$-operation and
our shifted tableau switching.
For this purpose, we first provide a shifted analogue of
\cite[Algorithm 5.2]{BSS},
which is called the {\it generalized evacuation}
there.

\begin{algo}
\label{algorithm3}
(The shifted generalized evacuation) \
Let $S$ be an SYT of normal shape.
\begin{description}
\item[Step 1]
Start with $S$ whose boxes are all unshaded.
Set $S^E:= \emptyset$ and $m:=\max(S)$.

\item[Step 2]
Set $a:=\min(S).$

\item[Step 3] Do one of the following on $S^{(a)}$,
where $S = S^{(a)} \cup S^{(a+1)} \cup \cdots \cup S^{(m)}$.
\vskip 1mm
\begin{itemize}
\setlength{\itemindent}{-5mm}
\item If $S_{1,1} = a$,
then replace it with $-a'$
and the entries in $S^{(a)}$ except $S_{1,1}$ with $-a$'s.

\item
If $S_{1,1} = a'$,
then replace it with $-a$
and the entries in $S^{(a)}$ except $S_{1,1}$ with $-a$'s.
\end{itemize}
\vskip 1mm
Set $S^{(a)}$ newly to be the resulting SYT of border strip shape.
Then shade all the boxes in $S^{(a)}$.

\item[Step 4] If $a = m$, then put $S^E := S^{(a)} \cup S^E$ and finish the algorithm.
If $a \neq m$, then apply Algorithm \ref{algorithm2} to
$(S^{(a)}, S^{(a+1)} \cup \cdots \cup S^{(m)})$
and then go to {\bf Step 5}.

\item[Step 5] Set $S:= {}^{S^{(a)}}(S^{(a+1)} \cup \cdots \cup S^{(m)})$
and
$S^E:= (S^{(a)})_{S^{(a+1)} \cup \cdots \cup S^{(m)}} \cup S^E$.
Then go to {\bf Step 2}.
\end{description}
\end{algo}

From Theorem \ref{Thm-Sec3-SW} it follows that $S^E$ is an SYT.
As an illustration, for the SYT $T$ following Lemma \ref{lem-property-J},
$T^E$ can be obtained as follows:
\vskip 2mm
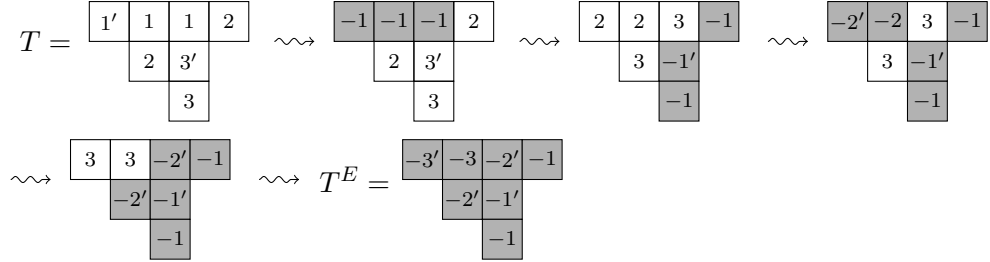
\begin{tikzpicture}
\def \hhh{5.3mm}    
\def \vvv{5.3mm}    
\node[] at (-\hhh*1.1,-\vvv*0) {$T=$};
\draw[-] (\hhh*0,0) rectangle (\hhh*1,\vvv*1);
\draw[-] (\hhh*1,0) rectangle (\hhh*2,\vvv*1);
\draw[-] (\hhh*2,0) rectangle (\hhh*3,\vvv*1);
\draw[-] (\hhh*3,0) rectangle (\hhh*4,\vvv*1);
\draw[-] (\hhh*1,-\vvv*1) rectangle (\hhh*2,\vvv*0);
\draw[-] (\hhh*2,-\vvv*1) rectangle (\hhh*3,\vvv*0);
\draw[-] (\hhh*2,-\vvv*2) rectangle (\hhh*3,-\vvv*1);
\node[] at (\hhh*0.5,\vvv*0.5) {\scriptsize $1'$};
\node[] at (\hhh*1.5,\vvv*0.5) {\scriptsize $1$};
\node[] at (\hhh*2.5,\vvv*0.5) {\scriptsize $1$};
\node[] at (\hhh*3.5,\vvv*0.5) {\scriptsize $2$};
\node[] at (\hhh*1.5,-\vvv*0.5) {\scriptsize $2$};
\node[] at (\hhh*2.5,-\vvv*0.5) {\scriptsize $3'$};
\node[] at (\hhh*2.5,-\vvv*1.5) {\scriptsize $3$};
\end{tikzpicture}
\hskip 2mm
\begin{tikzpicture}
\def \hhh{5.3mm}    
\def \vvv{5.3mm}    
\draw[->,decorate,decoration={snake,amplitude=.4mm,segment length=2mm,post length=1mm}] (-\hhh*1.5,-\vvv*0) -- (-\hhh*0.5,-\vvv*0);
\draw[-,fill=black!30] (\hhh*0,0) rectangle (\hhh*1,\vvv*1);
\draw[-,fill=black!30] (\hhh*1,0) rectangle (\hhh*2,\vvv*1);
\draw[-,fill=black!30] (\hhh*2,0) rectangle (\hhh*3,\vvv*1);
\draw[-] (\hhh*3,0) rectangle (\hhh*4,\vvv*1);
\draw[-] (\hhh*1,-\vvv*1) rectangle (\hhh*2,\vvv*0);
\draw[-] (\hhh*2,-\vvv*1) rectangle (\hhh*3,\vvv*0);
\draw[-] (\hhh*2,-\vvv*2) rectangle (\hhh*3,-\vvv*1);
\node[] at (\hhh*0.5,\vvv*0.5) {\scriptsize $-1$};
\node[] at (\hhh*1.5,\vvv*0.5) {\scriptsize $-1$};
\node[] at (\hhh*2.5,\vvv*0.5) {\scriptsize $-1$};
\node[] at (\hhh*3.5,\vvv*0.5) {\scriptsize $2$};
\node[] at (\hhh*1.5,-\vvv*0.5) {\scriptsize $2$};
\node[] at (\hhh*2.5,-\vvv*0.5) {\scriptsize $3'$};
\node[] at (\hhh*2.5,-\vvv*1.5) {\scriptsize $3$};
\end{tikzpicture}
\hskip 2mm
\begin{tikzpicture}
\def \hhh{5.3mm}    
\def \vvv{5.3mm}    
\draw[->,decorate,decoration={snake,amplitude=.4mm,segment length=2mm,post length=1mm}] (-\hhh*1.5,-\vvv*0) -- (-\hhh*0.5,-\vvv*0);
\draw[-] (\hhh*0,0) rectangle (\hhh*1,\vvv*1);
\draw[-] (\hhh*1,0) rectangle (\hhh*2,\vvv*1);
\draw[-] (\hhh*2,0) rectangle (\hhh*3,\vvv*1);
\draw[-,fill=black!30] (\hhh*3,0) rectangle (\hhh*4,\vvv*1);
\draw[-] (\hhh*1,-\vvv*1) rectangle (\hhh*2,\vvv*0);
\draw[-,fill=black!30] (\hhh*2,-\vvv*1) rectangle (\hhh*3,\vvv*0);
\draw[-,fill=black!30] (\hhh*2,-\vvv*2) rectangle (\hhh*3,-\vvv*1);
\node[] at (\hhh*0.5,\vvv*0.5) {\scriptsize $2$};
\node[] at (\hhh*1.5,\vvv*0.5) {\scriptsize $2$};
\node[] at (\hhh*2.5,\vvv*0.5) {\scriptsize $3$};
\node[] at (\hhh*3.5,\vvv*0.5) {\scriptsize $-1$};
\node[] at (\hhh*1.5,-\vvv*0.5) {\scriptsize $3$};
\node[] at (\hhh*2.5,-\vvv*0.5) {\scriptsize $-1'$};
\node[] at (\hhh*2.5,-\vvv*1.5) {\scriptsize $-1$};
\end{tikzpicture}
\hskip 2mm
\begin{tikzpicture}
\def \hhh{5.3mm}    
\def \vvv{5.3mm}    
\draw[->,decorate,decoration={snake,amplitude=.4mm,segment length=2mm,post length=1mm}] (-\hhh*1.5,-\vvv*0) -- (-\hhh*0.5,-\vvv*0);
\draw[-,fill=black!30] (\hhh*0,0) rectangle (\hhh*1,\vvv*1);
\draw[-,fill=black!30] (\hhh*1,0) rectangle (\hhh*2,\vvv*1);
\draw[-] (\hhh*2,0) rectangle (\hhh*3,\vvv*1);
\draw[-,fill=black!30] (\hhh*3,0) rectangle (\hhh*4,\vvv*1);
\draw[-] (\hhh*1,-\vvv*1) rectangle (\hhh*2,\vvv*0);
\draw[-,fill=black!30] (\hhh*2,-\vvv*1) rectangle (\hhh*3,\vvv*0);
\draw[-,fill=black!30] (\hhh*2,-\vvv*2) rectangle (\hhh*3,-\vvv*1);
\node[] at (\hhh*0.5,\vvv*0.5) {\scriptsize $-2'$};
\node[] at (\hhh*1.5,\vvv*0.5) {\scriptsize $-2$};
\node[] at (\hhh*2.5,\vvv*0.5) {\scriptsize $3$};
\node[] at (\hhh*3.5,\vvv*0.5) {\scriptsize $-1$};
\node[] at (\hhh*1.5,-\vvv*0.5) {\scriptsize $3$};
\node[] at (\hhh*2.5,-\vvv*0.5) {\scriptsize $-1'$};
\node[] at (\hhh*2.5,-\vvv*1.5) {\scriptsize $-1$};
\end{tikzpicture}
\vskip 2mm
\begin{tikzpicture}
\def \hhh{5.3mm}    
\def \vvv{5.3mm}    
\draw[->,decorate,decoration={snake,amplitude=.4mm,segment length=2mm,post length=1mm}] (-\hhh*1.5,-\vvv*0) -- (-\hhh*0.5,-\vvv*0);
\draw[-] (\hhh*0,0) rectangle (\hhh*1,\vvv*1);
\draw[-] (\hhh*1,0) rectangle (\hhh*2,\vvv*1);
\draw[-,fill=black!30] (\hhh*2,0) rectangle (\hhh*3,\vvv*1);
\draw[-,fill=black!30] (\hhh*3,0) rectangle (\hhh*4,\vvv*1);
\draw[-,fill=black!30] (\hhh*1,-\vvv*1) rectangle (\hhh*2,\vvv*0);
\draw[-,fill=black!30] (\hhh*2,-\vvv*1) rectangle (\hhh*3,\vvv*0);
\draw[-,fill=black!30] (\hhh*2,-\vvv*2) rectangle (\hhh*3,-\vvv*1);
\node[] at (\hhh*0.5,\vvv*0.5) {\scriptsize $3$};
\node[] at (\hhh*1.5,\vvv*0.5) {\scriptsize $3$};
\node[] at (\hhh*2.5,\vvv*0.5) {\scriptsize $-2'$};
\node[] at (\hhh*3.5,\vvv*0.5) {\scriptsize $-1$};
\node[] at (\hhh*1.5,-\vvv*0.5) {\scriptsize $-2'$};
\node[] at (\hhh*2.5,-\vvv*0.5) {\scriptsize $-1'$};
\node[] at (\hhh*2.5,-\vvv*1.5) {\scriptsize $-1$};
\end{tikzpicture}
\hskip 2mm
\begin{tikzpicture}
\def \hhh{5.3mm}    
\def \vvv{5.3mm}    
\draw[->,decorate,decoration={snake,amplitude=.4mm,segment length=2mm,post length=1mm}] (-\hhh*3.6,-\vvv*0) -- (-\hhh*2.6,-\vvv*0);
\node[] at (-\hhh*1.2,-\vvv*0) {$T^E=$};
\draw[-,fill=black!30] (\hhh*0,0) rectangle (\hhh*1,\vvv*1);
\draw[-,fill=black!30] (\hhh*1,0) rectangle (\hhh*2,\vvv*1);
\draw[-,fill=black!30] (\hhh*2,0) rectangle (\hhh*3,\vvv*1);
\draw[-,fill=black!30] (\hhh*3,0) rectangle (\hhh*4,\vvv*1);
\draw[-,fill=black!30] (\hhh*1,-\vvv*1) rectangle (\hhh*2,\vvv*0);
\draw[-,fill=black!30] (\hhh*2,-\vvv*1) rectangle (\hhh*3,\vvv*0);
%
\draw[-,fill=black!30] (\hhh*2,-\vvv*2) rectangle (\hhh*3,-\vvv*1);
\node[] at (\hhh*0.5,\vvv*0.5) {\scriptsize $-3'$};
\node[] at (\hhh*1.5,\vvv*0.5) {\scriptsize $-3$};
\node[] at (\hhh*2.5,\vvv*0.5) {\scriptsize $-2'$};
\node[] at (\hhh*3.5,\vvv*0.5) {\scriptsize $-1$};
\node[] at (\hhh*1.5,-\vvv*0.5) {\scriptsize $-2'$};
\node[] at (\hhh*2.5,-\vvv*0.5) {\scriptsize $-1'$};
\node[] at (\hhh*2.5,-\vvv*1.5) {\scriptsize $-1$};
\end{tikzpicture}

\vskip 2mm
Suppose that $A$ is an SYT consisting of $m$ $\bfa$-boxes.
Define $A^{\st}$ by the SYT obtained from $A$
by replacing all $a'$'s by $1',2',\ldots,k'$ from top to bottom,
and next all $a$'s by $k+1,k+2,\ldots,m$ from left to right,
where $k$ is the number of $a'$'s.
And define $A^{\widehat{\st}}$ as follows:

$\bullet$ If the southwesternmost box of $A$ is filled with $a'$,
then replace all $a$'s by $-1',-2',\ldots,-i'$ from right to left,
and next all $a'$'s by $-(i+1)',\ldots,-m'$ from bottom to top.

$\bullet$ If the southwesternmost box of $A$ is filled with $a$,
then replace all $a$'s by $-1,-2',\ldots,-i'$ from right to left,
and next all $a'$'s by $-(i+1)',\ldots,-m'$ from bottom to top.
\vskip 1mm \noindent
Note that if $A^{\widehat{\st}} = B^{\widehat{\st}}$,
then $A =B$ since the process obtaining $A^{\widehat{\st}}$ from $A$
is reversible.
The following theorem shows the relationship
between the shifted $J$-operation and the shifted generalized evacuation
for SYTs of normal shape.
\begin{thm}\label{thm-J=E}
Let $S$ be an SYT of normal shape.
Then $S^E = S^J$.
\end{thm}
\begin{proof}
We prove the assertion by induction on the number of boxes in $S$.
In case where $S$ consists of a single $\bfa$-box, it is clear that
$S^E = S^J =$
\begin{tikzpicture}[baseline=1mm]
\def\hhh{4.5mm}
\def\vvv{4.5mm}
\draw[-] (0,0) rectangle (\hhh,\vvv);
\node at (\hhh*0.5,\vvv*0.5) () {$\bfa^*$};
\end{tikzpicture}.
Now assume that the assertion holds for all SYTs of normal shape
with $\leq$ $n$ boxes.
Suppose that $S$ is an SYT of normal shape with $n+1$ boxes
and set $a:= \min(S).$
Let $T$ be the SYT obtained by deleting $S^{(a)}$ from $S$.
\vskip 1mm
{\bf Case 1:} $|(S^{(a)})| = 1$
\vskip 1mm
Let us obtain $(\Rect(T), ({S^{(a)}}^E)_T)$
by applying the shifted switching process to $({S^{(a)}}^E, T)$.
Then, applying Algorithm \ref{algorithm3} to $\Rect(T)$,
we see that $S^E =\Rect(T)^E \cup ({S^{(a)}}^E)_T$.
Notice that $({S^{(a)}}^E)_T$ is an $\bfa^*$-box
since ${S^{(a)}}^E$ is a single $\bfa^*$-box.
On the other hand,
$S^J = \Rect(T^*) \cup Y$,
where $Y$ consists of a single $\bfa^*$-box.
Since $\Rect(T)$ is an SYT of normal shape
and
$\Rect(T^*) = \Rect(T)^J$ by Lemma \ref{lem-property-J},
it follows that
$\Rect(T^*) = \Rect(T)^E$ by induction hypothesis.
Also,
$({S^{(a)}}^E)_T$ and $Y$ have the same letter and
have the same position since
$\sh(S^E) = \sh(S^J) = \sh(S)$,
hence $({S^{(a)}}^E)_T = Y$.

\vskip 2mm
{\bf Case 2:} $m:=|(S^{(a)})| \geq 2$
\vskip 1mm
\noindent

By shifting all the entries of $S$ if necessary,
we may assume that
$\wt_{\min(S)}(S) \leq \min(S)$.
As in {\bf Case 1},
we see that $S^E = \Rect(T)^E \cup ({S^{(a)}}^E)_T$
and $S^J = \Rect(T)^J \cup Y$, where $Y$ consists of all $\bfa^*$-boxes in $S^J$.
From induction hypothesis it follows that $\Rect(T)^E = \Rect(T)^J$.
Hence for the proof it suffices to show that $({S^{(a)}}^E)_T = Y$.

First, we claim that $(({S^{(a)}}^{\st})^E)_T = \widehat{Y}$,
where $\widehat{Y}$ is the SYT obtained by removing
$\Rect(T)^J$ from $({S^{(a)}}^{\st} \cup T)^J$.
When we replace $S^{(a)}$ by ${S^{(a)}}^{\st}$ in $S$,
${S^{(a)}}^{\st} \cup T$ is still an SYT by the assumption that $\wt_{\min(S)}(S) \leq \min(S)$ and therefore
we can deduce from {\bf Case 1} that
\begin{equation}\label{eq-case2-1}
\left({S^{(a)}}^{\st} \cup T \right)^E =  \left({S^{(a)}}^{\st} \cup T \right)^J.
\end{equation}
On the other hand, by the definition of $E$,
\begin{displaymath}
\left({S^{(a)}}^{\st} \cup T \right)^E = \Rect(T)^E \cup \left(\left({S^{(a)}}^{\st}\right)^E\right)_T.
\end{displaymath}
Now the claim follows from \eqref{eq-case2-1} along with the definition of $\widehat{Y}$.

Second, we claim that
\begin{equation}\label{eq-case2-2}
\left({({S^{(a)}}^E)_{T}}\right)^{\widehat\st} = \left({({{S^{(a)}}^{{\st}}})}^E\right)_{T}.
\end{equation}
Let us consider $\Sigma({S^{(a)}}^E, T^{(a+1)})$ and $\Sigma({({{S^{(a)}}^{{\st}}})}^E, T^{(a+1)})$.
Since
the (reverse) standardization of ${S^{(a)}}^E$
is equal to
that of ${({{S^{(a)}}^{{\st}}})}^E$,
it follows from Remark \ref{Remark-Diff}(v) that
\begin{displaymath}
\stan\left(({S^{(a)}}^E)_{T^{(a+1)}}\right) = \stan\left(\left({({{S^{(a)}}^{{\st}}})}^E\right)_{T^{(a+1)}}\right).
\end{displaymath}
In case where neither (S4) nor (S7) occurs in the shifted switching process on $({S^{(a)}}^E,T^{(a+1)})$
and in that of $({({{S^{(a)}}^{{\st}}})}^E, T^{(a+1)})$,
the $(\bf{-1})$-box in $({({{S^{(a)}}^{{\st}}})}^E)_{T^{(a+1)}}$
and
the southwesternmost $(-\bfa)$-box in $({S^{(a)}}^E)_{T^{(a+1)}}$
are both primed or both unprimed
since the southwesternmost $(-\bfa)$-box in ${S^{(a)}}^E$
is moved to
the southwesternmost $(-\bfa)$-box in $({S^{(a)}}^E)_{T^{(a+1)}}$
by Lemma \ref{lem-path-intersect}.
Even when (S4) or (S7)
occurs in the shifted switching process on $({S^{(a)}}^E, T^{(a+1)})$,
it is easy to see that they are both primed or both unprimed.
See following figures:
\vskip 2mm 
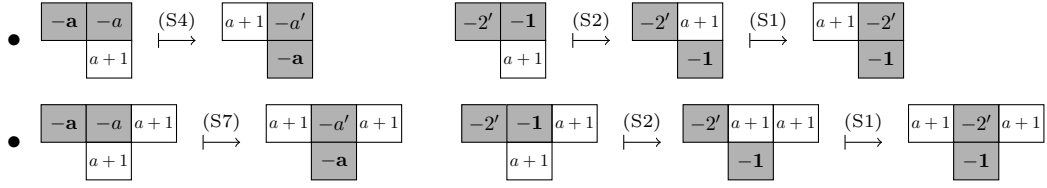
\noindent
\begin{tikzpicture}[baseline=-2pt]
\def \hhh{5mm}    
\def \vvv{5mm}    
\node[] at (\hhh*0,-\vvv*0.5) {$\bullet$};
\end{tikzpicture}
\begin{tikzpicture}[baseline=-2pt]
\def \hhh{6mm}    
\def \vvv{5mm}    
\draw[-,fill=black!30] (\hhh*0,-\vvv*0.5) rectangle (\hhh*1,\vvv*0.5);
\draw[-,fill=black!30] (\hhh*1,-\vvv*0.5) rectangle (\hhh*2,\vvv*0.5);
\draw[-] (\hhh*1,-\vvv*1.5) rectangle (\hhh*2,-\vvv*0.5);
\node[] at (\hhh*0.5,\vvv*0) {\scalebox{.7}{$-\bfa$}};
\node[] at (\hhh*1.5,\vvv*0) {\scalebox{.7}{$-a$}};
\node[] at (\hhh*1.5,-\vvv*1) {\scalebox{.6}{$a+1$}};
\end{tikzpicture}
\begin{tikzpicture}[baseline=-2pt]
\def \hhh{5mm}    
\def \vvv{5mm}    
\draw[|->] (\hhh*0,-\vvv*0.5) -- (\hhh*1,-\vvv*0.5)
node[midway,above,sloped] {\tiny (S4)};
\end{tikzpicture}
\begin{tikzpicture}[baseline=-2pt]
\def \hhh{6mm}    
\def \vvv{5mm}    
\draw[-] (\hhh*0,-\vvv*0.5) rectangle (\hhh*1,\vvv*0.5);
\draw[-,fill=black!30] (\hhh*1,-\vvv*0.5) rectangle (\hhh*2,\vvv*0.5);
\draw[-,fill=black!30] (\hhh*1,-\vvv*1.5) rectangle (\hhh*2,-\vvv*0.5);
\node[] at (\hhh*0.5,\vvv*0) {\scalebox{.6}{$a+1$}};
\node[] at (\hhh*1.5,\vvv*0) {\scalebox{.7}{$-a'$}};
\node[] at (\hhh*1.5,-\vvv*1) {\scalebox{.7}{$-\bfa$}};
\end{tikzpicture}
\hskip 17mm
\begin{tikzpicture}[baseline=-2pt]
\def \hhh{6mm}    
\def \vvv{5mm}    
\draw[-,fill=black!30] (\hhh*0,-\vvv*0.5) rectangle (\hhh*1,\vvv*0.5);
\draw[-,fill=black!30] (\hhh*1,-\vvv*0.5) rectangle (\hhh*2,\vvv*0.5);
\draw[-] (\hhh*1,-\vvv*1.5) rectangle (\hhh*2,-\vvv*0.5);
\node[] at (\hhh*0.5,\vvv*0) {\scalebox{.7}{$-2'$}};
\node[] at (\hhh*1.5,-\vvv*0) {\scalebox{.7}{$-\bf{1}$}};
\node[] at (\hhh*1.5,-\vvv*1) {\scalebox{.6}{$a+1$}};
\end{tikzpicture}
\begin{tikzpicture}[baseline=-2pt]
\def \hhh{5mm}    
\def \vvv{5mm}    
\draw[|->] (\hhh*0,-\vvv*0.5) -- (\hhh*1,-\vvv*0.5) node[midway,above,sloped] {\tiny (S2)};
\end{tikzpicture}
\begin{tikzpicture}[baseline=-2pt]
\def \hhh{6mm}    
\def \vvv{5mm}    
\draw[-,fill=black!30] (\hhh*0,-\vvv*0.5) rectangle (\hhh*1,\vvv*0.5);
\draw[-] (\hhh*1,-\vvv*0.5) rectangle (\hhh*2,\vvv*0.5);
\draw[-,fill=black!30] (\hhh*1,-\vvv*1.5) rectangle (\hhh*2,-\vvv*0.5);
\node[] at (\hhh*0.5,\vvv*0) {\scalebox{.7}{$-2'$}};
\node[] at (\hhh*1.5,\vvv*0) {\scalebox{.6}{$a+1$}};
\node[] at (\hhh*1.5,-\vvv*1) {\scalebox{.7}{$-\bf{1}$}};
\end{tikzpicture}
\begin{tikzpicture}[baseline=-2pt]
\def \hhh{5mm}    
\def \vvv{5mm}    
\draw[|->] (\hhh*0,-\vvv*0.5) -- (\hhh*1,-\vvv*0.5) node[midway,above,sloped] {\tiny (S1)};
\end{tikzpicture}
\begin{tikzpicture}[baseline=-2pt]
\def \hhh{6mm}    
\def \vvv{5mm}    
\draw[-] (\hhh*0,-\vvv*0.5) rectangle (\hhh*1,\vvv*0.5);
\draw[-,fill=black!30] (\hhh*1,-\vvv*0.5) rectangle (\hhh*2,\vvv*0.5);
\draw[-,fill=black!30] (\hhh*1,-\vvv*1.5) rectangle (\hhh*2,-\vvv*0.5);
\node[] at (\hhh*0.5,\vvv*0) {\scalebox{.6}{$a+1$}};
\node[] at (\hhh*1.5,\vvv*0) {\scalebox{.7}{$-2'$}};
\node[] at (\hhh*1.5,-\vvv*1) {\scalebox{.7}{$-\bf{1}$}};
\end{tikzpicture}
\vskip 3mm \noindent
\begin{tikzpicture}[baseline=-2pt]
\def \hhh{5mm}    
\def \vvv{5mm}    
\node[] at (\hhh*0,-\vvv*0.5) {$\bullet$};
\end{tikzpicture}
\begin{tikzpicture}[baseline=-2pt]
\def \hhh{6mm}    
\def \vvv{5mm}    
\draw[-,fill=black!30] (\hhh*0,-\vvv*0.5) rectangle (\hhh*1,\vvv*0.5);
\draw[-,fill=black!30] (\hhh*1,-\vvv*0.5) rectangle (\hhh*2,\vvv*0.5);
\draw[-] (\hhh*2,-\vvv*0.5) rectangle (\hhh*3,\vvv*0.5);
\draw[-] (\hhh*1,-\vvv*1.5) rectangle (\hhh*2,-\vvv*0.5);
\node[] at (\hhh*0.5,\vvv*0) {\scalebox{.7}{$-\bfa$}};
\node[] at (\hhh*1.5,\vvv*0) {\scalebox{.7}{$-a$}};
\node[] at (\hhh*2.5,\vvv*0) {\scalebox{.6}{$a+1$}};
\node[] at (\hhh*1.5,-\vvv*1) {\scalebox{.6}{$a+1$}};
\end{tikzpicture}
\begin{tikzpicture}[baseline=-2pt]
\def \hhh{5mm}    
\def \vvv{5mm}    
\draw[|->] (\hhh*0,-\vvv*0.5) -- (\hhh*1,-\vvv*0.5)
node[midway,above,sloped] {\tiny (S7)};
\end{tikzpicture}
\begin{tikzpicture}[baseline=-2pt]
\def \hhh{6mm}    
\def \vvv{5mm}    
\draw[-] (\hhh*0,-\vvv*0.5) rectangle (\hhh*1,\vvv*0.5);
\draw[-,fill=black!30] (\hhh*1,-\vvv*0.5) rectangle (\hhh*2,\vvv*0.5);
\draw[-] (\hhh*2,-\vvv*0.5) rectangle (\hhh*3,\vvv*0.5);
\draw[-,fill=black!30] (\hhh*1,-\vvv*1.5) rectangle (\hhh*2,-\vvv*0.5);
\node[] at (\hhh*0.5,\vvv*0) {\scalebox{.6}{$a+1$}};
\node[] at (\hhh*1.5,\vvv*0) {\scalebox{.7}{$-a'$}};
\node[] at (\hhh*2.5,\vvv*0) {\scalebox{.6}{$a+1$}};
\node[] at (\hhh*1.5,-\vvv*1) {\scalebox{.7}{$\bf{-a}$}};
\end{tikzpicture}
\hskip 5.5mm
\begin{tikzpicture}[baseline=-2pt]
\def \hhh{6mm}    
\def \vvv{5mm}    
\draw[-,fill=black!30] (\hhh*0,-\vvv*0.5) rectangle (\hhh*1,\vvv*0.5);
\draw[-,fill=black!30] (\hhh*1,-\vvv*0.5) rectangle (\hhh*2,\vvv*0.5);
\draw[-] (\hhh*2,-\vvv*0.5) rectangle (\hhh*3,\vvv*0.5);
\draw[-] (\hhh*1,-\vvv*1.5) rectangle (\hhh*2,-\vvv*0.5);
\node[] at (\hhh*0.5,\vvv*0) {\scalebox{.7}{$-2'$}};
\node[] at (\hhh*1.5,\vvv*0) {\scalebox{.7}{$-\bf{1}$}};
\node[] at (\hhh*2.5,\vvv*0) {\scalebox{.6}{$a+1$}};
\node[] at (\hhh*1.5,-\vvv*1) {\scalebox{.6}{$a+1$}};
\end{tikzpicture}
\begin{tikzpicture}[baseline=-2pt]
\def \hhh{5mm}    
\def \vvv{5mm}    
\draw[|->] (\hhh*0,-\vvv*0.5) -- (\hhh*1,-\vvv*0.5) node[midway,above,sloped] {\tiny (S2)};
\end{tikzpicture}
\begin{tikzpicture}[baseline=-2pt]
\def \hhh{6mm}    
\def \vvv{5mm}    
\draw[-,fill=black!30] (\hhh*0,-\vvv*0.5) rectangle (\hhh*1,\vvv*0.5);
\draw[-] (\hhh*1,-\vvv*0.5) rectangle (\hhh*2,\vvv*0.5);
\draw[-] (\hhh*2,-\vvv*0.5) rectangle (\hhh*3,\vvv*0.5);
\draw[-,fill=black!30] (\hhh*1,-\vvv*1.5) rectangle (\hhh*2,-\vvv*0.5);
\node[] at (\hhh*0.5,\vvv*0) {\scalebox{.7}{$-2'$}};
\node[] at (\hhh*1.5,\vvv*0) {\scalebox{.6}{$a+1$}};
\node[] at (\hhh*2.5,\vvv*0) {\scalebox{.6}{$a+1$}};
\node[] at (\hhh*1.5,-\vvv*1) {\scalebox{.7}{$-\bf{1}$}};
\end{tikzpicture}
\begin{tikzpicture}[baseline=-2pt]
\def \hhh{5mm}    
\def \vvv{5mm}    
\draw[|->] (\hhh*0,-\vvv*0.5) -- (\hhh*1,-\vvv*0.5) node[midway,above,sloped] {\tiny (S1)};
\end{tikzpicture}
\begin{tikzpicture}[baseline=-2pt]
\def \hhh{6mm}    
\def \vvv{5mm}    
\draw[-] (\hhh*0,-\vvv*0.5) rectangle (\hhh*1,\vvv*0.5);
\draw[-,fill=black!30] (\hhh*1,-\vvv*0.5) rectangle (\hhh*2,\vvv*0.5);
\draw[-] (\hhh*2,-\vvv*0.5) rectangle (\hhh*3,\vvv*0.5);
\draw[-,fill=black!30] (\hhh*1,-\vvv*1.5) rectangle (\hhh*2,-\vvv*0.5);
\node[] at (\hhh*0.5,\vvv*0) {\scalebox{.6}{$a+1$}};
\node[] at (\hhh*1.5,\vvv*0) {\scalebox{.7}{$-2'$}};
\node[] at (\hhh*2.5,\vvv*0) {\scalebox{.6}{$a+1$}};
\node[] at (\hhh*1.5,-\vvv*1) {\scalebox{.7}{$-\bf{1}$}};
\end{tikzpicture}
\vskip 2mm\noindent
It follows that
\begin{displaymath}
\left({({S^{(a)}}^E)_{T^{(a+1)}}}\right)^{\widehat\st} = \left({({{S^{(a)}}^{{\st}}})}^E\right)_{T^{(a+1)}}.
\end{displaymath}
Since the shifted switching process on $({S^{(a)}}^E, T)$ is performed inductively in the order $T^{(a+1)}, T^{(a+2)}, \ldots, T^{(\max(S))}$,
\eqref{eq-case2-2} can be obtained by
applying the above method to $a+1$,$a+2$,$\ldots,$ $\max(S)$.

Finally, we claim that $Y^{\widehat\st} = \widehat{Y}$.
Notice that
$T^J \cup Y = {\rm Rect}(T^\ast \cup {S^{(a)}}^\ast)$ and $T^J \cup \widehat{Y} = {\rm Rect}(T^\ast \cup {({S^{(a)}}^{\st})}^\ast)$.
Both ${S^{(a)}}^\ast$ and ${({S^{(a)}}^{\st})}^\ast$ extend $T^*$
and
they have the same standardization as a vertical strip.
Thus, when no special slides occur in the shifted sliding processes on
${S^{(a)}}^\ast$ in $T^J \cup Y$
and ${({S^{(a)}}^{\st})}^\ast$ in $T^J \cup \widehat{Y}$,
then we can apply the same shifted sliding process
to ${S^{(a)}}^\ast$ and to ${({S^{(a)}}^{\st})}^\ast$.
Here `the same shifted sliding process' means that
in each step we apply
the same shifted slide to the same position.
It implies that $\stan(Y) = \stan(\widehat{Y})$.
Furthermore,
the $(\bf{-1})$-box in $\widehat{Y}$
and
the southwesternmost $(-\bfa)$-box in $Y$
are both primed or both unprimed.
Even when a special slide occurs in these processes,
the same results are obtained.
See the following figures:
\vskip 2mm
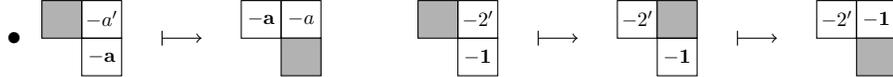
\begin{tikzpicture}
\def \hhh{5.3mm}       
\def \vvv{5mm}     
\def \hhhhh{50mm}    
\node[] at (-\hhh*0.7,-\vvv*0) {$\bullet$};
\draw[fill=black!30] (\hhh*0,\vvv*0) rectangle (\hhh*1,\vvv*1);
\draw[-] (\hhh*1,\vvv*0) rectangle (\hhh*2,\vvv*1);
\draw[-] (\hhh*1,-\vvv*1) rectangle (\hhh*2,\vvv*0);
\node[] at (\hhh*1.5,\vvv*0.5) {\scalebox{.7}{$-a'$}};
\node[] at (\hhh*1.5,-\vvv*0.5) {\scalebox{.7}{$-\bfa$}};
\draw[|->] (\hhh*3,-\vvv*0) to (\hhh*4,-\vvv*0);
\draw[-] (\hhh*5,\vvv*0) rectangle (\hhh*6,\vvv*1);
\draw[-] (\hhh*6,\vvv*0) rectangle (\hhh*7,\vvv*1);
\draw[fill=black!30] (\hhh*6,-\vvv*1) rectangle (\hhh*7,\vvv*0);
\node[] at (\hhh*5.5,\vvv*0.5) {\scalebox{.7}{$-\bfa$}};
\node[] at (\hhh*6.5,\vvv*0.5) {\scalebox{.7}{$-a$}};
\draw[fill=black!30] (\hhh*0+\hhhhh,\vvv*0) rectangle (\hhh*1+\hhhhh,\vvv*1);
\draw[-] (\hhh*1+\hhhhh,\vvv*0) rectangle (\hhh*2+\hhhhh,\vvv*1);
\draw[-] (\hhh*1+\hhhhh,-\vvv*1) rectangle (\hhh*2+\hhhhh,\vvv*0);
\node[] at (\hhh*1.5+\hhhhh,\vvv*0.5) {\scalebox{.7}{$-2'$}};
\node[] at (\hhh*1.5+\hhhhh,-\vvv*0.5) {\scalebox{.7}{$-\bf{1}$}};
\draw[|->] (\hhh*3+\hhhhh,-\vvv*0) to (\hhh*4+\hhhhh,-\vvv*0);
\draw[-] (\hhh*5+\hhhhh,\vvv*0) rectangle (\hhh*6+\hhhhh,\vvv*1);
\draw[fill=black!30] (\hhh*6+\hhhhh,\vvv*0) rectangle (\hhh*7+\hhhhh,\vvv*1);
\draw[-] (\hhh*6+\hhhhh,-\vvv*1) rectangle (\hhh*7+\hhhhh,\vvv*0);
\node[] at (\hhh*5.5+\hhhhh,\vvv*0.5) {\scalebox{.7}{$-2'$}};
\node[] at (\hhh*6.5+\hhhhh,-\vvv*0.5) {\scalebox{.7}{$-\bf{1}$}};
\draw[|->] (\hhh*8+\hhhhh,-\vvv*0) to (\hhh*9+\hhhhh,-\vvv*0);
\draw[-] (\hhh*10+\hhhhh,\vvv*0) rectangle (\hhh*11+\hhhhh,\vvv*1);
\draw[-] (\hhh*11+\hhhhh,\vvv*0) rectangle (\hhh*12+\hhhhh,\vvv*1);
\draw[fill=black!30] (\hhh*11+\hhhhh,-\vvv*1) rectangle (\hhh*12+\hhhhh,\vvv*0);
\node[] at (\hhh*10.5+\hhhhh,\vvv*0.5) {\scalebox{.7}{$-2'$}};
\node[] at (\hhh*11.5+\hhhhh,\vvv*0.5) {\scalebox{.7}{$-\bf{1}$}};
\end{tikzpicture}
\vskip 2mm \noindent
Thus our claim is verified.

Combining the above three claims yields that
\begin{displaymath}
\left({({S^{(a)}}^E)_{T}}\right)^{\widehat\st} = \left({({{S^{(a)}}^{{\st}}})}^E\right)_{T} =
\widehat{Y} = Y^{\widehat\st}.
\end{displaymath}
Now the equality
${({S^{(a)}}^E)_{T}}= Y$ follows from the reversibility of $\widehat{\st}$.
\end{proof}

\vskip 2mm
For any $\lambda \subset \nu$, let us consider the map
$$\varphi_\lambda: \widetilde{\mathcal{Y}}(\lambda)\times \widetilde{\mathcal{Y}}(\nu/\lambda) \rightarrow
\widetilde{\mathcal{Y}}_{-}(\lambda) \times \widetilde{\mathcal{Y}}(\nu/\lambda),
\qquad
(S,T) \mapsto (S^J,T).
$$
Set $\varphi:=\cup_{\lambda \subset \nu} \varphi_\lambda$.
We are ready to state the main result in this section.
\begin{thm}\label{thm-main-Section5}
For $\lambda,\nu \in \Lambda^+$ with $\lambda \subset \nu$,
let $S \in \widetilde{\mathcal{Y}}(\ld)$ and
$T \in \widetilde{\mathcal{Y}}(\nu/\ld)$.
Then $(S^J \cup T)^{J} = T^{J} \cup S_T,$
and the following diagram
\begin{equation*}
\begin{tikzpicture}
\def \hhhh{55mm}    
\def \vvvv{18mm}    
\node[] at (\hhhh*0,\vvvv*0) (A11) {$\widetilde{\mathcal{Y}}(\lambda)\times \widetilde{\mathcal{Y}}(\nu/\lambda)$};
\node[] at (\hhhh*1,-\vvvv*0) (A12) {$\bigsqcup_{\mu \subset \nu} 
\widetilde{\mathcal{Y}}(\displaystyle \mu)\times \widetilde{\mathcal{Y}}(\nu/\mu)$};
\node[] at (\hhhh*0,-\vvvv*1.2) (A21) {$\widetilde{\mathcal{Y}}_{-}(\lambda)\times \widetilde{\mathcal{Y}}(\nu/\lambda)$};
\node[] at (\hhhh*1,-\vvvv*1.2) (A22) {$\bigsqcup_{\mu\subset \nu} 
\widetilde{\mathcal{Y}}_{-}(\mu)\times \widetilde{\mathcal{Y}}(\nu/\mu)$};
\draw[->] (A11.east) -- (A12.west) node[above,midway] {$\Sigma$};
\draw[->] (A11.south) -- (A21.north) node[left,midway] {$\varphi_\lambda$};
\draw[->] (A12.south) -- (A22.north) node[right,midway] {$\varphi$};
\draw[->] (A21.east) -- (A22.west) node[below,midway] {$J$};
\end{tikzpicture}
\end{equation*}
commutes, that is, $J \circ \varphi_\lambda = \varphi \circ \Sigma$.
\end{thm}
\begin{proof}
Since $S^J \cup T \in \widetilde{\mathcal{Y}}_{\pm}(\nu)$,
it follows that $(S^J \cup T)^J \in \widetilde{\mathcal{Y}}_{\pm}(\nu)$.
Suppose that $(S^J \cup T)^J = \Rect(T)^J \cup Y$.
Since $S^J \cup T$ is an SYT of normal shape,
$(S^J \cup T)^{J} = (S^J \cup T)^{E}$ by Theorem \ref{thm-J=E}.
On the other hand,
$(S^J \cup T)^{E} = \Rect(T)^E \cup ((S^J)^E)_T$
by Algorithm \ref{algorithm3}.
In addition,
$(S^J)^E = (S^J)^J = S$ since $J$ is an involution and $S^J$ is of normal shape.
This proves the first assertion.
The second one is immediate from the first.
\end{proof}

\section{The modified shifted tableau switching}
\label{sect-mss-proce}
Throughout this section,
we will consider only shifted perforated $(\bfa,\bfb)$-pairs
with no primed letters on the main diagonal.
Unfortunately the set of these pairs is not
closed under the shifted tableau switching.
In more detail, if a switch defined in Section \ref{sect-ss-proce}
is applied to the following $\bfa$-boxes
\vskip 2mm
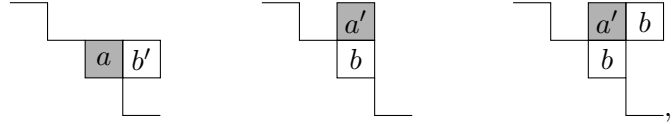
\begin{center}
\begin{tikzpicture}
\def \hhh{5mm}    
\def \vvv{5mm}    
\draw[-] (-\hhh*2,\vvv*2) -- (-\hhh*1,\vvv*2) -- (-\hhh*1,\vvv*1) -- (\hhh*0,\vvv*1);
\draw[fill=black!30] (0,0) rectangle (\hhh*1,\vvv*1);
\draw[-] (\hhh*1,0) rectangle (\hhh*2,\vvv*1);
\draw[-] (\hhh*1,\vvv*0) -- (\hhh*1,-\vvv*1) -- (\hhh*2,-\vvv*1);
\node[] at (\hhh*0.5,\vvv*0.5) () {$a$};
\node[] at (\hhh*1.5,\vvv*0.5) () {$b'$};
\end{tikzpicture}
\hskip 12mm
\begin{tikzpicture}
\def \hhh{5mm}    
\def \vvv{5mm}    
\draw[-] (-\hhh*2,\vvv*1) -- (-\hhh*1,\vvv*1) -- (-\hhh*1,\vvv*0) -- (\hhh*0,\vvv*0);
\draw[fill=black!30] (0,0) rectangle (\hhh*1,\vvv*1);
\draw[-] (\hhh*0,-\vvv*1) rectangle (\hhh*1,0);
\draw[-] (\hhh*1,-\vvv*1) -- (\hhh*1,-\vvv*2) -- (\hhh*2,-\vvv*2);
\node[] at (\hhh*0.5,\vvv*0.5) () {$a'$};
\node[] at (\hhh*0.5,-\vvv*0.5) () {$b$};
\end{tikzpicture}
\hskip 12mm
\begin{tikzpicture}
\def \hhh{5mm}   
\def \vvv{5mm}   
\draw[-] (-\hhh*1,\vvv*1) -- (-\hhh*0,\vvv*1) -- (-\hhh*0,\vvv*0) -- (\hhh*1,\vvv*0);
\draw[fill=black!30] (\hhh*1,0) rectangle (\hhh*2,\vvv*1);
\draw[-] (\hhh*2,0) rectangle (\hhh*3,\vvv*1);
\draw[-] (\hhh*1,-\vvv*1) rectangle (\hhh*2,0);
\draw[-] (\hhh*2,-\vvv*1) -- (\hhh*2,-\vvv*2) -- (\hhh*3,-\vvv*2);
\node[] at (\hhh*1.5,\vvv*0.5) () {$a'$};
\node[] at (\hhh*2.5,\vvv*0.5) () {$b$};
\node[] at (\hhh*1.5,-\vvv*0.5) () {$b$};
\end{tikzpicture},
\end{center}
then the resulting pairs have a primed letter on the main diagonal.
To avoid this phenomenon,
we introduce three new switches labelled by (S$1'$), (S$2'$) and (S$6'$).
\begin{itemize}
\item (S$1'$) is applied in case where an $a$-box on the main diagonal is horizontally adjacent to a unique $b'$-box as follows:
\begin{center}
\begin{tikzpicture}
\def \hhh{5mm}    
\def \vvv{5mm}    
\node[] at (-\hhh*1.5,\vvv*0.5) () {(S1$'$)};
\draw[fill=black!30] (0,0) rectangle (\hhh*1,\vvv*1);
\draw[-] (\hhh*1,0) rectangle (\hhh*2,\vvv*1);
\draw[|->] (\hhh*3,\vvv*0.5) to (\hhh*4,\vvv*0.5);
\draw[-] (\hhh*5,\vvv*0) rectangle (\hhh*6,\vvv*1);
\draw[fill=black!30] (\hhh*6,\vvv*0) rectangle (\hhh*7,\vvv*1);
\node[] at (\hhh*0.5,\vvv*0.5) () {$a$};
\node[] at (\hhh*1.5,\vvv*0.5) () {$b'$};
\node[] at (\hhh*5.5,\vvv*0.5) () {$b$};
\node[] at (\hhh*6.5,\vvv*0.5) () {$a'$};
\end{tikzpicture}
\end{center}
\item (S$2'$) is applied in case where an $a'$-box is vertically adjacent to a unique $b$-box
on the main diagonal as follows:
\begin{center}
\begin{tikzpicture}
\def \hhh{5mm}    
\def \vvv{5mm}    
\node[] at (-\hhh*1.5,0) () {(S2$'$)};
\draw[fill=black!30] (0,0) rectangle (\hhh*1,\vvv*1);
\draw[-] (\hhh*0,-\vvv*1) rectangle (\hhh*1,0);
\draw[|->] (\hhh*2.5,\vvv*0) to (\hhh*3.5,\vvv*0);
\draw[-] (\hhh*5,\vvv*0) rectangle (\hhh*6,\vvv*1);
\draw[fill=black!30] (\hhh*5,-\vvv*1) rectangle (\hhh*6,\vvv*0);
\node[] at (\hhh*0.5,\vvv*0.5) () {$a'$};
\node[] at (\hhh*0.5,-\vvv*0.5) () {$b$};
\node[] at (\hhh*5.5,\vvv*0.5) () {$b'$};
\node[] at (\hhh*5.5,-\vvv*0.5) () {$a$};
\node[] at (\hhh*6.5,-\vvv*0.5) () {\text{ }};
\end{tikzpicture}
\end{center}
\item (S$6'$) is applied in case where an $a'$-box is horizontally adjacent to a $b$-box
and vertically adjacent to a $b$-box on the main diagonal as follows:
\begin{center}
\begin{tikzpicture}
\def \hhh{5mm}   
\def \vvv{5mm}   
\def \hhhh{55mm}  
\def \mult{8}    
\node[] at (-\hhh*0.5,0) () {(S6$'$)};
\draw[|->] (\hhh*4,0) to (\hhh*5,0);
\draw[fill=black!30] (\hhh*1,0) rectangle (\hhh*2,\vvv*1);
\draw[-] (\hhh*2,0) rectangle (\hhh*3,\vvv*1);
\draw[-] (\hhh*1,-\vvv*1) rectangle (\hhh*2,0);
\draw[-] (\hhh*6,\vvv*0) rectangle (\hhh*7,\vvv*1);
\draw[-] (\hhh*7,\vvv*0) rectangle (\hhh*8,\vvv*1);
\draw[fill=black!30] (\hhh*6,-\vvv*1) rectangle (\hhh*7,\vvv*0);
\node[] at (\hhh*1.5,\vvv*0.5) () {$a'$};
\node[] at (\hhh*2.5,\vvv*0.5) () {$b$};
\node[] at (\hhh*1.5,-\vvv*0.5) () {$b$};
\node[] at (\hhh*6.5,\vvv*0.5) () {$b'$};
\node[] at (\hhh*7.5,\vvv*0.5) () {$b$};
\node[] at (\hhh*6.5,-\vvv*0.5) () {$a$};
\end{tikzpicture}
\end{center}
\end{itemize}
For the clarity
we call these the {\it modified switches}.
The original (S1), (S2) and (S6)
will be used only in the cases other than the above three ones.
Using the switches
(S1) through (S7) together with (S1$'$), (S2$'$) and (S6$'$),
let us define the {\it modified shifted (tableau) switching process}
on shifted perforated $(\bfa,\bfb)$-pairs
in the same way as the shifted switching process has been defined.
It should be remarked that
the modified shifted switching process in which no modified switches are used
is identical to the shifted switching process.

Suppose that $T$ is a filling of a double border strip shape
with letters $\bfa$ and $\bfb$ and is not fully switched.
Let $\widetilde{\sigma}(T)$ denote the resulting filling obtained from $T$
by applying a modified switch.

\begin{prop}\label{prop-main-sect6}
Let $(A,B)$ be a shifted perforated $(\bfa,\bfb)$-pair such that
$B$ extends $A$.
\begin{itemize}
\item[(a)] Modified switches in total can appear at most once in the modified shifted switching process on $(A,B)$.
\item[(b)] Each $(\bfa,\bfb)$-pair appearing in the modified shifted switching process on $(A,B)$ is still a shifted perforated $(\bfa,\bfb)$-pair.
\end{itemize}
\end{prop}
\begin{proof}
(a) Suppose that $\widetilde{\sigma}$ is applied to $\sigma^k (A,B)$ for some $k\geq 0$.
Here $\sigma^k (A,B)$ is a shifted perforated $(\bfa,\bfb)$-pair by Proposition \ref{prop-preserving-sp}.

When $\widetilde{\sigma} =$(S1$'$) is applied to
the $(i,i)$ $a$-box in $\sigma^k (A,B)$,
then
$(\widetilde{\sigma} \circ \sigma^k (A,B))_{i,i} =b$ and
$(\widetilde{\sigma} \circ \sigma^k (A,B))_{i,i+1}$ $=a'$.
Moreover,
the $(i+1,i+1)$ box is not contained in $\sigma^k (A,B)$
since if it is filled with $b$,
then
the switch to be applied to $\sigma^k(A, B)$ is not (S$1'$) but (S3).
Now suppose that another modified switch is applied to
$\sigma^j \circ \widetilde{\sigma} \circ \sigma^k (A,B)$ for some $j \geq 0$.
Then $\Box(\sigma^j \circ \widetilde{\sigma} \circ \sigma^k (A,B))$ is
$(i-1,i)$ and has the neighbor to the south.
Since $(\sigma^k (A \cup B))_{i,i+1}=b'$, it follows that
$(\sigma^j \circ \widetilde{\sigma} \circ \sigma^k (A,B))_{i-1,i+1}=b'$.
Therefore the switch to be applied to
$\sigma^j \circ \widetilde{\sigma} \circ \sigma^k (A,B)$ is
not a modified switch but (S5), which is absurd.

When $\widetilde{\sigma} =$(S2$'$) or (S$6'$) is applied to
the $(i-1,i)$ $a'$-box in $\sigma^k(A, B)$,
then
$(\widetilde{\sigma} \circ \sigma^k (A, B))_{i,i} =a$,
$(\widetilde{\sigma} \circ \sigma^k (A, B))_{i-1,i} =b'$,
and
the $(i-1,i-1)$ box is not contained in $\sigma^k(A, B)$
by Definition \ref{Def-Perforated}.
By the occurrence of (S2$'$) or (S$6'$),
the $(i,i+1)$ box in $\sigma^k(A,B)$ cannot be filled with $\bfb$,
so the $(i,i+1)$ box in $\sigma^k (A, B)$
is an $a$-box or an empty box.
This implies that the $(i,i)$ $a$-box in $(\widetilde{\sigma} \circ \sigma^k (A, B))$
is fully switched and is a unique box on the main diagonal in
$(\widetilde{\sigma} \circ \sigma^k (A, B))$.
As a result, no modified switches can be applied to
$(\sigma^j \circ \widetilde{\sigma} \circ \sigma^k)(A,B)$ for all $j \geq 0$.

(b) If a modified switch is not used in the modified shifted switching process,
there is nothing to prove.
Otherwise, by (a), each $(\bfa,\bfb)$-pair appearing in the modified shifted switching process can be written as
$\sigma^j \circ \widetilde{\sigma} \circ \sigma^k(A, B)$ for some nonnegative integers $j,k$.
Due to Proposition \ref{prop-preserving-sp},
we have only to see that
there exists a pair of SYTs $(\widehat{A},\widehat{B})$
such that $\widehat{B}$ extends $\widehat{A}$ and
\begin{equation*}\label{eq-prop(b)}
\widetilde{\sigma}\circ \sigma^k(A,B) = \sigma^{k+1}(\widehat{A},\widehat{B}).
\end{equation*}
Set $(A', B'):=\sigma^k(A, B).$

First, suppose that $\widetilde{\sigma} = $ (S$1'$) is applied to the $(i,i)$ $a$-box in $A'$.
Then
\begin{displaymath}
\widetilde{\sigma}(A', B') \approx \sigma(A',B') \quad
\text{up to }(i,i) \text{ and } (i,i+1),
\end{displaymath}
where $\sigma = $ (S1).
Let $B''$ be the filling obtained from $B'$ by changing
$b'$ into $b$ at $(i,i+1)$.
It is a shifted perforated $\bfb$-tableau
since the $(i+1,i+1)$ box is not contained in $B'$.
Similarly, let $A''$ be the filling obtained from $A'$ by changing
$a$ into $a'$ at $(i,i)$.
It is also a shifted perforated $\bfa$-tableau
with a primed letter on the main diagonal
since
the $(i,i)$ $a$-box is the westernmost box in $A'$.
Moreover, since the $(i,i)$ $a'$-box is the southwesternmost box in $A''$,
it follows that $\Box(A'' \cup B'') = (i,i)$ and thus
\begin{equation*}\label{eq-widehat-sigma}
\widetilde{\sigma}(A', B') = \sigma(A'',B'').
\end{equation*}
Now define $(\widehat{A}, \widehat{B})$ to be the pair of SYTs obtained from $(A, B)$
by changing $a$ into $a'$ at the westernmost $a$-box in $A$ and
by changing $b'$ into $b$ at the southernmost $b'$-box in $B$.
We claim that
$A'' \cup B'' = \sigma^k(\widehat{A},\widehat{B}).$
Since the westernmost $(i,i)$ $a$-box in $A$ is not affected by $\sigma^k$
and
the southernmost $b'$-box of $B$ has been moved to
the $(i,i+1)$ $b'$-box in $B'$ through $\sigma^k$,
$$
\sigma^k(\widehat{A},\widehat{B}) \approx A'\cup B' (\approx A'' \cup B'')
\quad \text{up to }(i,i) \text{ and }(i,i+1).
$$
The claim follows from the fact that the $(i,i)$ box and $(i,i+1)$
in $\sigma^k(\widehat{A},\widehat{B})$ are $a'$ and $b$, respectively.

Second, suppose that $\widetilde{\sigma}=$(S$2'$) is applied to the $(i-1,i)$ $a'$-box in $A'$.
Then
$$
\widetilde{\sigma}(A',B') \approx \sigma(A',B')
\quad \text{up to $(i-1,i)$ and $(i,i)$,}$$
where $\sigma = $ (S2).
Let $A''$ be the filling obtained from $A'$ by
changing $a'$ into $a$ at $(i-1,i)$.
It is a shifted perforated $\bfa$-tableau since
no $a'$-boxes are in the $i$th row in $A'$.
Similarly, let
$B''$ be the filling obtained from $B'$ by changing
$b$ into $b'$ at $(i,i)$.
It is also a shifted perforated $\bfb$-tableau
with a primed letter on the main diagonal
since the $(i,i)$ $b$-box is the westernmost box in $B'$.
In addition, since
the $(i-1,i)$ box in $A'$ is the southernmost $a'$-box
and
the $(i-1,i)$ box in $A''$ is the westernmost $a$-box,
it follows that $\Box(A' \cup B') = \Box(A'' \cup B'') = (i-1,i)$,
and thus
\begin{displaymath}
\widetilde{\sigma}(A',B') = \sigma(A'',B'').
\end{displaymath}
Define $(\widehat{A},\widehat{B})$ to be the pair of SYTs
obtained from $(A,B)$ by changing $a'$ into $a$
at the southernmost $a'$-box in $A$ and
$b$ into $b'$ at the westernmost $b$-box in $B$.
We claim that
$A'' \cup B'' = \sigma^k(\widehat{A},\widehat{B}).$
Notice that
the $(i,i+1)$ box is clearly not in $A\cup B$,
so the $(i,i)$ box is filled with $b$ in $A \cup B$.
Since (S$2'$) is applied to the  $(i-1,i)$ $a'$-box,
the $(i,i)$ box in $A' \cup B'$ is filled with $b$.
It means that the westernmost $(i,i)$ $b$-box in $B$ is not affected by $\sigma^k$.
On the other hand,
no boxes are in the $(i-1)$st column of $A \cup B$,
so the southernmost $a'$-box in the $i$th column of $A$ has been moved to
the $(i-1,i)$ $a'$-box in $A'$ through $\sigma^k$.
Combining these properties yields that
$$
\sigma^k(\widehat{A},\widehat{B}) \approx A'\cup B' (\approx A'' \cup B'')
\quad \text{up to }(i-1,i) \text{ and }(i,i).
$$
The claim now follows from the fact that the $(i-1,i)$ box and $(i,i)$
in $\sigma^k(\widehat{A},\widehat{B})$ are $a$ and $b'$, respectively.

Finally, in case where $\widetilde{\sigma}=$ (S6$'$) is applied to the  $(i-1,i)$ $a'$-box in $A'$,
$(\widehat{A},\widehat{B})$ can be obtained
in a similar way as in the second case.
In this case,
the $(i,i+1)$ box may be or may not be in $A \cup B$.
In any case, it can be easily seen that
the westernmost $(i,i)$ $b$-box in $B$ is not affected by $\sigma^k$
and
the southernmost $a'$-box in the $i$th column of $A$ has been
moved to the $(i-1,i)$ $a'$-box in $A'$ through $\sigma^k$.
\end{proof}

\vskip 3mm
From now on, we will use the notation
$\widetilde{\Sigma}$
to distinguish
the modified shifted switching process
from the shifted switching process.
For a shifted perforated $(\bfa,\bfb)$-pair $(A,B)$ such that $B$ extends $A$
we will write $\widetilde{A_B}$ for $\widetilde{\Sigma}(A, B)^{(a)}$
and
$\widetilde{{}^AB}$ for $\widetilde{\Sigma}(A, B)^{(b)}$.
We also call the map $\widetilde{\Sigma}$ obtained via this process
{\it the modified shifted tableau switching}.

Let $A$ and $B$ be SYTs consisting of $\bfa$-boxes and $\bfb$-boxes, respectively.
Define $\omega(A)$ by the SYT obtained by priming the southwesternmost box in $A$.
Recall that $(a')'$ is set to be $a$.
In Particular, for $(A,B)$ such that $B$ extends $A$,
define $\Omega(A,B)$ by $(\omega(A),\omega(B)).$

\begin{lem}\label{lem-hatsigma-sigma}
Suppose that
$(A, B)$ is a shifted perforated $(\bfa,\bfb)$-pair
such that $B$ extends $A$.
If $\Sigma(A, B)$ is not equal to $\widetilde{\Sigma}(A, B)$,
then we have
\begin{itemize}
  \item[(a)] $\widetilde{\Sigma}(A,B) = \Sigma \circ \Omega(A,B).$
  \item[(b)] $\Omega \circ \widetilde{\Sigma}(A,B) = \Sigma(A,B),$
  or equivalently, $\Omega \circ \Sigma = \Sigma \circ \Omega.$
\end{itemize}
\end{lem}
\begin{proof}
(a)
The proof of Proposition \ref{prop-main-sect6}(b) shows that
if a modified switch is applied to
$\sigma^k(A,B)$ for some $k \geq 0$
then
\begin{equation}\label{eq-lem-hatsigma-sigma}
\widetilde{\sigma}\circ \sigma^k(A,B) = \sigma^{k+1}(\omega(A),\omega(B)).
\end{equation}
By Proposition \ref{prop-main-sect6}(a),
only nonmodified switches can be applied to either side of \eqref{eq-lem-hatsigma-sigma}.
Therefore we can conclude that
$$
\sigma^l \circ \widetilde{\sigma}\circ \sigma^k(A,B) = \sigma^l \circ \sigma^{k+1}(\omega(A),\omega(B)) \quad \text{for all } l \geq 0.
$$

(b)
Note that
\begin{eqnarray*}
\Omega \circ \Sigma \circ \Omega (A,B) &=& \Omega \circ \Sigma(\omega(A),\omega(B))\\
&=&\Omega ({}^{\omega(A)}\omega(B),\omega(A)_{\omega(B)})\\
&=&(\omega({}^{\omega(A)}\omega(B)),\omega(\omega(A)_{\omega(B)})).
\end{eqnarray*}
Since $\Sigma(A,B) \neq \widetilde{\Sigma}(A,B)$ by assumption,
the switches (S3), (S4) and (S7) cannot occur in the modified shifted switching process.
This again implies that
they cannot occur in the shifted switching process, too.
In particular, since neither (S4) nor (S7) is not used in obtaining $({}^A\omega(B),A_{\omega(B)})$
from $(A, \omega(B))$, the southwesternmost box in $A$ is moved to that in $A_{\omega(B)}$
by Lemma \ref{lem-sw-path}
and these boxes are both primed or both unprimed.
It follows that
$\omega(A_{\omega(B)}) = {\omega(A)}_{\omega(B)}$.
In addition,
since (S3) is not used in the shifted jeu de taquin of $B$
depending on $\omega(A)$,
the southwesternmost box in $B$ is also moved to
that in ${}^{\omega(A)}B$
and these two boxes are both primed or both unprimed.
This implies that
$\omega({}^{\omega(A)}B) = {}^{\omega(A)}\omega(B)$.
Now combining Remark \ref{Remark-Diff}(iv) with
the fact that $\stan(A) = \stan(\omega(A))$ and $\stan(B) = \stan(\omega(B))$,
we can derive that
\begin{eqnarray*}
&{}^{\omega(A)}B = {}^{\stan(\omega(A))}B = {}^{\stan(A)}B = {}^AB &  \text{and}\\
&A_{\omega(B)} = A_{\stan(\omega(B))} =   A_{\stan(B)} = A_B. &
\end{eqnarray*}
Consequently
\begin{displaymath}
\Omega \circ \widetilde{\Sigma}(A,B)
=\left(\omega({}^{\omega(A)}\omega(B)),\omega(\omega(A)_{\omega(B)})\right)
=\left({}^{\omega(A)}B,A_{\omega(B)}\right)
=({}^{A}B,A_{B}).
\end{displaymath}
\end{proof}

\vskip 3mm
\begin{lem}\label{thm-AB-SSYT-MSA}
Let $(A,B)$ be a shifted perforated $(\bfa,\bfb)$-pair such that $B$ extends $A$.
Then we have
\begin{itemize}
\item[{\rm (a)}] $\widetilde{A_B} \cup \widetilde{{}^B A}$ has the same shape as $A \cup B$.
\item[{\rm (b)}] $\widetilde{A_B}$ extends $\widetilde{{}^AB}$. In addition, $\widetilde{{A}_B}$ and $\widetilde{{}^AB}$ are SSYTs of border strip shape.
\end{itemize}
\end{lem}
\begin{proof}
(a) The assertion follows from the fact that
neither a switch nor a modified switch
does not affect on the shape.

(b) When $\Sigma(A,B) = \widetilde{\Sigma}(A,B)$, there is nothing to prove.
When $\Sigma(A,B) \neq \widetilde{\Sigma}(A,B)$,
$\widetilde{\Sigma}(A,B)$ is equal to $\Omega \circ \Sigma(A,B)$ by Lemma \ref{lem-hatsigma-sigma}(b).
It follows that
$\widetilde{\Sigma}(A,B)^{(a)}$ extends $\widetilde{\Sigma}(A,B)^{(b)}$.
Since no primed letters are on the main diagonal in
$\widetilde{\Sigma}(A,B)$,
we obtain the desired result.
\end{proof}

\vskip 3mm
The subsequent theorem shows that
the set of pairs of SSYTs
is invariant
under the modified shifted tableau switching.

\begin{thm}\label{thm-main-section6}
Let $(S,T)$ be a pair of SSYTs such that $T$ extends $S$.
The modified shifted tableau switching $(S,T) \mapsto (\widetilde{{}^ST},\widetilde{S_T})$ satisfies the following.
\begin{itemize}
\item[{\rm (a)}] $\widetilde{S_T} \cup \widetilde{{}^S T}$ has the same shape as $S \cup T$.
In addition, $\widetilde{S_T}$ (resp., $\widetilde{{}^ST}$) has the same weight as $S$ (resp., $T$).

\item[{\rm (b)}] $\widetilde{S_T}$ extends $\widetilde{{}^S T}$. In addition, $\widetilde{S_T}$ and $\widetilde{{}^S T}$ are SSYTs.
\end{itemize}
\end{thm}
\begin{proof}
In view of Lemma \ref{thm-AB-SSYT-MSA},
the proof can be done essentially in the same way as in
the proof of Theorem \ref{Thm-Sec3-SW}.
\end{proof}

\vskip 3mm
The following example illustrates how
the modified shifted switching process works.
Compare the result below with Example \ref{example-section3}.
Let
\begin{center}
\begin{tikzpicture}
\def \hhh{5mm}    
\def \vvv{5mm}    
\def \hhhhh{55mm}  
\node[] at (\hhh*0.8,-\vvv*0.4) {$S=$};
\draw[fill=black!30] (\hhh*3,0) rectangle (\hhh*4,\vvv*1);
\draw[fill=black!30] (\hhh*4,0) rectangle (\hhh*5,\vvv*1);
\draw[-] (\hhh*5,0) rectangle (\hhh*6,\vvv*1);
\draw[-] (\hhh*6,0) rectangle (\hhh*7,\vvv*1);
\draw[fill=black!30] (\hhh*2,-\vvv*1) rectangle (\hhh*3,\vvv*0);
\draw[fill=black!30] (\hhh*3,-\vvv*1) rectangle (\hhh*4,\vvv*0);
\draw[fill=black!30] (\hhh*4,-\vvv*1) rectangle (\hhh*5,\vvv*0);
\draw[-] (\hhh*5,-\vvv*1) rectangle (\hhh*6,\vvv*0);
\draw[fill=black!30] (\hhh*3,-\vvv*2) rectangle (\hhh*4,-\vvv*1);
\draw[-] (\hhh*4,-\vvv*2) rectangle (\hhh*5,-\vvv*1);
\draw[-] (\hhh*5,-\vvv*2) rectangle (\hhh*6,-\vvv*1);
\draw[-] (\hhh*4,-\vvv*3) rectangle (\hhh*5,-\vvv*2);
\draw[-] (\hhh*5,-\vvv*3) rectangle (\hhh*6,-\vvv*2);
\node[] at (\hhh*3.5,\vvv*0.5) {$1'$};
\node[] at (\hhh*4.5,\vvv*0.5) {$1$};
\node[] at (\hhh*5.5,\vvv*0.5) {};
\node[] at (\hhh*6.5,\vvv*0.5) {};
\node[] at (\hhh*7.5,\vvv*0.5) {};
\node[] at (\hhh*2.5,-\vvv*0.5) {$1$};
\node[] at (\hhh*3.5,-\vvv*0.5) {$1$};
\node[] at (\hhh*4.5,-\vvv*0.5) {$2'$};
\node[] at (\hhh*5.5,-\vvv*0.5) {};
\node[] at (\hhh*6.5,-\vvv*0.5) {};
\node[] at (\hhh*3.5,-\vvv*1.5) {$2$};
\node[] at (\hhh*4.5,-\vvv*1.5) {};
\node[] at (\hhh*5.5,-\vvv*1.5) {};
\node[] at (\hhh*2.5,-\vvv*2.5) {};
\node[] at (\hhh*3.5,-\vvv*2.5) {};
\node[] at (\hhh*3.6+\hhhhh*0.5,-\vvv*0.4) {and};
\node[] at (\hhh*0.9+\hhhhh,-\vvv*0.4) {$T=$};
\draw[-] (\hhh*3+\hhhhh,0) rectangle (\hhh*4+\hhhhh,\vvv*1);
\draw[-] (\hhh*4+\hhhhh,0) rectangle (\hhh*5+\hhhhh,\vvv*1);
\draw[-] (\hhh*5+\hhhhh,0) rectangle (\hhh*6+\hhhhh,\vvv*1);
\draw[-] (\hhh*6+\hhhhh,0) rectangle (\hhh*7+\hhhhh,\vvv*1);
\draw[-] (\hhh*2+\hhhhh,-\vvv*1) rectangle (\hhh*3+\hhhhh,\vvv*0);
\draw[-] (\hhh*3+\hhhhh,-\vvv*1) rectangle (\hhh*4+\hhhhh,\vvv*0);
\draw[-] (\hhh*4+\hhhhh,-\vvv*1) rectangle (\hhh*5+\hhhhh,\vvv*0);
\draw[-] (\hhh*5+\hhhhh,-\vvv*1) rectangle (\hhh*6+\hhhhh,\vvv*0);
\draw[-] (\hhh*3+\hhhhh,-\vvv*2) rectangle (\hhh*4+\hhhhh,-\vvv*1);
\draw[-] (\hhh*4+\hhhhh,-\vvv*2) rectangle (\hhh*5+\hhhhh,-\vvv*1);
\draw[-] (\hhh*5+\hhhhh,-\vvv*2) rectangle (\hhh*6+\hhhhh,-\vvv*1);
\draw[-] (\hhh*4+\hhhhh,-\vvv*3) rectangle (\hhh*5+\hhhhh,-\vvv*2);
\draw[-] (\hhh*5+\hhhhh,-\vvv*3) rectangle (\hhh*6+\hhhhh,-\vvv*2);
\node[] at (\hhh*3.5+\hhhhh,\vvv*0.5) {};
\node[] at (\hhh*4.5+\hhhhh,\vvv*0.5) {};
\node[] at (\hhh*5.5+\hhhhh,\vvv*0.5) {$1'$};
\node[] at (\hhh*6.5+\hhhhh,\vvv*0.5) {$1$};
\node[] at (\hhh*2.5+\hhhhh,-\vvv*0.5) {};
\node[] at (\hhh*3.5+\hhhhh,-\vvv*0.5) {};
\node[] at (\hhh*4.5+\hhhhh,-\vvv*0.5) {};
\node[] at (\hhh*5.5+\hhhhh,-\vvv*0.5) {$1'$};
\node[] at (\hhh*3.5+\hhhhh,-\vvv*1.5) {};
\node[] at (\hhh*4.5+\hhhhh,-\vvv*1.5) {$1'$};
\node[] at (\hhh*5.5+\hhhhh,-\vvv*1.5) {$1$};
\node[] at (\hhh*4.5+\hhhhh,-\vvv*2.5) {$2$};
\node[] at (\hhh*5.5+\hhhhh,-\vvv*2.5) {$2$};
\end{tikzpicture}.
\end{center}
Applying the modified shifted switching process to $(S^{(2)}, T^{(1)})$ in
$S^{(1)}S^{(2)}T^{(1)}T^{(2)}$ yields that
\vskip 2mm
\begin{center}
\begin{tikzpicture}
\def \hhh{5mm}    
\def \vvv{5mm}    
\def \aaa{1}      
\def \bbb{1}      
\draw[-,black!20] (\hhh*3,0) rectangle (\hhh*4,\vvv*1);
\draw[-,black!20] (\hhh*4,0) rectangle (\hhh*5,\vvv*1);
\draw[-] (\hhh*5,0) rectangle (\hhh*6,\vvv*1);
\draw[-] (\hhh*6,0) rectangle (\hhh*7,\vvv*1);
\draw[-,black!20] (\hhh*2,-\vvv*1) rectangle (\hhh*3,\vvv*0);
\draw[-,black!20] (\hhh*3,-\vvv*1) rectangle (\hhh*4,\vvv*0);
\draw[fill=black!30] (\hhh*4,-\vvv*1) rectangle (\hhh*5,\vvv*0);
\draw[-] (\hhh*5,-\vvv*1) rectangle (\hhh*6,\vvv*0);
\draw[-,black!20] (\hhh*4,-\vvv*3) rectangle (\hhh*5,-\vvv*2);
\draw[-,black!20] (\hhh*5,-\vvv*3) rectangle (\hhh*6,-\vvv*2);
\draw[fill=black!30] (\hhh*3,-\vvv*2) rectangle (\hhh*4,-\vvv*1);
\draw[-] (\hhh*4,-\vvv*2) rectangle (\hhh*5,-\vvv*1);
\draw[-] (\hhh*5,-\vvv*2) rectangle (\hhh*6,-\vvv*1);
\node[] at (\hhh*3.5,\vvv*0.5) {};
\node[] at (\hhh*4.5,\vvv*0.5) {};
\node[] at (\hhh*5.5,\vvv*0.5) {$1'$};
\node[] at (\hhh*6.5,\vvv*0.5) {$1$};
\node[] at (\hhh*2.5,-\vvv*0.5) {};
\node[] at (\hhh*3.5,-\vvv*0.5) {};
\node[] at (\hhh*4.5,-\vvv*0.5) {$2'$};
\node[] at (\hhh*5.5,-\vvv*0.5) {$1'$};
\node[] at (\hhh*3.5,-\vvv*1.5) {$2$};
\node[] at (\hhh*4.5,-\vvv*1.5) {$1'$};
\node[] at (\hhh*5.5,-\vvv*1.5) {$1$};
\node[] at (\hhh*4.5,-\vvv*2.5) {};
\node[] at (\hhh*5.5,-\vvv*2.5) {};
\draw[-,thick,red] (\hhh*\aaa+\hhh*2,-\vvv*\bbb) to (\hhh*\aaa+\hhh*4,-\vvv*\bbb)
to (\hhh*\aaa+\hhh*4,-\vvv*\bbb-\vvv)
to (\hhh*\aaa+\hhh*2,-\vvv*\bbb-\vvv)
to (\hhh*\aaa+\hhh*2,-\vvv*\bbb-\vvv*0);
\end{tikzpicture}
\hskip 2mm
\begin{tikzpicture}
\def \hhh{5mm}    
\def \vvv{5mm}    
\def \aaa{2}      
\def \bbb{1}      
\draw[|->] (\hhh*0.5,-\vvv*0.5) -- (\hhh*1.5,-\vvv*0.5) node[midway,above] {\tiny (S1$'$)};
\draw[-,black!20] (\hhh*3,0) rectangle (\hhh*4,\vvv*1);
\draw[-,black!20] (\hhh*4,0) rectangle (\hhh*5,\vvv*1);
\draw[-,black!20] (\hhh*4,-\vvv*3) rectangle (\hhh*5,-\vvv*2);
\draw[-,black!20] (\hhh*5,-\vvv*3) rectangle (\hhh*6,-\vvv*2);
\draw[-] (\hhh*5,0) rectangle (\hhh*6,\vvv*1);
\draw[-] (\hhh*6,0) rectangle (\hhh*7,\vvv*1);
\draw[-,black!20] (\hhh*2,-\vvv*1) rectangle (\hhh*3,\vvv*0);
\draw[-,black!20] (\hhh*3,-\vvv*1) rectangle (\hhh*4,\vvv*0);
\draw[fill=black!30] (\hhh*4,-\vvv*1) rectangle (\hhh*5,\vvv*0);
\draw[-] (\hhh*5,-\vvv*1) rectangle (\hhh*6,\vvv*0);
\draw[-] (\hhh*3,-\vvv*2) rectangle (\hhh*4,-\vvv*1);
\draw[fill=black!30] (\hhh*4,-\vvv*2) rectangle (\hhh*5,-\vvv*1);
\draw[-] (\hhh*5,-\vvv*2) rectangle (\hhh*6,-\vvv*1);
\node[] at (\hhh*3.5,\vvv*0.5) {};
\node[] at (\hhh*4.5,\vvv*0.5) {};
\node[] at (\hhh*5.5,\vvv*0.5) {$1'$};
\node[] at (\hhh*6.5,\vvv*0.5) {$1$};
\node[] at (\hhh*2.5,-\vvv*0.5) {};
\node[] at (\hhh*3.5,-\vvv*0.5) {};
\node[] at (\hhh*4.5,-\vvv*0.5) {$2'$};
\node[] at (\hhh*5.5,-\vvv*0.5) {$1'$};
\node[] at (\hhh*3.5,-\vvv*1.5) {$1$};
\node[] at (\hhh*4.5,-\vvv*1.5) {$2'$};
\node[] at (\hhh*5.5,-\vvv*1.5) {$1$};
\node[] at (\hhh*4.5,-\vvv*2.5) {};
\node[] at (\hhh*5.5,-\vvv*2.5) {};
\draw[-,thick,red] (\hhh*\aaa+\hhh*2,-\vvv*\bbb) to (\hhh*\aaa+\hhh*4,-\vvv*\bbb)
to (\hhh*\aaa+\hhh*4,-\vvv*\bbb-\vvv)
to (\hhh*\aaa+\hhh*2,-\vvv*\bbb-\vvv)
to (\hhh*\aaa+\hhh*2,-\vvv*\bbb-\vvv*0);
\end{tikzpicture}
\hskip 2mm
\begin{tikzpicture}
\def \hhh{5mm}    
\def \vvv{5mm}    
\def \aaa{2}      
\def \bbb{0}      
\draw[|->] (\hhh*0.5,-\vvv*0.5) -- (\hhh*1.5,-\vvv*0.5) node[midway,above] {\tiny (S1)};
\draw[-,black!20] (\hhh*3,0) rectangle (\hhh*4,\vvv*1);
\draw[-,black!20] (\hhh*4,0) rectangle (\hhh*5,\vvv*1);
\draw[-] (\hhh*5,0) rectangle (\hhh*6,\vvv*1);
\draw[-] (\hhh*6,0) rectangle (\hhh*7,\vvv*1);
\draw[-,black!20] (\hhh*2,-\vvv*1) rectangle (\hhh*3,\vvv*0);
\draw[-,black!20] (\hhh*3,-\vvv*1) rectangle (\hhh*4,\vvv*0);
\draw[fill=black!30] (\hhh*4,-\vvv*1) rectangle (\hhh*5,\vvv*0);
\draw[-] (\hhh*5,-\vvv*1) rectangle (\hhh*6,\vvv*0);
\draw[-,black!20] (\hhh*4,-\vvv*3) rectangle (\hhh*5,-\vvv*2);
\draw[-,black!20] (\hhh*5,-\vvv*3) rectangle (\hhh*6,-\vvv*2);
\draw[-] (\hhh*3,-\vvv*2) rectangle (\hhh*4,-\vvv*1);
\draw[-] (\hhh*4,-\vvv*2) rectangle (\hhh*5,-\vvv*1);
\draw[fill=black!30] (\hhh*5,-\vvv*2) rectangle (\hhh*6,-\vvv*1);
\node[] at (\hhh*3.5,\vvv*0.5) {};
\node[] at (\hhh*4.5,\vvv*0.5) {};
\node[] at (\hhh*5.5,\vvv*0.5) {$1'$};
\node[] at (\hhh*6.5,\vvv*0.5) {$1$};
\node[] at (\hhh*2.5,-\vvv*0.5) {};
\node[] at (\hhh*3.5,-\vvv*0.5) {};
\node[] at (\hhh*4.5,-\vvv*0.5) {$2'$};
\node[] at (\hhh*5.5,-\vvv*0.5) {$1'$};
\node[] at (\hhh*3.5,-\vvv*1.5) {$1$};
\node[] at (\hhh*4.5,-\vvv*1.5) {$1$};
\node[] at (\hhh*5.5,-\vvv*1.5) {$2'$};
\draw[-,thick,red] (\hhh*\aaa+\hhh*2,-\vvv*\bbb) to (\hhh*\aaa+\hhh*4,-\vvv*\bbb)
to (\hhh*\aaa+\hhh*4,-\vvv*\bbb-\vvv)
to (\hhh*\aaa+\hhh*3,-\vvv*\bbb-\vvv)
to (\hhh*\aaa+\hhh*3,-\vvv*\bbb-\vvv*2)
to (\hhh*\aaa+\hhh*2,-\vvv*\bbb-\vvv*2)
to (\hhh*\aaa+\hhh*2,-\vvv*\bbb-\vvv*0);
\end{tikzpicture}
\hskip 2mm
\begin{tikzpicture}
\def \hhh{5mm}    
\def \vvv{5mm}    
\draw[|->] (\hhh*0.5,-\vvv*0.5) -- (\hhh*1.5,-\vvv*0.5) node[midway,above] {\tiny (S5)};
\draw[-,black!20] (\hhh*3,0) rectangle (\hhh*4,\vvv*1);
\draw[-,black!20] (\hhh*4,0) rectangle (\hhh*5,\vvv*1);
\draw[-] (\hhh*5,0) rectangle (\hhh*6,\vvv*1);
\draw[-] (\hhh*6,0) rectangle (\hhh*7,\vvv*1);
\draw[-,black!20] (\hhh*2,-\vvv*1) rectangle (\hhh*3,\vvv*0);
\draw[-,black!20] (\hhh*3,-\vvv*1) rectangle (\hhh*4,\vvv*0);
\draw[-] (\hhh*4,-\vvv*1) rectangle (\hhh*5,\vvv*0);
\draw[fill=black!30] (\hhh*5,-\vvv*1) rectangle (\hhh*6,\vvv*0);
\draw[-,black!20] (\hhh*4,-\vvv*3) rectangle (\hhh*5,-\vvv*2);
\draw[-,black!20] (\hhh*5,-\vvv*3) rectangle (\hhh*6,-\vvv*2);
\draw[-] (\hhh*3,-\vvv*2) rectangle (\hhh*4,-\vvv*1);
\draw[-] (\hhh*4,-\vvv*2) rectangle (\hhh*5,-\vvv*1);
\draw[fill=black!30] (\hhh*5,-\vvv*2) rectangle (\hhh*6,-\vvv*1);
\node[] at (\hhh*3.5,\vvv*0.5) {};
\node[] at (\hhh*4.5,\vvv*0.5) {};
\node[] at (\hhh*5.5,\vvv*0.5) {$1'$};
\node[] at (\hhh*6.5,\vvv*0.5) {$1$};
\node[] at (\hhh*2.5,-\vvv*0.5) {};
\node[] at (\hhh*3.5,-\vvv*0.5) {};
\node[] at (\hhh*4.5,-\vvv*0.5) {$1'$};
\node[] at (\hhh*5.5,-\vvv*0.5) {$2'$};
\node[] at (\hhh*3.5,-\vvv*1.5) {$1$};
\node[] at (\hhh*4.5,-\vvv*1.5) {$1$};
\node[] at (\hhh*5.5,-\vvv*1.5) {$2'$};
\end{tikzpicture}.
\end{center}
Keeping the modified shifted switching process, we finally obtain
\vskip 2mm
\begin{center}
\begin{tikzpicture}
\def \hhh{5mm}    
\def \vvv{5mm}    
\def \hhhhh{60mm}  
\node[] at (\hhh*0.6,-\vvv*0.4) {$\widetilde{{}^ST}=$};
\draw[-] (\hhh*3+\hhhhh*0,0) rectangle (\hhh*4+\hhhhh*0,\vvv*1);
\draw[-] (\hhh*4+\hhhhh*0,0) rectangle (\hhh*5+\hhhhh*0,\vvv*1);
\draw[-] (\hhh*5+\hhhhh*0,0) rectangle (\hhh*6+\hhhhh*0,\vvv*1);
\draw[-] (\hhh*6+\hhhhh*0,0) rectangle (\hhh*7+\hhhhh*0,\vvv*1);
\draw[-] (\hhh*2+\hhhhh*0,-\vvv*1) rectangle (\hhh*3+\hhhhh*0,\vvv*0);
\draw[-] (\hhh*3+\hhhhh*0,-\vvv*1) rectangle (\hhh*4+\hhhhh*0,\vvv*0);
\draw[-] (\hhh*4+\hhhhh*0,-\vvv*1) rectangle (\hhh*5+\hhhhh*0,\vvv*0);
\draw[-] (\hhh*5+\hhhhh*0,-\vvv*1) rectangle (\hhh*6+\hhhhh*0,\vvv*0);
\draw[-] (\hhh*3+\hhhhh*0,-\vvv*2) rectangle (\hhh*4+\hhhhh*0,-\vvv*1);
\draw[-] (\hhh*4+\hhhhh*0,-\vvv*2) rectangle (\hhh*5+\hhhhh*0,-\vvv*1);
\draw[-] (\hhh*5+\hhhhh*0,-\vvv*2) rectangle (\hhh*6+\hhhhh*0,-\vvv*1);
\draw[-] (\hhh*4+\hhhhh*0,-\vvv*3) rectangle (\hhh*5+\hhhhh*0,-\vvv*2);
\draw[-] (\hhh*5+\hhhhh*0,-\vvv*3) rectangle (\hhh*6+\hhhhh*0,-\vvv*2);
\node[] at (\hhh*3.5+\hhhhh*0,\vvv*0.5) {$1'$};
\node[] at (\hhh*4.5+\hhhhh*0,\vvv*0.5) {$1$};
\node[] at (\hhh*5.5+\hhhhh*0,\vvv*0.5) {$1$};
\node[] at (\hhh*2.5+\hhhhh*0,-\vvv*0.5) {$1$};
\node[] at (\hhh*3.5+\hhhhh*0,-\vvv*0.5) {$1$};
\node[] at (\hhh*4.5+\hhhhh*0,-\vvv*0.5) {$2$};
\node[] at (\hhh*3.5+\hhhhh*0,-\vvv*1.5) {$2$};
\node[] at (\hhh*3.3+\hhhhh*0.5,-\vvv*0.5) {and};
\node[] at (\hhh*0.3+\hhhhh*1,-\vvv*0.4) {$\widetilde{S_T}=$};
\draw[-] (\hhh*3+\hhhhh*1,0) rectangle (\hhh*4+\hhhhh*1,\vvv*1);
\draw[-] (\hhh*4+\hhhhh*1,0) rectangle (\hhh*5+\hhhhh*1,\vvv*1);
\draw[-] (\hhh*5+\hhhhh*1,0) rectangle (\hhh*6+\hhhhh*1,\vvv*1);
\draw[fill=black!30] (\hhh*6+\hhhhh*1,0) rectangle (\hhh*7+\hhhhh*1,\vvv*1);
\draw[-] (\hhh*2+\hhhhh*1,-\vvv*1) rectangle (\hhh*3+\hhhhh*1,\vvv*0);
\draw[-] (\hhh*3+\hhhhh*1,-\vvv*1) rectangle (\hhh*4+\hhhhh*1,\vvv*0);
\draw[-] (\hhh*4+\hhhhh*1,-\vvv*1) rectangle (\hhh*5+\hhhhh*1,\vvv*0);
\draw[fill=black!30] (\hhh*6+\hhhhh*1,-\vvv*1) rectangle (\hhh*5+\hhhhh*1,-\vvv*0);
\draw[-] (\hhh*3+\hhhhh*1,-\vvv*2) rectangle (\hhh*4+\hhhhh*1,-\vvv*1);
\draw[fill=black!30] (\hhh*4+\hhhhh*1,-\vvv*2) rectangle (\hhh*5+\hhhhh*1,-\vvv*1);
\draw[fill=black!30] (\hhh*5+\hhhhh*1,-\vvv*2) rectangle (\hhh*6+\hhhhh*1,-\vvv*1);
\draw[fill=black!30] (\hhh*4+\hhhhh*1,-\vvv*3) rectangle (\hhh*5+\hhhhh*1,-\vvv*2);
\draw[fill=black!30] (\hhh*5+\hhhhh*1,-\vvv*3) rectangle (\hhh*6+\hhhhh*1,-\vvv*2);
\node[] at (\hhh*6.5+\hhhhh*1,\vvv*0.5) {$1$};
\node[] at (\hhh*5.5+\hhhhh*1,-\vvv*0.5) {$1'$};
\node[] at (\hhh*4.5+\hhhhh*1,-\vvv*1.5) {$1'$};
\node[] at (\hhh*5.5+\hhhhh*1,-\vvv*1.5) {$2'$};
\node[] at (\hhh*4.5+\hhhhh*1,-\vvv*2.5) {$1$};
\node[] at (\hhh*5.5+\hhhhh*1,-\vvv*2.5) {$2'$};
\end{tikzpicture}.
\end{center}
\vskip 2mm

\begin{rem}
In [18] Stembridge gave a combinatorial proof that Schur $Q$-functions are symmetric.
To do this he showed that there is a
content-reversing involution on the set of SYTs of a given shape,
say $\lambda$.
The set of SSYTs of shape $\lambda$,
however, is not closed under this involution,
thus Stembridge's method does not induce a combinatorial proof
that Schur $P$-functions are symmetric.

We remark that our shifted tableau switching (Theorem~\ref{Thm-Sec3-SW})
and modified shifted tableau switching (Theorm~\ref{thm-main-section6})
also give a combinatorial proof of the symmetry of Schur $Q$- and Schur $P$-functions, respectively.
To be precise we can define a content-reversing involution
by switching the border strip formed by $i,i'$ with the border strip formed by $i+1,(i+1)'$
and then performing the interchanges $i\leftrightarrow i+1, i' \leftrightarrow (i+1)'$.
It should be emphasized that the involution coming from our shifted tableau switching is different from Stembridge's involution.
Also we remark that Cho [3] gave another combinatorial proof of the symmetry of Schur P-functions
by giving a Bender-Knuth type involution on the set of semistandard decomposition tableaux.
\end{rem}

\vskip 2mm
The rest of this section is devoted to providing analogues of 
Theorem~\ref{thm-main-section4} and Example~\ref{example-section4}
for the modified shifted tableau switching.
For $\lambda, \mu, \nu \in \Lambda^+$,
the coefficients $g^\nu_{\lambda \mu}$ appearing in 
\begin{equation*}\label{expansion of skew Schur P function}
P_{\nu/\lambda}(x) = \underset{\mu}{\sum}g_{\lambda  \mu}^{\nu}P_{\mu}(x)
\qquad \text{or} \qquad
Q_{\ld}(x) Q_{\mu}(x) = \sum_{\nu} g_{\lambda  \mu}^{\nu} Q_\nu(x) \, 
\end{equation*}
are nonnegative integer coefficients 
since $g^\nu_{\lambda \mu} = 2^{\ell(\lambda)+\ell(\mu)-\ell(\nu)} f_{\lambda  \mu}^{\nu}$
and 
$\ell(\lambda)+\ell(\mu)-\ell(\nu) \geq 0$ whenever $f_{\lambda  \mu}^{\nu} \geq 0$.

To find an appropriate tableau model for these coefficients,
we note that, for each $i$,
the rightmost entry $i$ of $|\w(T)|$ in Theorem \ref{Stembridge's result}(b) is unprimed.
Indeed it was pointed out in \cite[page 126]{Ste1}
that the lattice property
is unaffected when the primes of these extremal entries are changed arbitrarily.
This shows that the number of SYTs of shape $\nu/\lambda$ and weight $\mu$
satisfying the lattice property is given by
$2^{\ell(\mu)} f^{\nu}_{\lambda \mu}.$
Since every entry on the main diagonal of $T$ is an extremal entry,
we see, consequently, that the number of SSYTs
of shape $\nu/\lambda$ and weight $\mu$ satisfying the lattice property
is given by
\begin{displaymath}
\frac{2^{\ell(\mu)} f^{\nu}_{\lambda \mu}}{2^{\ell(\nu) - \ell(\lambda)}}
\ \left(= g^{\nu}_{\lambda \mu}\right).
\end{displaymath}
Denote by $\mathcal{G}^\nu_{\lambda\mu}$
the set of SSYTs
of shape $\nu/\lambda$ and weight $\mu$
satisfying the lattice property.
We call an SSYT in $\mathcal{G}^\nu_{\lambda\mu}$
a {\it modified Littlewood-Richardson-Stembridge (LRS) tabelau}.
\begin{example} \label{exm-modi-LRS}
(a) Let $\ld = (4,2),$ $\mu =(4,3,1)$ and $\nu =(6,5,2,1)$.
Then
there are exactly eight modified {\rm LRS} tableaux of shape $\nu / \ld$ and weight $\mu$
as follows:
\vskip 2mm
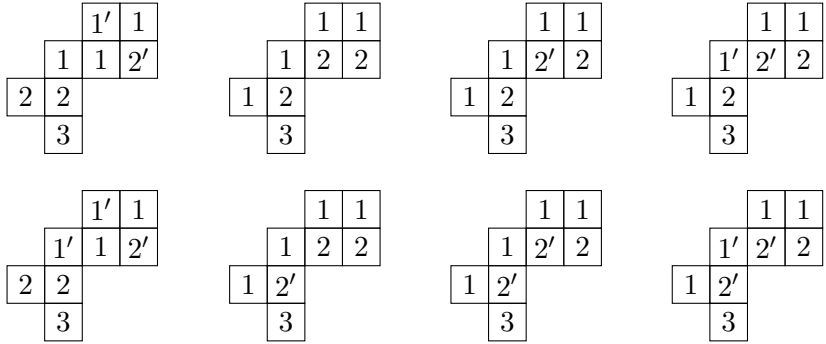
\begin{center}
\begin{tikzpicture}
\def\hhh{5mm}
\def\vvv{5mm}
\draw[-] (\hhh*5,\vvv*0) rectangle (\hhh*6,\vvv*1);
\draw[-] (\hhh*6,\vvv*0) rectangle (\hhh*7,\vvv*1);
\draw[-] (\hhh*4,-\vvv*1) rectangle (\hhh*5,\vvv*0);
\draw[-] (\hhh*5,-\vvv*1) rectangle (\hhh*6,\vvv*0);
\draw[-] (\hhh*6,-\vvv*1) rectangle (\hhh*7,\vvv*0);
\draw[-] (\hhh*3,-\vvv*2) rectangle (\hhh*4,-\vvv*1);
\draw[-] (\hhh*4,-\vvv*2) rectangle (\hhh*5,-\vvv*1);
\draw[-] (\hhh*4,-\vvv*3) rectangle (\hhh*5,-\vvv*2);
\node at (\hhh*5.5,\vvv*0.5) () {$1'$};
\node at (\hhh*6.5,\vvv*0.5) () {$1$};
\node at (\hhh*4.5,-\vvv*0.5) () {$1$};
\node at (\hhh*5.5,-\vvv*0.5) () {$1$};
\node at (\hhh*6.5,-\vvv*0.5) () {$2'$};
\node at (\hhh*3.5,-\vvv*1.5) () {$2$};
\node at (\hhh*4.5,-\vvv*1.5) () {$2$};
\node at (\hhh*4.5,-\vvv*2.5) () {$3$};
\end{tikzpicture}
\hskip 8mm
\begin{tikzpicture}
\def\hhh{5mm}
\def\vvv{5mm}
\draw[-] (\hhh*5,\vvv*0) rectangle (\hhh*6,\vvv*1);
\draw[-] (\hhh*6,\vvv*0) rectangle (\hhh*7,\vvv*1);
\draw[-] (\hhh*4,-\vvv*1) rectangle (\hhh*5,\vvv*0);
\draw[-] (\hhh*5,-\vvv*1) rectangle (\hhh*6,\vvv*0);
\draw[-] (\hhh*6,-\vvv*1) rectangle (\hhh*7,\vvv*0);
\draw[-] (\hhh*3,-\vvv*2) rectangle (\hhh*4,-\vvv*1);
\draw[-] (\hhh*4,-\vvv*2) rectangle (\hhh*5,-\vvv*1);
\draw[-] (\hhh*4,-\vvv*3) rectangle (\hhh*5,-\vvv*2);
\node at (\hhh*5.5,\vvv*0.5) () {$1$};
\node at (\hhh*6.5,\vvv*0.5) () {$1$};
\node at (\hhh*4.5,-\vvv*0.5) () {$1$};
\node at (\hhh*5.5,-\vvv*0.5) () {$2$};
\node at (\hhh*6.5,-\vvv*0.5) () {$2$};
\node at (\hhh*3.5,-\vvv*1.5) () {$1$};
\node at (\hhh*4.5,-\vvv*1.5) () {$2$};
\node at (\hhh*4.5,-\vvv*2.5) () {$3$};
\end{tikzpicture}
\hskip 8mm
\begin{tikzpicture}
\def\hhh{5mm}
\def\vvv{5mm}
\draw[-] (\hhh*5,\vvv*0) rectangle (\hhh*6,\vvv*1);
\draw[-] (\hhh*6,\vvv*0) rectangle (\hhh*7,\vvv*1);
\draw[-] (\hhh*4,-\vvv*1) rectangle (\hhh*5,\vvv*0);
\draw[-] (\hhh*5,-\vvv*1) rectangle (\hhh*6,\vvv*0);
\draw[-] (\hhh*6,-\vvv*1) rectangle (\hhh*7,\vvv*0);
\draw[-] (\hhh*3,-\vvv*2) rectangle (\hhh*4,-\vvv*1);
\draw[-] (\hhh*4,-\vvv*2) rectangle (\hhh*5,-\vvv*1);
\draw[-] (\hhh*4,-\vvv*3) rectangle (\hhh*5,-\vvv*2);
\node at (\hhh*5.5,\vvv*0.5) () {$1$};
\node at (\hhh*6.5,\vvv*0.5) () {$1$};
\node at (\hhh*4.5,-\vvv*0.5) () {$1$};
\node at (\hhh*5.5,-\vvv*0.5) () {$2'$};
\node at (\hhh*6.5,-\vvv*0.5) () {$2$};
\node at (\hhh*3.5,-\vvv*1.5) () {$1$};
\node at (\hhh*4.5,-\vvv*1.5) () {$2$};
\node at (\hhh*4.5,-\vvv*2.5) () {$3$};
\end{tikzpicture}
\hskip 8mm
\begin{tikzpicture}
\def\hhh{5mm}
\def\vvv{5mm}
\draw[-] (\hhh*5,\vvv*0) rectangle (\hhh*6,\vvv*1);
\draw[-] (\hhh*6,\vvv*0) rectangle (\hhh*7,\vvv*1);
\draw[-] (\hhh*4,-\vvv*1) rectangle (\hhh*5,\vvv*0);
\draw[-] (\hhh*5,-\vvv*1) rectangle (\hhh*6,\vvv*0);
\draw[-] (\hhh*6,-\vvv*1) rectangle (\hhh*7,\vvv*0);
\draw[-] (\hhh*3,-\vvv*2) rectangle (\hhh*4,-\vvv*1);
\draw[-] (\hhh*4,-\vvv*2) rectangle (\hhh*5,-\vvv*1);
\draw[-] (\hhh*4,-\vvv*3) rectangle (\hhh*5,-\vvv*2);
\node at (\hhh*5.5,\vvv*0.5) () {$1$};
\node at (\hhh*6.5,\vvv*0.5) () {$1$};
\node at (\hhh*4.5,-\vvv*0.5) () {$1'$};
\node at (\hhh*5.5,-\vvv*0.5) () {$2'$};
\node at (\hhh*6.5,-\vvv*0.5) () {$2$};
\node at (\hhh*3.5,-\vvv*1.5) () {$1$};
\node at (\hhh*4.5,-\vvv*1.5) () {$2$};
\node at (\hhh*4.5,-\vvv*2.5) () {$3$};
\end{tikzpicture}
\end{center}
\vskip 3mm
\begin{center}
\begin{tikzpicture}
\def\hhh{5mm}
\def\vvv{5mm}
\draw[-] (\hhh*5,\vvv*0) rectangle (\hhh*6,\vvv*1);
\draw[-] (\hhh*6,\vvv*0) rectangle (\hhh*7,\vvv*1);
\draw[-] (\hhh*4,-\vvv*1) rectangle (\hhh*5,\vvv*0);
\draw[-] (\hhh*5,-\vvv*1) rectangle (\hhh*6,\vvv*0);
\draw[-] (\hhh*6,-\vvv*1) rectangle (\hhh*7,\vvv*0);
\draw[-] (\hhh*3,-\vvv*2) rectangle (\hhh*4,-\vvv*1);
\draw[-] (\hhh*4,-\vvv*2) rectangle (\hhh*5,-\vvv*1);
\draw[-] (\hhh*4,-\vvv*3) rectangle (\hhh*5,-\vvv*2);
\node at (\hhh*5.5,\vvv*0.5) () {$1'$};
\node at (\hhh*6.5,\vvv*0.5) () {$1$};
\node at (\hhh*4.5,-\vvv*0.5) () {$1'$};
\node at (\hhh*5.5,-\vvv*0.5) () {$1$};
\node at (\hhh*6.5,-\vvv*0.5) () {$2'$};
\node at (\hhh*3.5,-\vvv*1.5) () {$2$};
\node at (\hhh*4.5,-\vvv*1.5) () {$2$};
\node at (\hhh*4.5,-\vvv*2.5) () {$3$};
\end{tikzpicture}
\hskip 8mm
\begin{tikzpicture}
\def\hhh{5mm}
\def\vvv{5mm}
\draw[-] (\hhh*5,\vvv*0) rectangle (\hhh*6,\vvv*1);
\draw[-] (\hhh*6,\vvv*0) rectangle (\hhh*7,\vvv*1);
\draw[-] (\hhh*4,-\vvv*1) rectangle (\hhh*5,\vvv*0);
\draw[-] (\hhh*5,-\vvv*1) rectangle (\hhh*6,\vvv*0);
\draw[-] (\hhh*6,-\vvv*1) rectangle (\hhh*7,\vvv*0);
\draw[-] (\hhh*3,-\vvv*2) rectangle (\hhh*4,-\vvv*1);
\draw[-] (\hhh*4,-\vvv*2) rectangle (\hhh*5,-\vvv*1);
\draw[-] (\hhh*4,-\vvv*3) rectangle (\hhh*5,-\vvv*2);
\node at (\hhh*5.5,\vvv*0.5) () {$1$};
\node at (\hhh*6.5,\vvv*0.5) () {$1$};
\node at (\hhh*4.5,-\vvv*0.5) () {$1$};
\node at (\hhh*5.5,-\vvv*0.5) () {$2$};
\node at (\hhh*6.5,-\vvv*0.5) () {$2$};
\node at (\hhh*3.5,-\vvv*1.5) () {$1$};
\node at (\hhh*4.5,-\vvv*1.5) () {$2'$};
\node at (\hhh*4.5,-\vvv*2.5) () {$3$};
\end{tikzpicture}
\hskip 8mm
\begin{tikzpicture}
\def\hhh{5mm}
\def\vvv{5mm}
\draw[-] (\hhh*5,\vvv*0) rectangle (\hhh*6,\vvv*1);
\draw[-] (\hhh*6,\vvv*0) rectangle (\hhh*7,\vvv*1);
\draw[-] (\hhh*4,-\vvv*1) rectangle (\hhh*5,\vvv*0);
\draw[-] (\hhh*5,-\vvv*1) rectangle (\hhh*6,\vvv*0);
\draw[-] (\hhh*6,-\vvv*1) rectangle (\hhh*7,\vvv*0);
\draw[-] (\hhh*3,-\vvv*2) rectangle (\hhh*4,-\vvv*1);
\draw[-] (\hhh*4,-\vvv*2) rectangle (\hhh*5,-\vvv*1);
\draw[-] (\hhh*4,-\vvv*3) rectangle (\hhh*5,-\vvv*2);
\node at (\hhh*5.5,\vvv*0.5) () {$1$};
\node at (\hhh*6.5,\vvv*0.5) () {$1$};
\node at (\hhh*4.5,-\vvv*0.5) () {$1$};
\node at (\hhh*5.5,-\vvv*0.5) () {$2'$};
\node at (\hhh*6.5,-\vvv*0.5) () {$2$};
\node at (\hhh*3.5,-\vvv*1.5) () {$1$};
\node at (\hhh*4.5,-\vvv*1.5) () {$2'$};
\node at (\hhh*4.5,-\vvv*2.5) () {$3$};
\end{tikzpicture}
\hskip 8mm
\begin{tikzpicture}
\def\hhh{5mm}
\def\vvv{5mm}
\draw[-] (\hhh*5,\vvv*0) rectangle (\hhh*6,\vvv*1);
\draw[-] (\hhh*6,\vvv*0) rectangle (\hhh*7,\vvv*1);
\draw[-] (\hhh*4,-\vvv*1) rectangle (\hhh*5,\vvv*0);
\draw[-] (\hhh*5,-\vvv*1) rectangle (\hhh*6,\vvv*0);
\draw[-] (\hhh*6,-\vvv*1) rectangle (\hhh*7,\vvv*0);
\draw[-] (\hhh*3,-\vvv*2) rectangle (\hhh*4,-\vvv*1);
\draw[-] (\hhh*4,-\vvv*2) rectangle (\hhh*5,-\vvv*1);
\draw[-] (\hhh*4,-\vvv*3) rectangle (\hhh*5,-\vvv*2);
\node at (\hhh*5.5,\vvv*0.5) () {$1$};
\node at (\hhh*6.5,\vvv*0.5) () {$1$};
\node at (\hhh*4.5,-\vvv*0.5) () {$1'$};
\node at (\hhh*5.5,-\vvv*0.5) () {$2'$};
\node at (\hhh*6.5,-\vvv*0.5) () {$2$};
\node at (\hhh*3.5,-\vvv*1.5) () {$1$};
\node at (\hhh*4.5,-\vvv*1.5) () {$2'$};
\node at (\hhh*4.5,-\vvv*2.5) () {$3$};
\end{tikzpicture}
\end{center}
Therefore $g^\nu_{\lambda \mu} = 8$.
\vskip 2mm
\noindent
(b) For $\nu \in \Lambda^+$,
$g^{\nu}_{\emptyset \, \nu}=1$
and the corresponding modified LRS tableau is $R_\nu$.
\end{example}


\begin{thm}\label{thm-main-section7}
Suppose that $S$ and $T$ are SSYTs such that $T$ extends $S$.
If the modified shifted tableau switching transforms $(S,T)$ into $(\widetilde{{}^ST}, \widetilde{S_T})$,
then we have the following.
\begin{itemize}
\item[(a)] The modified shifted tableau switching transforms
 $(\widetilde{{}^ST}, \widetilde{S_T})$ into $(S,T)$.

\item[(b)] When $T$ is a modified LRS tableau, so is $\widetilde{{}^ST}$.

\item[(c)] When $S$ is a modified LRS tableau, so is $\widetilde{S_T}$.
\end{itemize}
\end{thm}
\begin{proof}
(a)
It is enough to show that
$\widetilde{\Sigma} \circ \widetilde{\Sigma}(A \cup B) = A \cup B$
for any shifted perforated $(\bfa,\bfb)$-pair $(A,B)$ such that
$B$ extends $A$. 
If a modified switch is not used in the modified shifted switching process,
the assertion follows from Lemma \ref{lem-sw-involution}.
Otherwise,
we have
\begin{align*}
\widetilde{\Sigma}(\widetilde{{}^AB},\widetilde{A_B})
&= \widetilde{\Sigma}(\widetilde{\Sigma}(A,B))  &   \\
&=  \Sigma\left(\Omega\left( \Sigma \left(\Omega (A,B)\right)\right)\right)
  & \text{(by Lemma \ref{lem-hatsigma-sigma}(a))}\\
&=  \Sigma\left(\Omega\left(\Omega \left(\Sigma (A,B)\right)\right)\right)
   & \text{(by Lemma \ref{lem-hatsigma-sigma}(b))}\\
&=  \Sigma\left(\Sigma (A,B)\right)
 &  \text{(since $\Omega \circ \Omega = {\rm id}$)}\\
&= (A,B)  &  \text{(by Lemma \ref{lem-sw-involution})} \,
\end{align*}
as desired.

(b)
Suppose that $(S,T)$ is a pair of SSYTs such that $T$ extends $S$.
Denote by $(S',T')$ any pair obtained by applying
a sequence of shifted and modified switches to $(S,T)$.
To prove our assertion,
it is suffices to show that
$\w(T)$ is a lattice word if and only if $\w(T')$ is a lattice word. 
Notice that every word shifted Knuth equivalent to an LRS word is an LRS word 
by Lemma \ref{lem-BKT}
and every word obtained by changing the primes of any rightmost letters
is also a lattice word by \cite[page 126]{Ste1},
and hence one can immediately derive that 
every word shifted Knuth equivalent to a lattice word is a lattice word.
Moreover, by Lemma \ref{lem-hatsigma-sigma}(a),
$\w(T')$ is obtained from $\w(T)$ by applying a sequence of shifted Knuth transformations and
by priming some rightmost letters if necessary.
This shows that
$\w(T)$ is a lattice word if and only if $\w(T')$ is a lattice word.

(c) Let $S$ be a modified LRS tableau.
From (b) it follows that $\widetilde{S_T}$ is a modified LRS tableau if and only if $\widetilde{{}^{\widetilde{{}^ST}}\widetilde{S_T}}$
is a modified LRS tableau.
Now the assertion is straightforward by (a).
\end{proof}
%

\begin{example}
\label{example-section7}
Let $\lambda, \mu, \nu \in \Lambda^+$ with $\lambda,\mu \subseteq \nu$.

(a) (The modified shifted Littlewood-Richardson rule)
By Theorem \ref{thm-main-section7} we obtain a bijection
\begin{displaymath}
\Sigma: \mathcal{G}^\ld_{\emptyset \ld} \times {\mathcal{Y}}(\nu/\lambda)
\rightarrow
\bigcup_{\mu \subseteq \nu} {\mathcal{Y}}(\mu) \times
\mathcal{G}^\nu_{\mu \lambda},
\quad
(R_\lambda, T) \mapsto \left(\widetilde{{}^{R_\lambda}T}, \widetilde{({R_\lambda})_T}\right).
\end{displaymath}
This immediately gives a combinatorial interpretation of
$P_{\nu/\lambda}(x) = \sum_{\mu} g^{\nu}_{\mu \lambda} P_\mu(x)$.

(b) (Symmetry of the modified shifted Littlewood-Richardson coefficients)
By Theorem \ref{thm-main-section7} we obtain a bijection
\begin{displaymath}
\Sigma :
\mathcal{G}^\ld_{\emptyset \ld} \times \mathcal{G}^\nu_{\lambda \mu}
\rightarrow
\mathcal{G}^\mu_{\emptyset \mu} \times \mathcal{G}^\nu_{\mu \lambda},
\quad
(R_\lambda,T) \mapsto \left(R_\mu,\widetilde{({R_\lambda})_T}\right).
\end{displaymath}
Hence, not only have we proven $g^{\nu}_{\ld \mu} = g^{\nu}_{\mu \ld},$
we have constructed an explicit involution that interchanges the inner shape with
the weight of a modified LRS tableau.
\end{example}

\end{document}